\documentclass{article}
\usepackage{geometry}
\usepackage[utf8]{inputenc}
\usepackage{amsmath}
\usepackage{amsfonts}
\usepackage{amssymb}
\usepackage{amsthm}
\usepackage{graphicx}
\usepackage{caption}
\usepackage{subcaption}
\usepackage{float}
\usepackage{diagbox}
\usepackage{cite}
\usepackage{algorithm}

\usepackage{pgf,tikz,pgfplots}
\usepgfplotslibrary{groupplots}
\pgfplotsset{compat=1.15}
\usepackage{mathrsfs}
\usetikzlibrary{arrows}

\usepackage[pdfencoding=auto]{hyperref}
\hypersetup{
    colorlinks,
    citecolor=green,
    filecolor=black,
    linkcolor=blue,
    urlcolor=black
}
\newtheorem{Proposition}{Proposition}
\newtheorem{Lemma}{Lemma}
\newtheorem{Remark}{Remark}
\newtheorem{Assumption}{Assumption}

\newcommand{\dt}{\,\partial_t\, }
\newcommand{\dx}{\,\partial_x\, }
\newcommand{\dv}{\,\partial_v\, }

\newcommand{\dxx}{\,\partial_{xx}\, }
\newcommand{\dtx}{\,\partial_{tx}\, }

\newcommand{\dvv}{\,\partial_{vv}\, }
\newcommand{\dxxx}{\,\partial_{xxx}\, }
\newcommand{\dvvv}{\,\partial_{vvv}\, }
\newcommand{\dxxxx}{\,\partial_{xxxx}\, }

\newcommand{\dd}{\,\mathrm{d}}

\newcommand{\R}{\mathbb{R}}

\newcommand{\ipd}{{i+\frac{1}{2}}}
\newcommand{\imd}{{i-\frac{1}{2}}}
\newcommand{\jpd}{{j+\frac{1}{2}}}
\newcommand{\jmd}{{j-\frac{1}{2}}}

\newcommand{\M}{\mathcal{M}}
\newcommand{\T}{\mathcal{T}}

\newcommand{\Oeps}[1]{\mathcal{O}(\varepsilon^{#1})}
\newcommand*\circled[1]{\tikz[baseline=(char.base)]{\node[shape=circle,draw,inner sep=2pt] (char) {#1};}}
\numberwithin{equation}{section}

\title{Hybrid Kinetic/Fluid numerical method for the Vlasov-BGK equation in the diffusive scaling}
\author{Tino Laidin \footnote{Univ. Lille, CNRS, Inria, UMR 8524 - Laboratoire Paul Painlevé, F-59000 Lille, France (\texttt{tino.laidin@univ-lille.fr})}}

\begin{document}

\maketitle

\begin{abstract}
    This paper presents a hybrid numerical method for linear collisional kinetic equations with diffusive scaling. The aim of the method is to reduce the computational cost of kinetic equations by taking advantage of the lower dimensionality of the asymptotic fluid model while reducing the error induced by the latter approach. It relies on two criteria motivated by a pertubative approach to obtain a dynamic domain decomposition. The first criterion quantifies how far from a local equilibrium in velocity the distribution function of particles is. The second one depends only on the macroscopic quantities that are available on the whole computing domain. Interface conditions are dealt with using a micro-macro decomposition and the method is significantly more efficient than a standard full kinetic approach. Some properties of the hybrid method are also investigated, such as the conservation of mass.\\[1em]
    \textsc{Mathematics Subjects Classification:} (Primary) 65M08, 82M12 (Secondary) 35B40, 65M55
    
\end{abstract}
\section{Introduction}\label{sec:Intro}
Modelling semiconductors has become a major issue in the last decades in many fields of science such as physics, mathematics, and engineering. They have a wide scope of applications going from computer hardware (CPU, GPU, etc...) to smart devices, medical equipment and even solar panels. Therefore, it seems clear that understanding the mathematical models describing those uses is crucial along with enhancing the computing and energetic performances of the simulations.

\noindent Semiconductor models can be classified in three major scales: particles (microscopic), kinetic (mesoscopic), and fluid (macroscopic). The first type of model is about describing the system as point particles interacting with each other via collisions or electromagnetic forces. Such a system is in practice extremely large and its study both theoretically and numerically becomes unattainable. The solution is instead of considering each individual particle, to describe them with a probability distribution $f^\varepsilon = f^\varepsilon(t,x,v)$ which depends on the time $t\geq0$, the space variable $x$ and the velocity variable $v$. This is the kinetic scale and it is computationally more accessible while still keeping a sense of the individuality of the particles. At this scale, the semiconductor is typically described by a Vlasov-type equation modelling the long-range interactions to which we add a collision operator $\mathcal{Q}(f)$ to take into account the short-range ones. 

\noindent Let $d\geq 1$ be an integer. We denote by $\Omega_x$ and $\Omega_v$ two open and bounded subsets of $\R^d$. In this article, we are interested in the following scaled equation to find a particle distribution $f$ depending on $t\geq0$, $x\in\Omega_x$ and $v\in\Omega_v$, solution to:
\begin{equation*}\label{prob:Peps}\tag{$P^\varepsilon$}
    \left\lbrace\begin{aligned}
        &\frac{\partial}{\partial t} f^\varepsilon + \frac{v}{\varepsilon^{\alpha}}\,\cdot\,\nabla_x f^\varepsilon + \frac{E}{\varepsilon^{\alpha}}\,\cdot\,\nabla_v f^\varepsilon=\frac{1}{\varepsilon^{\alpha+1}}\mathcal{Q}(f^\varepsilon),\\
        &f^\varepsilon(0,x,v)=f_0(x,v),
    \end{aligned}\right.
\end{equation*}
Here, $E$ is a given exterior electrical field depending only on the space variable $x$, and $\varepsilon$ is the scaling parameter. We also assume that the initial condition $f_0$ is nonnegative and does not depend on $\varepsilon$. Throughout this work, the collision operator will be the linearized Bhatnagar–Gross–Krook (BGK) operator \cite{BhatnagarGrossKrook1954}: 
\begin{equation}\label{Q_BGK}
    \mathcal{Q}(f)=\rho\M - m_0f ,\quad \mbox{ where } \rho=\langle f\rangle \mbox{ and }\langle f\rangle=\int_{\Omega_v} f \,\dd v.
\end{equation}
Here $\M(v)$ denotes a given nonnegative even function of $v$, the so-called Maxwellian. We assume that it admits at least finite zeroth, second, and fourth moments $m_0, m_2, m_4$ in velocity where $m_k$ given by $$m_k =\int_{\Omega_v} |v|^k\M(v)\,\dd v.$$ A standard function satisfying these assumptions is the centered Gaussian:
$$\mathcal{M}(v)=\frac{e^{-|v|^2/2}}{(2\pi)^{d/2}}.$$ In a more general setting, the electric field $E$ is the gradient of a potential $V\in\mathcal{C}^2(\Omega_x)$, $E=-\nabla V$, and $\eqref{prob:Peps}$ admits an equilibrium given by 
\begin{equation}\label{GlobalEq}
    F(x,v) = \frac{M_0}{\mu_0}e^{-(V(x)+\frac{|v|^2}{2})},\quad (x,v)\in\Omega_x\times\Omega_v
\end{equation}
where $\mu_0=\int_{\Omega_x\times\Omega_v} e^{-(V(x)+\frac{|v|^2}{2})}\,\dd x\,\dd v$ and $M_0=\int_{\Omega_x\times\Omega_v} f_0(x,v)\,\dd x\,\dd v$ is the mass of the initial condition. In particular, $F$ can be written under a separate variable form:
\begin{equation}
    F(x,v) = M_0\phi(x)\mathcal{M}(v),\quad \text{where }\phi=\frac{e^{-V(x)}}{\int_{\Omega_x} e^{-V(x)}\,\dd x}.
\end{equation}
The functions $\M$ and $\phi$ are called local equilibria in velocity and space respectively.
\noindent Equation \eqref{prob:Peps} is a scaled equation. The parameter $\varepsilon$ is often called the Knudsen number. It is the ratio between the mean free path of the particles and the length scale of observation. By considering different values of $\alpha$ in \eqref{prob:Peps} one can obtain different asymptotic descriptions of the model as $\varepsilon$ tends to $0$. The choice $\alpha=0$ corresponds to the hydrodynamic scaling. Such a model can be shown to reach a fluid limit given by the Euler or Navier-Stokes equations  \cite{SaintRaymond2009, jungel_transport_2009}. This work will focus on $\alpha=1$: the diffusive scaling. Such a scaling and its asymptotic limit have first been studied in \cite{BensoussanLionsPapanicolaou1979}. The asymptotic expansion of the distribution function $f^\varepsilon$ in $\varepsilon$ is justified in \cite{BardosSantosSentis1984} for the neutron transport and in \cite{Poupaud1991} for the linear Boltzmann equation. In \cite{GoudonPoupaudDegond2000} a large class of linear collision operators is dealt with and in \cite{GoudonPoupaud2001}, the authors justified an approximation of the kinetic equation by diffusion using homogenization.

\noindent In practice, the Knudsen number can be of order 1 down to 0 depending on the physics we are modelling. On the one hand, when $\varepsilon\sim 1$, the system is said to be in the kinetic regime. It models a system with few collisions between particles. On the other hand, when $\varepsilon\ll 1$, the system reaches the fluid regime. This limit case is described by a drift-diffusion model:
\begin{equation*}\label{prob:P}\tag{$P$}
    \left\lbrace\begin{aligned}
        &\dt \rho -\mathrm{div}_x\left(\,\nabla_x\rho - E\rho\right) = 0,\\
        &\rho(0,x)=\rho_0(x).
    \end{aligned}\right.
\end{equation*}

\noindent While the kinetic model is precise in its description of the system, it remains expensive to compute numerically. Indeed, if we consider the full $3$D case, the phase space is 7 dimensional and the collision operator can induce important nonlinearities. On the other side, the fluid model simplifies greatly the description of the system and is much less expensive in computational resources as the solution does not depend on the velocity variable anymore. The use of the later becomes therefore attractive. Nevertheless, one must take into account that this description only applies for small values of the Knudsen number. 

\noindent A strategy to take advantage of both scales of description and reduce the computational time is to use a hybrid method. In the diffusive setting, various methods were developed. In \cite{CrestettoCrouseillesDimarcoLemou2019} the authors considered a coupling between a Monte-Carlo approximation for the microscopic part of the equation and a finite volume method for the macroscopic one. An automatic domain decomposition method with a buffer zone is used in \cite{DimarcoMieussensRispoli2014}. In \cite{EinkemmerHuWang2021}, a dynamic low-rank method is applied to a linear transport equation based on a micro-macro approach.

\noindent In this paper, we shall adapt the method developed in \cite{FilbetRey2015} for a hydrodynamic scaling to the diffusive setting. The strategy is to consider the Chapman-Enskog expansion of $f^\varepsilon$ to derive criteria that allow to determine the best regime to use in a given subdomain. This defines a hybrid kinetic/fluid solver with an automatic domain decomposition. Furthermore, interface conditions are dealt with using a micro-macro decomposition of the distribution. The coupling method can be adapted to different solvers and the only addition is the implementation of the subdomains indicators.

\noindent On the discrete level, the scaling parameter $\varepsilon$ can result in longer computation time. Indeed as $\varepsilon$ tends to $0$, the transport velocity in \eqref{prob:Peps} formally goes to infinity. Numerically, it translates to smaller and smaller time steps to guarantee the stability of the scheme. A solution to this problem is to use schemes that remain stable in the diffusive limit $\varepsilon\rightarrow0$. These schemes fall into the framework of Asymptotic Preserving (AP) schemes, a notion introduced in \cite{Klar1999} and \cite{Jin1999}. We also refer to the recent review article \cite{Jin2021ReviewAP}. This AP property can be summarized by the diagram in Figure \ref{fig:APdiag}. In the diagram, $\rho$ corresponds to a solution to the problem \eqref{prob:P} and $\rho_h$ is an approximation of $\rho$. On the other hand, $f^\varepsilon$ is a solution to the problem \eqref{prob:Peps} and $f^\varepsilon_h$ is an approximation of $f^\varepsilon$. The idea behind AP scheme is threefold. Firstly, the scheme for \eqref{prob:Peps} has to be a consistent discretization of the limit model as $\varepsilon\rightarrow0$. Secondly, a scheme is considered truly AP only if the stability criterion on the time step is independent on the parameter $\varepsilon$. Thirdly, one can explicitly take $\varepsilon=0$ in the scheme. The need of an AP scheme in the kinetic domain of our hybrid scheme is crucial. The limit scheme is used in the fluid regions of the domain decomposition and it aims to ensure good transitions between kinetic and fluid states. While AP schemes are designed to resolve both the mesoscopic and the macroscopic scales automatically, it often implies more expensive computation even in a fluid regime. By using a hybrid method, one can effectively take advantage of the properties of an AP scheme while limiting its use and therefore reduce the computation time. This method falls into the framework of Asymptotically Complexity Diminishing Schemes (ACDS) \cite{DegondDimarco2012, CrestettoCrouseillesDimarcoLemou2019}. Indeed, as the Knudsen number tends to zero, the hybrid method is designed to use the less complex fluid model more often.

\paragraph{Plan of the paper.} The outline is as follows. Section \ref{sec:CE} is dedicated to the derivation of a hierarchy of macroscopic models based on the Chapman-Enskog expansion of the distribution. In Section \ref{sec:MM} we present a micro-macro reformulation of the Vlasov-BGK equation. This reformulation is then used to develop an Asymptotic Preserving method with a finite volume approach. Section \ref{sec:DynCoup} is dedicated to the the hybrid method. The coupling indicators based on the hierarchy introduced in Section \ref{sec:CE} are presented and the implementation of the hybrid scheme is discussed. Finally, numerical experiments are performed in Section \ref{sec:NumSim}.

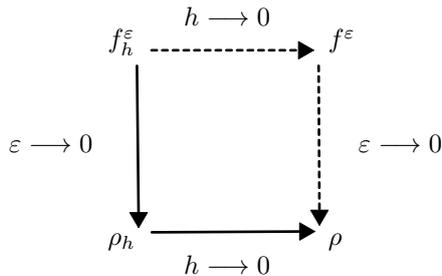
\begin{figure}
    \centering
    \begin{tikzpicture}[line cap=round,line join=round,>=triangle 45,x=1.0cm,y=1.0cm]
\clip(-3.0,-4) rectangle (3.0,0.0);
\draw [line width=1.pt,dash pattern=on 2pt off 2pt] (-1.,-1.)-- (1.,-1.);
\draw [line width=1.pt,dash pattern=on 2pt off 2pt] (1.198729454088928,-1.2068637637137198)-- (1.2095466323889421,-3.1864073926163154);
\draw [line width=1.pt] (1.0148374229886865,-3.402750958616599)-- (-0.9971577408139543,-3.4135681369166133);
\draw [line width=1.pt] (-1.1918669502142099,-1.2068637637137198)-- (-1.1702325936141813,-3.218858927516358);
\draw (1.6,-2.0) node[anchor=north west] {$\varepsilon\longrightarrow0$};
\draw (-3.0,-2.0) node[anchor=north west] {$\varepsilon\longrightarrow0$};
\draw (-0.7,-0.3) node[anchor=north west] {$h\longrightarrow0$};
\draw (-0.7,-3.6) node[anchor=north west] {$h\longrightarrow0$};
\draw (1.2,-0.6) node[anchor=north west] {$f^\varepsilon$};
\draw (-1.7,-0.6) node[anchor=north west] {$f^\varepsilon_h$};
\draw (-1.7,-3.3) node[anchor=north west] {$\rho_h$};
\draw (1.2,-3.3) node[anchor=north west] {$\rho$};
\begin{scriptsize}
\draw [fill=black,shift={(1.,-1.)},rotate=270] (0,0) ++(0 pt,3.75pt) -- ++(3.2475952641916446pt,-5.625pt)--++(-6.495190528383289pt,0 pt) -- ++(3.2475952641916446pt,5.625pt);
\draw [fill=black,shift={(1.2095466323889421,-3.1864073926163154)},rotate=180] (0,0) ++(0 pt,3.75pt) -- ++(3.2475952641916446pt,-5.625pt)--++(-6.495190528383289pt,0 pt) -- ++(3.2475952641916446pt,5.625pt);
\draw [fill=black,shift={(1.0148374229886865,-3.402750958616599)},rotate=270] (0,0) ++(0 pt,3.75pt) -- ++(3.2475952641916446pt,-5.625pt)--++(-6.495190528383289pt,0 pt) -- ++(3.2475952641916446pt,5.625pt);
\draw [fill=black,shift={(-1.1702325936141813,-3.218858927516358)},rotate=180] (0,0) ++(0 pt,3.75pt) -- ++(3.2475952641916446pt,-5.625pt)--++(-6.495190528383289pt,0 pt) -- ++(3.2475952641916446pt,5.625pt);
\end{scriptsize}
\end{tikzpicture}
    \caption{The AP diagram ($h$ denotes the size of the discretization)}
    \label{fig:APdiag}
\end{figure}

\section{Chapman-Enskog expansion}\label{sec:CE}
The aim of this section is to derive a hierarchy of macroscopic models from which we will derive a macroscopic coupling criterion. From now on, we set $d=1$, $\Omega_x=[0,x_\star]$ with periodic boundary conditions and $\Omega_v=\R$. We also assume that the electrical field $E$ is periodic on $[0,x_\star]$. Let us first recall the integro-differential problem we are interested in :
\begin{equation}\label{VlasovBGK}\tag{VBGK}
    \left\lbrace\begin{aligned}
        \dt f^\varepsilon + \frac{1}{\varepsilon}\mathcal{T}(f^\varepsilon)=\frac{1}{\varepsilon^2}(\rho^\varepsilon\M-f^\varepsilon),\\
        f^\varepsilon(0,x,v)=f_0(x,v),
    \end{aligned}\right.
\end{equation}
where $$\mathcal{T}(f)= v\dx f + E\dv f.$$ Let us make some assumptions on the Maxwellian $\M$:
\begin{equation}
    \left\lbrace\begin{aligned}
        &\M(v) > 0 ,\quad\forall v\in\R,\\
        &\M(-v)=\M(v), \quad\forall v\in\R,\\
        &m_0 = 1,\quad m_2<+\infty \mbox{ and } m_4 <\infty.
    \end{aligned}\right.
\end{equation}
As a consequence of the symmetric domain in velocity and the symmetry of $\M$, odd moments of $\M$ vanish. The last assumption $m_0 = 1$ can easily be obtained by normalizing the Maxwellian. In particular, these assumptions are satisfied with the centered Gaussian. Let us set $\gamma(v)=\frac{1}{\M(v)}$ and introduce the measure $$\dd\gamma=\gamma(v)\,\dd v=\frac{\dd v}{\M(v)},$$ and $L^2(\dd x\,\dd\gamma)$ the space of square integrable functions against the measure $\dd x\dd\gamma$ equipped with the scalar product $$(f_1,f_2)_{L^2(\dd x\,\dd\gamma)}=\int_{\Omega_{x}\times\Omega_{v}} f_{1}f_{2} \,\dd x\,\dd\gamma.$$ With an initial data in $L^2(\dd x\,\dd\gamma)$, there is a unique solution to \eqref{VlasovBGK} (see, e.g., \cite{AllaireBlancDespresGolse:lecturenotes}) which conserves mass and nonnegativity.

\noindent One can define the null space of the linear BGK operator \eqref{Q_BGK}. Let $$\mathcal{N}=\left\lbrace f=\rho\M\mbox{ where } f\in L^2(\dd x\,\dd\gamma),\,\rho=\langle f\rangle \right\rbrace.$$
The space $\mathcal{N}$ is sometimes referred to as the equilibrium manifold. In particular, one has that $$\mathcal{N}^\perp=\left\lbrace f\in L^2(\dd x\,\dd\gamma)\mbox{ such that }\langle f\rangle=0 \right\rbrace.$$ With these notations, one can decompose $f$ as its equilibrium part in $\mathcal{N}$ plus a perturbative part in $\mathcal{N}^\perp$. Note that the perturbative part is not necessarily small.

\noindent Let us now introduce the so-called Chapman-Enskog expansion of the distribution function $f^\varepsilon$: 
\begin{equation}\label{eq:CE}
    f^\varepsilon(t,x,v)=\rho^\varepsilon(t,x)\M(v) + \sum^\infty_{k=1}\varepsilon^k h^{(k)}(t,x,v).
\end{equation}
This expansion comes with the following assumptions. First, the functions $h^{(k)}$ do not depend on the parameter $\varepsilon$. Secondly, the functions $h^{(k)}$ are functions of the density $\rho^\varepsilon$, the electric field $E$, the velocity variable $v$ and the Maxwellian $\M$. Thirdly, we assume that $h^{(k)}\in\mathcal{N}^\perp$ for all $k$ and is therefore mean free: $$\langle\, h^{(k)}\,\rangle=0,\quad\forall k\geq1.$$ 

\noindent To derive a hierarchy of models, let us consider truncations of order $K\in\mathbb{N}$ of the Chapman-Enskog expansion:
\begin{equation}\label{eq:CEdiscrete}
    f^\varepsilon(t,x,v)=\rho^\varepsilon(t,x)\M(v)+\sum_{k=1}^K\varepsilon^k h^{(k)}(t,x,v).
\end{equation}
The next step is to plug this expansion in \eqref{VlasovBGK}. It leads to 
\begin{equation*}
    \begin{aligned}
        \dt(\rho^\varepsilon\M) + \dt\sum_{k=1}^K\varepsilon^k h^{(k)} = -\frac{1}{\varepsilon}\T(\rho^\varepsilon M) -\sum_{k=1}^K\varepsilon^{k-1} \T(h^{(k)}) - \frac{1}{\varepsilon}\sum_{k=1}^K\varepsilon^{k-1} h^{(k)}.
    \end{aligned}
\end{equation*}
Multiplying by $\varepsilon$ and rearranging the terms, one obtains
\begin{equation*}
    \begin{aligned}
         \sum_{k=0}^{K-1}\varepsilon^{k} h^{(k+1)} = -\T(\rho^\varepsilon\M) - \sum_{k=1}^K\varepsilon^{k} \T(h^{(k)}) -\dt\sum_{k=2}^{K+1}\varepsilon^k h^{(k-1)} -\varepsilon\dt(\rho^\varepsilon\M).
    \end{aligned}
\end{equation*}
We now identify powers of $\varepsilon$:

\begin{subequations}\label{eq:identification}
\begin{alignat}{2}
k=0:&\quad &h^{(1)} = &-\T(\rho^\varepsilon\M),\label{eq:g1}\\
k=1:&\quad &h^{(2)} = &-\dt(\rho^\varepsilon\M) -\T(h^{(1)}),\label{eq:g2}\\
2\leq k\leq K-1:&\quad &h^{(k+1)} = &-\dt h^{(k-1)} - \T(h^{(k)}).\label{eq:gk}
\end{alignat}
\end{subequations}

\subsection{Macroscopic model}\label{sec:fluid models}
To derive the fluid model, let us truncate the Chapman-Enskog expansion at first order $K=1$: 
\begin{equation}\label{eq:CEorder1}
    f^\varepsilon=\rho^\varepsilon\M + \varepsilon h^{(1)}.
\end{equation}
We start by integrating \eqref{VlasovBGK} in velocity: 
$$\dt\rho^\varepsilon+\frac{1}{\varepsilon}\dx\langle\, vf^\varepsilon \,\rangle=0.$$
Then $f^\varepsilon$ is replaced by its expression \eqref{eq:CEorder1}:
$$\dt\rho^\varepsilon+\frac{1}{\varepsilon}\dx\left(\langle\, v\rho^\varepsilon\M\,\rangle+\langle\,\varepsilon vh^{(1)}\,\rangle\right)=0,$$
where $h^{(1)}$ is given by the identification \eqref{eq:g1}. Using the fact that $\rho^\varepsilon$ does not depend on the velocity and that odd moments of the Maxwellian are zero, we obtain by plugging in the expression of $h^{(1)}$:
\begin{equation}\label{intVH1}
    \dt\rho^\varepsilon+\dx\langle\, v( - v\dx \rho^\varepsilon\M - E\rho^\varepsilon\dv\M)\,\rangle=0.
\end{equation}
Finally, assuming that $\rho^\varepsilon\rightarrow\rho$ as $\varepsilon\rightarrow0$, we formally obtain the drift-diffusion model:
$$\dt\rho-\dx\left(m_2\dx\rho+m_1'E\rho\right)=0,$$
where $m_1'=\langle\, v\dv\M\,\rangle$ denotes the first moment of the derivative of the Maxwellian. Note that with our choice of $\M(v)$ as a centered Gaussian, an integration by parts allows us to obtain $m_1'=-1$. The following proposition is also rigorously proven in \cite{DolbeaultMouhotSchmeiser2015}.
\begin{Proposition}
Let $f$ be a solution to \eqref{VlasovBGK} with an initial data in $L^2(\frac{\dd x\dd v}{F})$. We assume that $E\in\mathcal{C}^3(\Omega_x)$. Therefore, $f^\varepsilon\rightarrow \rho\M$ as $\varepsilon\rightarrow0$ in $L^2(\frac{\dd x\dd v}{F})$ where $\rho$ is solution of the following drift-diffusion equation:
\begin{equation}\label{DD}\tag{$DD$}
    \dt\rho-\dx\left(m_2\dx\rho+m_1'E\rho\right)=0\quad\mbox{with}\quad \rho_0=\langle\, f_0\,\rangle.
\end{equation}    
\end{Proposition}

\subsection{Higher order macroscopic model}\label{sec:HigherOrder}
To derive the third order fluid model, we truncate the Chapman-Enskog expansion at third order $K=3$: 
\begin{equation}\label{eq:CEorder3}
    f^\varepsilon=\rho^\varepsilon\M+\varepsilon h^{(1)}+\varepsilon^2h^{(2)}+\varepsilon^3h^{(3)}.
\end{equation}
We will see in the derivation process that the second order yields no additional information.

\noindent Again, we start by integrating \eqref{VlasovBGK} in velocity and we replace $f$ by its expansion \eqref{eq:CEorder3}:
\begin{equation}\label{eq:CE_DD_tmp}
    \dt\rho^\varepsilon+\frac{1}{\varepsilon}\dx\left\langle v\rho^\varepsilon\M \right\rangle+\dx\left\langle vh^{(1)}+\varepsilon vh^{(2)}+\varepsilon^2vh^{(3)}\right\rangle=0.
\end{equation}
At this point, we use the identification \eqref{eq:identification} to compute the perturbations $h^{(1)}$, $h^{(2)}$ and $h^{(3)}$. One obtains
$$h^{(1)} = -v\dx(\rho^\varepsilon\M)-E\dv(\rho^\varepsilon\M),$$
$$h^{(2)}=-\M\dt\rho^\varepsilon+v^2\M\dxx\rho^\varepsilon+\dx(E\rho^\varepsilon)v\dv\M+E\dx\rho^\varepsilon\dv(v\M)+E^2\rho^\varepsilon\dvv\M.$$
We replace $h^{(1)}$ and $h^{(2)}$ by their expressions in $h^{(3)}=-\dt h^{(1)} - \T(h^{(2)})$ to obtain:
\begin{equation*}\begin{aligned}
    h^{(3)} = 2\dtx\rho^\varepsilon v\M &+ 2E\dt\rho^\varepsilon\dv\M -\dxxx\rho^\varepsilon v^3\M -\dxx(E\rho^\varepsilon)v^2\dv\M\\ &-\dx(E\dx\rho^\varepsilon)v\dv(v\M) -\dx(E^2\rho^\varepsilon)v\dvv\M\\ &-E\dxx\rho^\varepsilon\dv(v^2\M) -E\dx(E\rho^\varepsilon)\dv(v\dv\M)\\ &-E^2\dx\rho^\varepsilon\dvv(v\M) - E^3\rho^\varepsilon\dvvv\M.
\end{aligned}\end{equation*}
Using the identity $\dv\M(v)=-v\M(v)$, one obtains
\begin{equation*}\begin{aligned}
    h^{(3)} = 2\dtx\rho^\varepsilon v\M &- 2E\dt\rho^\varepsilon v\M -\dxxx\rho^\varepsilon v^3\M +\dxx(E\rho^\varepsilon)v^3\M\\ &-\dx(E\dx\rho^\varepsilon)(v\M-v^3\M) -\dx(E^2\rho^\varepsilon)(-v\M+v^3\M)\\ &-E\dxx\rho^\varepsilon(2v\M-v^3\M) -E\dx(E\rho^\varepsilon)(-2v\M+v^3\M)\\ &-E^2\dx\rho^\varepsilon(-3v\M+v^3\M) - E^3\rho^\varepsilon(3v\M-v^3\M).
\end{aligned}\end{equation*}
We set $J^\varepsilon=(\dx\rho^\varepsilon-E\rho^\varepsilon)$ and $h^{(3)}$ rewrites as:
\begin{equation}\label{eq:g3DD}
\begin{aligned}
    h^{(3)} =&~2v\M\dt J^\varepsilon -v^3\M \dxx J^\varepsilon -(v\M-v^3\M)\dx(EJ)\\
    &-(2v\M-v^3\M)E\dx J^\varepsilon +(3v\M-v^3\M) E^2J.
\end{aligned}\end{equation}
We simplify \eqref{eq:CE_DD_tmp} by using the fact that $v\M$ and $vh^{(2)}$ are odd in $v$, obtaining
\begin{equation}\label{ToPDD}
    \dt\rho^\varepsilon+\dx\left\langle vh^{(1)}+\varepsilon^2vh^{(3)}\right\rangle=0.
\end{equation}
We have already shown in \eqref{intVH1} that $$\dx\left\langle vh^{(1)}\right\rangle= -m_2\dx\left(\dx\rho^\varepsilon-E\rho^\varepsilon\right)=-m_2\dx J^\varepsilon.$$ Therefore, \eqref{ToPDD} gives $\dt\rho^\varepsilon= m_2\dx J^\varepsilon + \mathcal{O}(\varepsilon^2)$. It follows that when $\varepsilon$ is small, $$\dt J^\varepsilon=m_2(\dxx J^\varepsilon-E\dx J^\varepsilon) + \mathcal{O}(\varepsilon^2).$$ We use this relation to replace the time derivative of $J^\varepsilon$ in \eqref{eq:g3DD} which gives:
\begin{equation*}\begin{aligned}
    h^{(3)} = &~2m_2v\M\dxx J^\varepsilon - 2m_2v\M E\dx J^\varepsilon -v^3\M \dxx J^\varepsilon -(v\M-v^3\M)\dx(EJ^\varepsilon)\\
    &-(2v\M-v^3\M)E\dx J^\varepsilon +(3v\M-v^3\M) E^2J^\varepsilon + \Oeps{2}.
\end{aligned}\end{equation*}
We can now compute the remaining integral:
\begin{equation*}\begin{aligned}
    \dx\left\langle vh^{(3)}\right\rangle = \dx\bigg[&2m_2^2\dxx J^\varepsilon - 2m_2^2 E\dx J^\varepsilon -m_4 \dxx J^\varepsilon -(m_2-m_4)\dx(EJ^\varepsilon)\\
    &-(2m_2-m_4)E\dx J^\varepsilon +(3m_2-m_4) E^2J^\varepsilon + \Oeps{2}\bigg].
\end{aligned}\end{equation*}
With our choice of $\M(v)$, we can explicitly compute $m_2=1$ and $m_4=3$. Therefore one has:
\begin{equation*}\begin{aligned}
    \dx\left\langle vh^{(3)}\right\rangle = \dx\bigg[ 2\dx(EJ^\varepsilon) - E\dx J^\varepsilon -\dxx J^\varepsilon  +\mathcal{O}(\varepsilon^2)\bigg].
\end{aligned}\end{equation*}
Finally, we obtain a higher order model in the drift-diffusion limit.
\begin{Proposition}
(formal) Let $f^\varepsilon$ be a solution of \eqref{VlasovBGK} with an initial data in $L^2(\frac{\dd x\dd v}{F})$. Assuming that $f^\varepsilon$ admits a Chapman-Enskog expansion of order $K=3$, the truncated model up to order 2 in $\varepsilon$ is given by a higher order drift-diffusion equation. The macroscopic density $\rho^\varepsilon=\langle\, f^\varepsilon\,\rangle$ is a solution of:
\begin{equation}\label{DDH}\tag{$\widetilde{DD}$}
    \dt\rho^\varepsilon-\dx J^\varepsilon + \varepsilon^2\dx(2\dx(EJ^\varepsilon) - E\dx J^\varepsilon -\dxx J^\varepsilon) = \mathcal{O}(\varepsilon^4),
\end{equation}   
where $\quad J^\varepsilon=(\dx\rho^\varepsilon-E\rho^\varepsilon)$.
\end{Proposition}
\section{Micro-Macro Model}\label{sec:MM}
In this part, we derive a micro-macro model for \eqref{VlasovBGK}. The micro-macro approach was first used to derive AP schemes for the radiative heat transfer in \cite{KlarSchmeiser2001}. It was applied to the Boltzmann equation in \cite{BennouneLemouMieussens2008, LemouMieussens2008} and to the Vlasov-Poisson-BGK equation in \cite{CrouseillesLemou2011}. We will then introduce a micro-macro finite volume scheme that enjoys the property of being Asymptotic Preserving which is a crucial point of the coupling method we want to introduce. 

\subsection{Continuous setting}
Let us decompose the distribution $f$ as follows:
\begin{equation}\label{rhoM+g}
    f^\varepsilon = \rho^\varepsilon\M+g^\varepsilon.
\end{equation}
We introduce the orthogonal projector $\Pi$ in $L^2(\dd x\dd\gamma)$ on $\mathcal{N}$ defined for all $f\in L^2(\dd x\dd\gamma)$ by:
$$\Pi f=\langle f\rangle\M.$$
To help us in the derivation of the micro-macro model, let us first introduce the following lemma:
\begin{Lemma}\label{lem:gMM}
    Let $f^\varepsilon=\rho^\varepsilon\M+g^\varepsilon$ be a solution of \eqref{VlasovBGK}. Therefore one has: $$\Pi(g^\varepsilon)=\Pi(\dt g^\varepsilon)=\Pi (\T(\rho^\varepsilon\M)) = (I-\Pi)(\dt(\rho^\varepsilon\M))=0.$$
    Moreover, one has $\Pi(\T g^\varepsilon) =\dx\langle vg^\varepsilon\rangle\M$.
\end{Lemma}
\begin{proof}
    Let us first show that $\Pi g^\varepsilon=0$ which means that $g^\varepsilon\in\mathcal{N}^\perp$. Using the definitions \eqref{rhoM+g} of $\rho^\varepsilon$ and $g^\varepsilon$, one has $$\Pi g^\varepsilon=\left\langle f^\varepsilon-\rho^\varepsilon\M \right\rangle\M=\left(\langle f^\varepsilon\rangle-\rho^\varepsilon\langle\M\rangle\right)\M=(\rho^\varepsilon-\rho^\varepsilon)\M=0.$$
    Next, using the fact that $\Pi g^\varepsilon=0$, it follows that $$\Pi\dt g^\varepsilon = \dt\Pi g^\varepsilon = 0.$$
    To show that $\Pi(\T(\rho\M))=0$, the definition of the transport operator $\mathcal{T}$ is used: $$\Pi(\T(\rho\M)) = \Pi(v\dx (\rho^\varepsilon\M)+E\dv (\rho\M)) =\Pi(v\dx (\rho^\varepsilon\M))=\dx\rho^\varepsilon\Pi(v\M)=\dx\rho^\varepsilon\langle v\M\rangle\M.$$ Recalling that the velocity domain of integration is symmetric and that $\M(v)$ is even in $v$, one gets that $$\Pi(\T(\rho^\varepsilon\M))=0.$$
    Let us show that $(I-\Pi)(\dt(\rho^\varepsilon\M))=0$. We use the fact that $\rho$ does not depend on the velocity:
    $$(I-\Pi)(\dt(\rho^\varepsilon\M))=\dt(\rho^\varepsilon\M)-\dt\rho^\varepsilon\langle\M\rangle\M=\dt(\rho^\varepsilon\M)-\dt(\rho^\varepsilon\M) = 0.$$
    Finally, $$\Pi(\T g^\varepsilon) = \langle v\dx g^\varepsilon + E\dv g^\varepsilon \rangle\M = \dx\langle vg^\varepsilon\rangle\M + E\langle\dv g^\varepsilon \rangle\M = \dx\langle vg^\varepsilon\rangle\M.$$
\end{proof}
\noindent To derive the micro-macro model, we start by injecting \eqref{rhoM+g} in \eqref{VlasovBGK}:  
\begin{equation}\label{eq:VlasovCE}
    \dt(\rho^\varepsilon\M)+\dt g^\varepsilon+\frac{1}{\varepsilon} (\T(\rho^\varepsilon\M)+\T g^\varepsilon)=\frac{-1}{\varepsilon^2} g^\varepsilon.
\end{equation}
\noindent We then apply $(I-\Pi)$ to \eqref{eq:VlasovCE}, and simplify using Lemma \ref{lem:gMM} to obtain the micro part of the model:
\begin{equation}\label{eq:micro}
    \dt g^\varepsilon+\frac{1}{\varepsilon}(\T g^\varepsilon-\Pi(\T g^\varepsilon)) + \frac{1}{\varepsilon}\T(\rho^\varepsilon\M)=\frac{-1}{\varepsilon^2}g^\varepsilon.
\end{equation}
The macro part is obtained by applying $\Pi$ to \eqref{eq:VlasovCE} and using Lemma \ref{lem:gMM}. It leads to:
\begin{equation}\label{eq:macro}
    \dt\rho^\varepsilon\M+\frac{1}{\varepsilon}\Pi(\T g^\varepsilon) = 0.
\end{equation}
Finally, the micro-macro model is given by:
\begin{equation}\label{Micro}\tag{$Micro$}
    \dt g+\frac{1}{\varepsilon}(\T g^\varepsilon - \dx\langle vg^\varepsilon\rangle\M + \T(\rho^\varepsilon\M)) =\frac{-1}{\varepsilon^2}g^\varepsilon,
\end{equation}
\begin{equation}\label{Macro}\tag{$Macro$}
    \dt\rho^\varepsilon+\frac{1}{\varepsilon}\dx \langle vg^\varepsilon\rangle=0.
\end{equation}
The following proposition states the equivalence between the \eqref{Micro}-\eqref{Macro} model and the original equation \eqref{VlasovBGK} \cite{CrouseillesLemou2011}.
\begin{Proposition}
(formal)
\begin{enumerate}
    \item If $f^\varepsilon$ is a solution to \eqref{VlasovBGK} with an initial data in $L^2(\dd x\dd\gamma)$, then $(\rho^\varepsilon, g^\varepsilon) = (\langle f^\varepsilon\rangle, f^\varepsilon-\langle f^\varepsilon\rangle\M)$ is a solution to \eqref{Micro}-\eqref{Macro} with the associated initial data $$\rho_0 = \langle f_0\rangle,\quad g_0 = f_0-\rho_0\M.$$
    \item Conversely, if $(\rho^\varepsilon, g^\varepsilon)$ is a solution to \eqref{Micro}-\eqref{Macro} with initial data $\rho^\varepsilon(t = 0) = \rho_0$ and $g^\varepsilon(t = 0) = g_0$ with $\langle g_0\rangle = 0$ then $\langle g^\varepsilon(t)\rangle = 0,$ for all $t > 0$ and $f^\varepsilon = \rho^\varepsilon\M + g^\varepsilon$ is a solution to \eqref{VlasovBGK} with initial data $f_0 = \rho_0\M +g_0.$
\end{enumerate}
\end{Proposition}
\subsection{Discrete setting}\label{sec:DiscSetting}
Let us now tackle the discretization of the \eqref{Micro}-\eqref{Macro} model. We adopt a finite volume approach to discretize in the phase-space. 
\paragraph{The mesh.}We restrict the velocity domain to a bounded symmetric segment $[-v_\star,v_\star]$ as it is impractical to implement a numerical scheme on an unbounded domain. We consider a mesh of the interval composed of $N_v=2L$ velocity cells arranged symmetrically around $v=0$. We thus obtain $2L+1$ interface points that we label $v_\jpd$ with $j=-L,\dots,L$. Therefore: 
$$v_{-L+\frac{1}{2}}=-v_\star, \quad v_{\frac{1}{2}}=0,\quad v_\jpd=-v_{-\jpd} \quad\forall j=0,\dots,L.$$
The cells of the velocity mesh are given by 
$$\mathcal{V}_j=(v_\jmd,v_\jpd),\quad j\in \mathcal{J}=\lbrace -L+1,\dots,L \rbrace.$$
Each cell $\mathcal{V}_j$ has a constant length $\Delta v$ and midpoint $v_j$. The velocity mesh is illustrated in Figure~\ref{fig:VelDisc}.\\
\begin{figure}
    \centering
    \begin{tikzpicture}[line cap=round,line join=round,>=triangle 45,x=1.4605733794804243cm,y=1.407896962515574cm]
\definecolor{ttqqqq}{rgb}{0.2,0.,0.}
\definecolor{qqttcc}{rgb}{0.,0.2,0.8}
\definecolor{ffqqtt}{rgb}{1.,0.,0.2}
\clip(2.0,3.0) rectangle (13.0,5.0);
\draw [line width=1.6pt,color=qqttcc] (3.5,4.)-- (4.,4.);
\draw [line width=1.6pt,dash pattern=on 4pt off 4pt,color=qqttcc] (4.,4.)-- (6.,4.);
\draw [line width=1.6pt,color=qqttcc] (6.,4.)-- (7.,4.);
\draw [line width=1.6pt,color=qqttcc] (7.,4.)-- (8.,4.);
\draw [line width=1.6pt,color=qqttcc] (8.,4.)-- (9.,4.);
\draw [line width=1.6pt,dash pattern=on 4pt off 4pt,color=qqttcc] (9.,4.)-- (11.,4.);
\draw [line width=1.6pt,color=qqttcc] (11.,4.)-- (11.5,4.);
\draw [line width=1.2pt,dash pattern=on 2pt off 2pt] (7.5,4.5)-- (8.5,4.5);
\draw [line width=1.2pt,dash pattern=on 2pt off 2pt] (7.,3.5)-- (8.,3.5);
\draw (7.36,4.4) node[anchor=north west] {$0$};
\draw (7.748048297914481,4.892080548015483) node[anchor=north west] {$\Delta v_1$};
\draw (7.304785568887788,3.4736398151300825) node[anchor=north west] {$\Delta v_{\frac{1}{2}}$};
\draw (3.2,4.4) node[anchor=north west] {$-v_\star$};
\draw (11.3,4.4) node[anchor=north west] {$v_\star$};
\draw (5.7,4.4) node[anchor=north west] {$v_{-1}$};
\draw (6.8,4.4) node[anchor=north west] {$v_{0}$};
\draw (7.8,4.4) node[anchor=north west] {$v_{1}$};
\draw (8.8,4.4) node[anchor=north west] {$v_{2}$};
\draw (10.8,4.4) node[anchor=north west] {$v_{L}$};
\draw (3.8,4.4) node[anchor=north west] {$v_{-L+1}$};
\draw (3.15,4.0) node[anchor=north west] {$v_{-L+\frac{1}{2}}$};
\draw (11.320745893869624,4.0) node[anchor=north west] {$v_{L+\frac{1}{2}}$};
\draw (6.3562033287706665,4.0) node[anchor=north west] {$ v_{-\frac{1}{2}}$};
\draw (7.4466296421763305,4.0) node[anchor=north west] {$ v_{\frac{1}{2}}$};
\draw (8.430672900615589,4.0) node[anchor=north west] {$ v_{\frac{3}{2}}$};
\draw (3.297690498486487,3.7) node[anchor=north west] {$ =v_{-L}$};
\draw (11.294150130128022,3.7) node[anchor=north west] {$ =v_{L+1}$};
\begin{scriptsize}
\draw [fill=black] (3.5,4.) circle (1.5pt);
\draw [fill=ffqqtt] (4.,4.) ++(-2.0pt,0 pt) -- ++(2.0pt,2.0pt)--++(2.0pt,-2.0pt)--++(-2.0pt,-2.0pt)--++(-2.0pt,2.0pt);
\draw [fill=ffqqtt] (6.,4.) ++(-2.0pt,0 pt) -- ++(2.0pt,2.0pt)--++(2.0pt,-2.0pt)--++(-2.0pt,-2.0pt)--++(-2.0pt,2.0pt);
\draw [fill=ffqqtt] (7.,4.) ++(-2.0pt,0 pt) -- ++(2.0pt,2.0pt)--++(2.0pt,-2.0pt)--++(-2.0pt,-2.0pt)--++(-2.0pt,2.0pt);
\draw [fill=ffqqtt] (8.,4.) ++(-2.0pt,0 pt) -- ++(2.0pt,2.0pt)--++(2.0pt,-2.0pt)--++(-2.0pt,-2.0pt)--++(-2.0pt,2.0pt);
\draw [fill=ffqqtt] (9.,4.) ++(-2.0pt,0 pt) -- ++(2.0pt,2.0pt)--++(2.0pt,-2.0pt)--++(-2.0pt,-2.0pt)--++(-2.0pt,2.0pt);
\draw [fill=ffqqtt] (11.,4.) ++(-2.0pt,0 pt) -- ++(2.0pt,2.0pt)--++(2.0pt,-2.0pt)--++(-2.0pt,-2.0pt)--++(-2.0pt,2.0pt);
\draw [fill=black] (11.5,4.) circle (1.5pt);
\draw [fill=black] (6.5,4.) circle (1.5pt);
\draw [fill=black] (7.51822482357297,4.) circle (1.5pt);
\draw [fill=black] (8.5,4.) circle (1.5pt);
\draw [fill=ttqqqq,shift={(7.5,4.5)},rotate=90] (0,0) ++(0 pt,2.25pt) -- ++(1.9485571585149868pt,-3.375pt)--++(-3.8971143170299736pt,0 pt) -- ++(1.9485571585149868pt,3.375pt);
\draw [fill=ttqqqq,shift={(8.5,4.5)},rotate=270] (0,0) ++(0 pt,2.25pt) -- ++(1.9485571585149868pt,-3.375pt)--++(-3.8971143170299736pt,0 pt) -- ++(1.9485571585149868pt,3.375pt);
\draw [fill=ttqqqq,shift={(7.,3.5)},rotate=90] (0,0) ++(0 pt,2.25pt) -- ++(1.9485571585149868pt,-3.375pt)--++(-3.8971143170299736pt,0 pt) -- ++(1.9485571585149868pt,3.375pt);
\draw [fill=ttqqqq,shift={(8.,3.5)},rotate=270] (0,0) ++(0 pt,2.25pt) -- ++(1.9485571585149868pt,-3.375pt)--++(-3.8971143170299736pt,0 pt) -- ++(1.9485571585149868pt,3.375pt);
\end{scriptsize}
\end{tikzpicture}
    \caption{Discretization of the velocity domain.}
    \label{fig:VelDisc}
\end{figure}

\noindent In space, because of the periodic boundary conditions, we consider a discretization of the 1-D torus $\mathbb{T}$ into $N_x$ primal cells $$\mathcal{X}_i=(x_\imd,x_\ipd), \quad i\in\mathcal{I}=\mathbb{Z}/N_x\mathbb{Z},$$ of constant length $\Delta x$ and centers $x_i$. We also define the dual cells $$\mathcal{X}_\ipd=(x_i,x_{i+1}), \quad i\in\mathcal{I},$$ of constant length $\Delta x$ and centers $x_\ipd$.
The primal control volumes in the phase space are defined by 
$$K_{ij}= \mathcal{X}_i\times\mathcal{V}_j,\quad\forall(i,j)\in\mathcal{I}\times\mathcal{J}.$$
The dual control volumes are defined by 
$$K_{\ipd,j}= \mathcal{X}_\ipd\times\mathcal{V}_j,\quad\forall(i,j)\in\mathcal{I}\times\mathcal{J}.$$
Finally, we set a time step $\Delta t>0$ and we define $t^n=n\Delta t,\quad n\in\mathbb{N}$.

\paragraph{The discrete Maxwellian.} We assume that we are given cell values $(M_j)_{j\in\mathcal{J}}$. We then assume that:
\begin{equation}\label{eq:hypDiscMaxw}
    \left\lbrace\begin{aligned}
    &M_j>0,\quad M_j=M_{-j+1} \quad\forall j=1,\dots,L\,,\\
    &\sum_{j\in\mathcal{J}}M_j\Delta v_j = m_0^{\Delta v}=1,\\
    &M_{L}=M_{L+1},\quad M_{-L}=M_{-L+1}.
    \end{aligned}\right.
\end{equation}
These assumptions are a discrete version of the ones we made in the continuous case, namely, the positivity, the fact that $\M$ is even in velocity, and the unit mass. The last assumption is simply a zero boundary flux. For a sufficiently large domain in velocity, it is relevant due to the fast decay of the Gaussian. The unit mass can easily be obtained by taking $M_j=c\M(v_j)$, $j\in\mathcal{J}$ where $c$ is an appropriate constant. Let us introduce the discrete integration operator in velocity: for $f=\left(f_j\right)_{j\in\mathcal{J}}$
\begin{equation}
    \langle f\rangle_\Delta=\sum_{j\in\mathcal{J}}f_j\Delta v.
\end{equation}
Finally, we introduce the discrete moments of the discrete Maxwellian \eqref{eq:hypDiscMaxw}:
\begin{equation}\label{discMoments}
    m_k^{\Delta v}=\langle v^kM \rangle_\Delta=\sum_{j\in\mathcal{J}}v_j^kM_j\Delta v.
\end{equation}
\paragraph{Semi-discretization in the phase-space.}
We start by considering a semi-discretization in the phase-space of the \eqref{Micro}-\eqref{Macro} model. We choose to approximate the perturbation $g^\varepsilon$ on the dual cells while the density $\rho^\varepsilon$ is approximated on the primal mesh. This choice of staggered meshes will result in a more compact stencil for the asymptotic scheme and is quite standard in the literature \cite{LemouMieussens2008,LiuMieusens2010, CrouseillesLemou2011}. 

\noindent Let $(i,j)\in \mathcal{I}\times\mathcal{J}$. We start start by integrating \eqref{Macro} on $\mathcal{X}_{i}$:
$$\int_{\mathcal{X}_{i}}\left[\dt\rho^\varepsilon + \frac{1}{\varepsilon}\dx\langle vg^\varepsilon\rangle\right]\dd x=0.$$
Let $\rho^\varepsilon_i(t)$ be an approximation of $\frac{1}{\Delta x}\int_{\mathcal{X}_{i}} \rho^\varepsilon(t,x)\dd x$. After integrating the space derivative, one then obtains a continuous in time finite volume scheme for \eqref{Macro}:
\begin{equation}\label{MacroSpace}
    \frac{\dd}{\dd t}\rho^\varepsilon_i = \frac{1}{\Delta x}\left(J_\ipd^K-J_\imd^K \right),
\end{equation}
where the macroscopic flux $\left(J^K_\ipd\right)_{i\in\mathcal{I}}$ is approximated by
\begin{equation}\label{MacroFluxCont}
    J^K_\ipd = -\frac{1}{\varepsilon}\langle vg^\varepsilon_\ipd \rangle_\Delta.
\end{equation}
Next, we deal with the \eqref{Micro} equation which is integrated on $K_{\ipd,j}$:
$$\int_{K_{\ipd,j}}\dt g^\varepsilon\dd x\dd v + \frac{1}{\varepsilon}\int_{K_{\ipd,j}} \left[\T(g^\varepsilon)-\dx \langle vg^\varepsilon\rangle\M+\T(\rho^\varepsilon\M) \right]\dd x\dd v=\frac{-1}{\varepsilon^2}\int_{K_{\ipd,j}} g^\varepsilon\dd x\dd v.$$
Let $g^\varepsilon_{\ipd,j}(t)$ be an approximation of $\frac{1}{\Delta x\Delta v}\int_{K_{\ipd,j}} g^\varepsilon(t,x,v)\dd x\dd v.$ One then obtains
\begin{equation}
\begin{aligned}
    -\varepsilon\Delta x\Delta v\left(\frac{\dd}{\dd t} g^\varepsilon_{\ipd,j} + \frac{1}{\varepsilon^2} g^\varepsilon_{\ipd,j}\right) = & \underbrace{\int_{K_{\ipd,j}} \T(g^\varepsilon)\dd x\dd v}_\text{\circled{A}} - \underbrace{\int_{K_{\ipd,j}} \dx \langle vg^\varepsilon\rangle\M\dd x\dd v}_\text{\circled{B}}\\
    &+ \underbrace{\int_{K_{\ipd,j}}\T(\rho^\varepsilon\M) \dd x\dd v}_\text{\circled{C}}.        
\end{aligned}
\end{equation}
$$ $$
Using the definition of the transport operator $\T$, one has:
\begin{equation*}
    \begin{aligned}
    \circled{A} &= \int_{\mathcal{V}_j} v\big(g^\varepsilon(t,x_{i+1},v)-g^\varepsilon(t,x_{i},v)\big)\dd v + \int_{\mathcal{X}_\ipd} E\big(g^\varepsilon(t,x,v_\jpd)-g^\varepsilon(t,x,v_\jmd)\big)\dd x,\\
    \circled{B} &= \int_{\mathcal{V}_j} \M\big(\langle vg^\varepsilon(t,x_{i+1},v)\rangle-\langle vg^\varepsilon(t,x_{i},v)\rangle\big)\dd v,\\
    \circled{C} &= \int_{\mathcal{V}_j} v\M\big(\rho^\varepsilon(t,x_{i+1})-\rho^\varepsilon(t,x_{i})\big)\dd v + \int_{\mathcal{X}_\ipd} E\rho^\varepsilon\big(\M(v_\jpd)-\M(v_\jmd)\big)\dd x.
    \end{aligned}
\end{equation*}
We now denote by $\left(\mathcal{F}^\varepsilon_{i,j}\right)_{ij}$ an approximation of the microscopic flux in space at interfaces $\left(x_i\right)_i$, namely
\begin{equation}\label{SpaceFluxIntegral}
    \mathcal{F}^\varepsilon_{i,j} \approx \int_{\mathcal{V}_j} \left[vg^\varepsilon(t,x_{i},v) - \M\langle vg^\varepsilon(t,x_{i},v)\rangle + v\M\rho^\varepsilon(t,x_{i})\right]\dd v.
\end{equation}
We also denote by $\left(\mathcal{G}^\varepsilon_{\ipd,\jpd}\right)_{ij}$ an approximation of the microscopic flux in velocity, namely
\begin{equation}\label{VelocityFluxIntegral}
    \mathcal{G}^\varepsilon_{\ipd,\jpd} \approx \int_{\mathcal{X}_\ipd} \left[Eg^\varepsilon(t,x,v_\jpd) + E\rho^\varepsilon\M(v_\jpd) \right]\dd x.
\end{equation}
Let us now present our choice of numerical fluxes. Let us first focus on the position. We choose a first-order upwind approximation for the first two terms of \eqref{SpaceFluxIntegral}. The use of staggered grids allows us to directly use the value $\rho^{\varepsilon}_i$ at interface $x_i$ for the last term. Summarizing, the numerical flux in position is given by:
\begin{equation}\label{FluxPosition}
    \mathcal{F}_{i,j}^{\varepsilon} = \left(v^{+}_jg_{\imd,j}^{\varepsilon, n}+v^{-}_jg_{\ipd,j}^{\varepsilon}\right)\Delta v - M_j\left\langle v^{+}g_{\imd}^{\varepsilon}+v^{-}g_{\ipd}^{\varepsilon}\right\rangle_\Delta\Delta v + v_jM_j\rho^{\varepsilon}_i\Delta v,\\
\end{equation}
where the notation $r^\pm=\frac{r\pm|r|}{2}$ is used. At the boundaries in position, the periodic setting implies
\begin{equation}
    \mathcal{F}_{0,j}^{\varepsilon} = \mathcal{F}_{N_x,j}^{\varepsilon}.
\end{equation}
In velocity, a first-order upwind approximation is used for the first term of \eqref{VelocityFluxIntegral}. The second term is a centered approximation of the discrete Maxwellian at interface $v_\jpd$. Since $E$ is given, it is explicitly discretized on the dual mesh, $E_\ipd=E(x_\ipd)$ and $\rho(x_\ipd)$ is approximated by 
\begin{equation}\label{RhoHalf}
    \rho^\varepsilon_\ipd=\frac{1}{2}(\rho^\varepsilon_i+\rho^\varepsilon_{i+1}).
\end{equation}
The numerical flux in velocity then read:
\begin{equation}\label{FluxVelocity}
    \mathcal{G}_{\ipd,\jpd}^{\varepsilon} = \left(E^{+}_{\ipd}g_{\ipd,j}^{\varepsilon, n}+E^{-}_{\ipd}g_{\ipd,j+1}^{\varepsilon}\right)\Delta x + E_\ipd\rho^{\varepsilon}_\ipd\frac{M_{j+1}+M_j}{2}\Delta x.
\end{equation}
Zero flux boundary conditions are applied in velocity and therefore we set
\begin{equation}
    \mathcal{G}_{\ipd,-L+\frac{1}{2}}^{\varepsilon}=\mathcal{G}_{\ipd,L+\frac{1}{2}}^{\varepsilon}=0.
\end{equation}
A continuous in time finite volume scheme for equation \eqref{Micro} finally reads:
\begin{equation}\label{MicroPhase}
    \frac{\dd}{\dd t} g^\varepsilon_{\ipd,j} + \frac{1}{\varepsilon^2} g^\varepsilon_{\ipd,j} = \frac{-T^\varepsilon_{\ipd,j}}{\varepsilon\Delta x\Delta v},
\end{equation}
where $T^\varepsilon_{\ipd,j}=\mathcal{F}^\varepsilon_{i+1,j}-\mathcal{F}^\varepsilon_{i,j} + \mathcal{G}^\varepsilon_{\ipd,\jpd} - \mathcal{G}^\varepsilon_{\ipd,\jmd}$.
\paragraph{Full discretization.}
In order to obtain an AP scheme, one must carefully choose the discretization in time. Following \cite{Lemou2010}, we adapt the so-called relaxed micro-macro scheme to our finite volume setting. This method falls into the framework of exponential time integrators \cite{HochbruckOstermann2010}.
Let $n\in\mathbb{N}$ and $(i,j)\in \mathcal{I}\times\mathcal{J}$. Let $\left(g_{\ipd,j}^{\varepsilon, n}\right)_{ij}$ be an approximation of $\left(g_{\ipd,j}^\varepsilon(t^n)\right)_{ij}$ and $\left(\rho^{\varepsilon,n}_i\right)_{i}$ an approximation of $\left(\rho^\varepsilon_i(t^n)\right)_{i}$. The first step is to multiply \eqref{MicroPhase} by $e^{t/\varepsilon^2}$ which gives:
\begin{equation}
    \frac{\dd}{\dd t}\left(g^\varepsilon_{\ipd,j}(t)e^{t/\varepsilon^2}\right) = \frac{-e^{t/\varepsilon^2}}{\varepsilon\Delta x\Delta v}T^\varepsilon_{\ipd,j}(t).
\end{equation}
Let us then integrate between $t^n$ and $t^{n+1}$ and divide by $e^{t^{n+1}/\varepsilon^2}$:
\begin{equation}
    g^\varepsilon_{\ipd,j}(t^{n+1}) = g^\varepsilon_{\ipd,j}(t^{n})e^{-\Delta t/\varepsilon^2} +  \int_{t^n}^{t^{n+1}}\frac{-e^{(t-t^{n+1})/\varepsilon^2}}{\varepsilon\Delta x\Delta v}T^\varepsilon_{\ipd,j}(t)\dd t.
\end{equation}
Then, $T^\varepsilon_{\ipd,j}(t)$ is approximated by $T^{\varepsilon,n}_{\ipd,j}=T^\varepsilon_{\ipd,j}(t^n)$ and the integral can be computed explicitly:
\begin{equation}
    \int_{t^n}^{t^{n+1}}\frac{-e^{(t-t^{n+1})/\varepsilon^2}}{\varepsilon\Delta x\Delta v}T^\varepsilon_{\ipd,j}(t)\dd t = \frac{-\varepsilon(1-e^{-\Delta t/\varepsilon^2})}{\Delta x\Delta v}T^{\varepsilon,n}_{\ipd,j}.
\end{equation}
Finally, the fully discretized microscopic equation reads:
\begin{equation}\label{MicroFull}
    g_{\ipd,j}^{\varepsilon,n+1} = g_{\ipd,j}^{\varepsilon,n}e^{-\Delta t/ \varepsilon^2} - \varepsilon(1-e^{-\Delta t/ \varepsilon^2})\frac{T^{\varepsilon,n}_{\ipd,j}}{\Delta x\Delta v},
\end{equation}
where 
\begin{equation}\label{TransportFull}
    T^{\varepsilon,n}_{\ipd,j}=\mathcal{F}^{\varepsilon,n}_{i+1,j}-\mathcal{F}^{\varepsilon,n}_{i,j} + \mathcal{G}^{\varepsilon,n}_{\ipd,\jpd} - \mathcal{G}^{\varepsilon,n}_{\ipd,\jmd}.
\end{equation}
The discretization in time of \eqref{MacroSpace} is quite standard and the stiff term is implicit:
\begin{equation}\label{rhoTime}
    \rho_i^{\varepsilon,n+1} = \rho_i^{\varepsilon,n} + \frac{\Delta t}{\Delta x}\left( J^{K,n+1}_\ipd-J^{K,n+1}_\imd \right).
\end{equation}
Note that \eqref{MicroFull} defines an explicit scheme. Moreover, \eqref{rhoTime} does not requires the inversion of a system. Indeed, \eqref{MicroFull} is explicitly computed at time $t^{n+1}$ and is then used to update the density in \eqref{rhoTime}. In practice, the method is therefore fully explicit.

\noindent Before stating the next proposition, let us introduce the following assumption: 
\begin{Assumption}\label{H0}
Let $\left(\rho^{\varepsilon,n}_i\right)_{i\in\mathcal{I}}$ be given by \eqref{rhoTime}. Then, 
\begin{equation}
 \quad\rho^{\varepsilon,n}_i\longrightarrow\rho_i^n,\quad\forall i\in\mathcal{I}.
\end{equation}
\end{Assumption}
\noindent This assumption corresponds to the convergence of $\rho^\varepsilon$ to $\rho$ as $\varepsilon\rightarrow0$. Such property is not trivial to obtain in the discrete setting. A rigorous proof of this result requires, among other things, uniform estimates in $\varepsilon$ of the discrete $L^2$-norm of $\rho$, $g$, and moments of $g$. It is outside the scope of this article and may be thoroughly investigated in upcoming work.

\noindent The following proposition states the AP property of our discretization of the \eqref{Micro}-\eqref{Macro} model.
\begin{Proposition}\label{prop:AP}
    Let $n\in \mathbb{N}$. Let $\left(g_{\ipd,j}^{\varepsilon,n}\right)_{ij}$ and $\left(\rho^{\varepsilon,n}_i\right)_{i}$ be given by the following micro-macro finite volume scheme:
    \begin{equation}\label{S_micro}\tag{$S_{Micro}$}
        g_{\ipd,j}^{\varepsilon,n+1} = g_{\ipd,j}^{\varepsilon,n}e^{-\Delta t/ \varepsilon^2} - \varepsilon(1-e^{-\Delta t/ \varepsilon^2})\frac{T^{\varepsilon,n}_{\ipd,j}}{\Delta x\Delta v},
    \end{equation}
    \begin{equation}\label{S_macro}\tag{$S_{Macro}$}
    \begin{aligned}
        \rho_i^{\varepsilon,n+1} = \rho_i^{\varepsilon,n} + \frac{\Delta t}{\Delta x}\left( J^{K,n+1}_\ipd-J^{K,n+1}_\imd \right),
    \end{aligned}
    \end{equation}
    where the macroscopic flux is given by \eqref{MacroFluxCont} and $T^{\varepsilon,n}_{\ipd,j}$ by \eqref{TransportFull}.\\
    Assuming that $\left(\rho_i\right)_{i\in\mathcal{I}}$ satisfies Assumption \ref{H0} and for a fixed mesh size $\Delta x$, $\Delta v>0$, the scheme enjoys the AP property in the diffusion limit. This property does not depend on the initial data and the associated limit scheme reads
    \begin{equation}\label{S_lim}\tag{$S_{Lim}$}
        \rho_i^{n+1} = \rho_i^{n} + \frac{\Delta t}{\Delta x}\left(J_\ipd^{F,n}-J_\imd^{F,n}\right),
    \end{equation}
    with the limit flux
    \begin{equation}\label{LimFlux}
        J_\ipd^{F,n} = \frac{m_2^{\Delta v}}{\Delta x}\left(\rho^n_{i+1}-\rho^n_{i}\right) + m_1'^{\Delta v} E_\ipd\rho^n_\ipd,
    \end{equation}
    where $m_2^{\Delta v}$ is given by \eqref{discMoments} and $m_1'^{\Delta v}$ is an approximation of the first moment of the derivative of the Maxwellian $m'_1$:
    \begin{equation}\label{m1Prime}
        m_1'^{\Delta v} = \left\langle v\frac{M_{\cdot+1}-M_{\cdot-1}}{2\Delta v}\right\rangle_\Delta.
    \end{equation}
\end{Proposition}
\begin{Remark}
    We want to point out that with our specific choice of $\M$ as a Gaussian, the quantity $m_1'~=~\langle v\dv M\rangle$ can be explicitly computed using an integration by part and $m_1'=-m_0$. On the discrete level, a discrete integration by parts can be performed. However, to obtain an equivalent result to the continuous case, one must neglect the boundary terms of the discrete Maxwellian. Under the assumption that the velocity domain is sufficiently large, the Gaussian decays fast enough that in practice one has $m_1'^{\Delta v}~=~-m_0^{\Delta v}$.
\end{Remark}
\begin{proof}
The mesh size $\Delta x$, $\Delta v>0$ being set, let us emphasize that we consider only the pointwise convergence of the scheme as $\varepsilon$ tends to 0. The first step is to study the asymptotic behaviour of the perturbation $\left(g_{\ipd,j}^{\varepsilon,n+1}\right)_{ij}$. By induction on $n$, let us show that $g_{\ipd,j}^{\varepsilon,n+1}\underset{\varepsilon\rightarrow0}{\longrightarrow}0$ for any initial data $(\rho^0,g^0)$ and for all $(i,j)\in\mathcal{I}\times\mathcal{J}$. At $n=0$, one has
\begin{equation}
    g_{\ipd,j}^{\varepsilon,1} = g_{\ipd,j}^{0}e^{-\Delta t/ \varepsilon^2} - \varepsilon(1-e^{-\Delta t/ \varepsilon^2})\frac{T^0_{\ipd,j}}{\Delta x\Delta v}.
\end{equation}
As $e^{-\Delta t/ \varepsilon^2}\underset{\varepsilon\rightarrow0}{\longrightarrow}0$ and since $\left(T^0_{\ipd,j}\right)_{ij}$ depends only on the initial data which itself is independent of $\varepsilon$, $$g_{\ipd,j}^{\varepsilon,1}\underset{\varepsilon\rightarrow0}{\longrightarrow}0,\quad\forall (i,j)\in\mathcal{I}\times\mathcal{J}.$$ Let us now assume that 
\begin{equation}\label{HypInduction}
    g_{\ipd,j}^{\varepsilon,n}\underset{\varepsilon\rightarrow0}{\longrightarrow}0\quad\forall (i,j)\in\mathcal{I}\times\mathcal{J}.
\end{equation}
We decompose the discrete transport operator as a macroscopic part that depends only on the density $\rho^\varepsilon$ and a microscopic part depending on the perturbation $g^\varepsilon$. In particular, we will show that the microscopic contribution to the transport vanishes as $\varepsilon\rightarrow0$. The decomposition reads:
\begin{equation}\label{P+Q}
    T^{\varepsilon,n}_{\ipd,j}=P^{\varepsilon,n}_{\ipd,j}+Q^{\varepsilon,n}_{\ipd,j},
\end{equation}
where the macroscopic part $\left(P^{\varepsilon,n}_{\ipd,j}\right)_{ij}$ is given by
\begin{equation}\label{P}
    P^{\varepsilon,n}_{\ipd,j}=v_jM_j(\rho_{i+1}^{\varepsilon,n}-\rho^{\varepsilon,n}_i)\Delta v + E_\ipd\rho^{\varepsilon,n}_\ipd\frac{M_{j+1}-M_{j-1}}{2}\Delta x,
\end{equation} 
and the microscopic part $\left(Q^{\varepsilon,n}_{\ipd,j}\right)_{ij}$ by
\begin{equation}\label{Q}
\begin{aligned}
    Q^{\varepsilon,n}_{\ipd,j}=&\left(v^{+}_j(g_{\ipd,j}^{\varepsilon,n}-g_{\imd,j}^n)+v^{-}_j(g_{i+\frac{3}{2},j}^{\varepsilon,n}-g_{\ipd,j}^{\varepsilon,n})\right)\Delta v\\ 
    &- M_j\left\langle v^{+}(g_{\ipd}^{\varepsilon,n}-g_{\imd}^n)+v^{-}(g_{i+\frac{3}{2}}^{\varepsilon,n}-g_{\ipd}^{\varepsilon,n})\right\rangle_\Delta\Delta v\\
    &+\left(E^{+}_{\ipd}(g_{\ipd,j}^{\varepsilon,n}-g_{\ipd,j-1}^{\varepsilon,n})+E^{-}_{\ipd}(g_{\ipd,j+1}^{\varepsilon,n}-g_{\ipd,j}^{\varepsilon,n})\right)\Delta x.
\end{aligned}
\end{equation}
Under the hypothesis \eqref{HypInduction}, one obtains that $Q^{\varepsilon,n}_{\ipd,j}\underset{\varepsilon\rightarrow0}{\longrightarrow}0$ for all $(i,j)\in\mathcal{I}\times\mathcal{J}$. Let us now use Assumption \ref{H0}: namely $\rho^{\varepsilon,n}_i\underset{\varepsilon\rightarrow0}{\longrightarrow}\rho^n_i$, and one has $$P^{\varepsilon,n}_{\ipd,j}\underset{\varepsilon\rightarrow0}{\longrightarrow}P^{n}_{\ipd,j}=v_jM_j(\rho_{i+1}^{n}-\rho^{n}_i)\Delta v + E_\ipd\rho^{n}_\ipd\frac{M_{j+1}-M_{j-1}}{2}\Delta x.$$
Then, as $\left(P^{n}_{\ipd,j}\right)_{ij}$ depends neither on $g$ nor $\varepsilon$,
\begin{equation}\label{LimitTtmp}
    T^{\varepsilon,n}_{\ipd,j}\underset{\varepsilon\rightarrow0}{\longrightarrow}P^n_{\ipd,j},\quad\forall (i,j)\in\mathcal{I}\times\mathcal{J}.
\end{equation}
Now, the asymptotic limit of \eqref{S_micro} can be computed. Since $\varepsilon(1-e^{-\Delta t/ \varepsilon^2})\underset{\varepsilon\rightarrow0}{\longrightarrow}0$ and $e^{-\Delta t/ \varepsilon^2}\underset{\varepsilon\rightarrow0}{\longrightarrow}0$ we can use \eqref{LimitTtmp} to obtain:
\begin{equation}\label{LimitG}
    g_{\ipd,j}^{\varepsilon,n+1}\underset{\varepsilon\rightarrow0}{\longrightarrow}0,\quad\forall n,\quad\forall (i,j)\in\mathcal{I}\times\mathcal{J}.
\end{equation}
As a consequence, one also has that for all $n$,
\begin{equation}\label{LimitT}
    T^{\varepsilon,n}_{\ipd,j}\underset{\varepsilon\rightarrow0}{\longrightarrow}P^n_{\ipd,j},\quad\forall (i,j)\in\mathcal{I}\times\mathcal{J}.
\end{equation}
The next step is to plug \eqref{S_micro} into \eqref{S_macro}. Using definition \eqref{MacroFluxCont} of the macro flux, one obtains:
\begin{equation}\label{rhoManFlux}
\begin{aligned}
    \rho_i^{\varepsilon,n+1} = \rho_i^{\varepsilon,n} - \frac{\Delta t}{\varepsilon\Delta x}\bigg\langle v\bigg[&e^{-\Delta t/ \varepsilon^2}(g_{\ipd}^{\varepsilon,n}-g_{\imd}^{\varepsilon,n})- \varepsilon\frac{(1-e^{-\Delta t/ \varepsilon^2})}{\Delta x\Delta v}(T^{\varepsilon,n}_{\ipd}-T^{\varepsilon,n}_{\imd})\bigg]\bigg\rangle_\Delta.
\end{aligned}
\end{equation}
We can then take the limit $\varepsilon\longrightarrow0$ in \eqref{rhoManFlux} using the previous asymptotic limits \eqref{LimitG} and \eqref{LimitT}. We then replace $P^n_{\ipd}$ by its expression:
\begin{equation}
\begin{aligned}
    \rho_i^{n+1}&= \rho_i^n+\frac{\Delta t}{\Delta x^2\Delta v}\left\langle v\left(P^n_{\ipd}-P^n_{\imd}\right)\right\rangle_\Delta\\
    &=\rho_i^n+\langle v^2M\rangle\frac{\Delta t}{\Delta x^2}(\rho^n_{i+1}-2\rho^n_{i}+\rho^n_{i-1})+\frac{\Delta t}{\Delta x}( E_\ipd\rho^n_\ipd-E_\imd\rho^n_\imd)\left\langle v\frac{M_{\cdot+1}-M_{\cdot-1}}{2\Delta v}\right\rangle_\Delta.
\end{aligned}
\end{equation}
Finally, using the definition of the discrete moments \eqref{discMoments} and \eqref{m1Prime}, we obtain the asymptotic scheme \eqref{S_lim}:
\begin{equation*}
\begin{aligned}
        \rho_i^{n+1} &= \rho_i^{n} + \frac{\Delta t}{\Delta x^2}m_2^{\Delta v}\big(\rho^n_{i+1}-2\rho^n_{i}+\rho^n_{i-1}\big) + \frac{\Delta t}{\Delta x}m_1'^{\Delta v}\big(E_\ipd\rho^n_\ipd - E_\imd\rho^n_\imd\big)\\
        &=\rho_i^{n} + \frac{\Delta t}{\Delta x}\left(J_\ipd^{F,n}-J_\imd^{F,n}\right).
\end{aligned}
\end{equation*}
\end{proof}
\noindent In order to show that the scheme \eqref{S_micro}-\eqref{S_macro} is truly AP, one also needs that the stability condition is independent or at least does not degenerate as $\varepsilon\rightarrow0$. While we do not prove the stability of the scheme, in practice, we can indeed use the same time-step for both large and small values of $\varepsilon$.
\section{Dynamic coupling}\label{sec:DynCoup}
The aim of this section is to introduce a coupling method between kinetic and fluid schemes. The objective is to obtain a coupled solver that is faster than a full kinetic one to solve \eqref{prob:Peps} while still being accurate. These methods come naturally when designing accurate numerical codes while guarantying reasonable computation times.~In the field of semiconductors equations, the idea to use the asymptotic limit of the kinetic model to achieve a domain decomposition can be found in \cite{Klar1998}. Such methods are also heavily dependent on the treatment of interface conditions between subdomains.

\noindent Following \cite{FilbetRey2015} we first construct a hybrid kinetic/fluid solver with a dynamic domain decomposition method and present its implementation. In the second part, we are interested in understanding the conservative aspect of the method. More precisely, we give a lemma on the mass variation induced by the coupling.

\subsection{Coupling criteria}
The idea of the dynamic domain decomposition method is twofold. First, the subdomains must accurately describe the state of the solution. In particular, the fluid model is only valid where the solution is near the local equilibrium in velocity. Secondly, we want the method to be dynamic in the sense that the subdomains are adapted at each time step. For this purpose, let us introduce $\Omega^n_\mathcal{K}$ the kinetic domain and $\Omega^n_\mathcal{F}$ the fluid one at time $t^n$. To determine in which domain each cell lies, we introduce criteria based on the higher order fluid model introduced in Section \ref{sec:HigherOrder} and the norm of the perturbation $g^\varepsilon=f^\varepsilon-\rho^\varepsilon\M$. Indeed, when $g^\varepsilon$ is close to $0$, it means that the solution is close to the local equilibrium.

\paragraph{Macroscopic criterion.}Let us consider a fluid subdomain. In this subdomain, one only has access to the macroscopic quantity $\rho$ and the given electrical field $E$.~Therefore, one cannot consider the perturbation $g^\varepsilon$. The solution we propose is to use the higher order model \eqref{DDH} and derive a macroscopic criterion. We have formally shown that \eqref{DDH} can be written in the form: $$\dt\rho^\varepsilon-\dx J^\varepsilon =\mathcal{R}^\varepsilon,$$ where $\mathcal{R}^\varepsilon$ is a remainder that depends only on the density $\rho^\varepsilon$ and the electrical field $E$. During the coupling procedure it will be computed using the kinetic density $\rho^\varepsilon$ in kinetic cells and using the fluid density $\rho$ in fluid cells. It is given by 
\begin{equation}\label{eq:Remainder}
    \mathcal{R^\varepsilon}=-\varepsilon^2\dx(2\dx(EJ^\varepsilon) - E\dx J^\varepsilon -\dxx J^\varepsilon), \quad\mbox{ where } J=\dx\rho-E\rho^\varepsilon.
\end{equation} 
Expanding $\mathcal{R}$ shows that one needs derivatives of $\rho$ up to fourth order and of $E$ up to third order:
\begin{equation}
\begin{aligned}
        \mathcal{R^\varepsilon} = -\varepsilon^2\big(&-\dxxxx\rho^\varepsilon\\
        &+E(2\dxxx\rho^\varepsilon-E\dxx\rho^\varepsilon)\\
        &+\dx E(-3\rho^\varepsilon\dx E - 5E\dx\rho^\varepsilon+6\dxx\rho^\varepsilon)\\
        &+\dxx E(-3\rho^\varepsilon E+5\dx\rho^\varepsilon)\\
        &+\rho^\varepsilon\dxxx E\big).
\end{aligned}
\end{equation}
Let us denote by $\mathcal{R}_i^{\varepsilon,n}$ a discretization of the remainder $\mathcal{R^\varepsilon}$. High order finite difference schemes are used (See Table \ref{tab:finiteDiff}).

\begin{table}
    \centering
    \begin{tabular}{|c|c|c|c|c|c|c|c|}
        \hline
        \diagbox{Derivative}{Index} & -3 & -2 & -1 & 0 & 1 & 2 & 3\\
        \hline
        1 & -1/60 & 3/20 & -3/4 & 0 & 3/4 & -3/20 & 1/60\\
        \hline
        2 & 1/90 & -3/20 & 3/2 & -49/18 & 3/2 & -3/20 & 1/90\\
        \hline
        3 & 1/8 & -1 & 13/8 & 0 & -13/8 & 1 & -1/8\\
        \hline
        4 & -1/6 & 2 & -13/2 & 28/3 & -13/2 & 2 & -1/6\\
        \hline
    \end{tabular}
    \caption{Central finite differences coefficient.}
    \label{tab:finiteDiff}
\end{table}

\noindent Let $\eta_0,\,\delta_0>0$ be the coupling thresholds. In a fluid domain, when $\mathcal{R}^n$ is large, the model \eqref{DDH} is far from the limit model \eqref{DD} and one must use the kinetic one instead. More specifically, consider a fluid cell $\mathcal{X}_i\subset\Omega^n_\mathcal{F}$.
\begin{itemize}
\item If $\left|\mathcal{R}^{n}_i\right|\leq\eta_0$, then the cell stays fluid at $t^{n+1}$.
\item If $\left|\mathcal{R}^{n}_i\right|>\eta_0$, then the cell becomes kinetic at $t^{n+1}$:
$$\mathcal{X}_i\not\subset\Omega^{n+1}_\mathcal{F}\quad\mbox{ and }\quad \mathcal{X}_i\subset\Omega^{n+1}_\mathcal{K}.$$
\end{itemize}
In a kinetic subdomain, unlike the previous case, one has access to the perturbation $g^\varepsilon$. When this perturbation is small, it means that the solution is near a local equilibrium with respect to the velocity variable. As a consequence, the model is close to the fluid one and one can use the limit model. Moreover, we also use the criterion that the remainder $\mathcal{R^\varepsilon}$ must be small. Consider now a kinetic cell $\mathcal{X}_i\subset\Omega^n_\mathcal{K}$:
\begin{itemize} 
\item If $\,||g^{\varepsilon,n}_{\imd}||_{\gamma}>\delta_0$ and $||g^{\varepsilon,n}_{\ipd}||_{\gamma}>\delta_0$ then the cell stays kinetic at $t^{n+1}$.
\item If $\,||g^{\varepsilon,n}_{\imd}||_{\gamma}\leq\delta_0$, $\,||g^{\varepsilon,n}_{\ipd}||_{\gamma}\leq\delta_0$ and $\left|\mathcal{R}^{\varepsilon,n}_i\right|>\eta_0$, then the cell stays kinetic at $t^{n+1}$.
\item If $\,||g^{\varepsilon,n}_{\imd}||_{\gamma}\leq\delta_0$, $\,||g^{\varepsilon,n}_{\ipd}||_{\gamma}\leq\delta_0$ and $\left|\mathcal{R}^{\varepsilon,n}_i\right|\leq\eta_0$, then the cell becomes fluid at $t^{n+1}$: $$\mathcal{X}_i\not\subset\Omega^{n+1}_\mathcal{K}\quad\mbox{ and }\quad \mathcal{X}_i\subset\Omega^{n+1}_\mathcal{F}.$$
\end{itemize}
The discrete norm $||g^{\varepsilon,n}_{\imd}||_{\gamma}$ is a discrete version of the norm on $L^2(\dd\gamma)$. It is defined for $\left(g^{\varepsilon,n}_{\ipd}\right)_{j\in\mathcal{J}}$ by:
\begin{equation}\label{NormG}
    ||g^{\varepsilon,n}_{\imd}||_{\gamma}=\sum_{j\in\mathcal{J}} \left(g^{\varepsilon,n}_{\imd}\right)^2M_j^{-1}\Delta v.
\end{equation}
\begin{Remark}
Note that in a kinetic cell, the criterion on the norm of $g^\varepsilon$ is mandatory. Indeed, the remainder $\mathcal{R}^n_i$ could be small, but the perturbation large. In this situation, one does not want to change from kinetic to fluid. As an example, one could take a distribution function constant in position and far from the Maxwellian in velocity.
\end{Remark}

\subsection{Implementation}\label{sec:Implementation}
We now present in more details the implementation of the coupling method. 

\noindent An important part of this approach is the management of boundary conditions. When solving on the whole space domain, periodic boundary conditions are applied. However, when solving in the subdomains $\Omega^n_\mathcal{K}$ and $\Omega^n_\mathcal{F}$, we need to adapt our solver. Our strategy is to use ghost cell values that are chosen appropriately. The difficulty lies in the fact that the limit scheme only computes the density $\rho$ and not the pair $(\rho^\varepsilon,g^\varepsilon)$. As the coupling method is dynamic, one does not know in advance the state of the cells. As a consequence, one must be able to access all unknowns on the whole domain at any time. Our solution is to take advantage of the structure of the micro-macro scheme. Indeed, aside from visualisation and diagnostics, an explicit discretization of the distribution function isn't necessary. We are working only with $\rho^{\varepsilon,n}$ and $g^{\varepsilon,n}$. As a consequence, we have access to the macro unknown on the whole domain and there is no information missing in the arrays. The distribution $f^\varepsilon$ is reconstructed using $f^{\varepsilon,n}_{i,j}=\rho^{\varepsilon,n}_i\M_j+\frac{1}{2}\left(g^{\varepsilon,n}_{\imd,j}+g^{\varepsilon,n}_{\ipd,j}\right)$ only for posttreatment. However, the kinetic solver may still needs values of $g^\varepsilon$ on the whole space domain. Therefore, the array storing $g^\varepsilon$ must be filled in the fluid domain. In practice, to improve the performance, $g^\varepsilon$ is not updated in fluid regions and it is set to $0$ only when needed. In particular, it occurs when a fluid cell becomes kinetic. 

\noindent Another important remark is that since $g^\varepsilon$ is approximated on the dual mesh, one must be careful at the interfaces between kinetic and fluid subdomains. To avoid any ambiguity on the state of an interface when updating the perturbation $g^\varepsilon$, we impose that a fluid subdomain is at least two cells wide. Under this condition the state of the ghost interface is well determined. See Figure \ref{fig:TransMM} for an illustration of such a situation.

\begin{figure}
    \centering
    \definecolor{ttttff}{rgb}{0.2,0.2,1.}
\definecolor{ffqqqq}{rgb}{1.,0.,0.}
\begin{tikzpicture}[line cap=round,line join=round,>=triangle 45,x=1.0cm,y=1.0cm]
\clip(-3.0,-3.5) rectangle (6.0,-1.0);
\draw [line width=1.pt] (-2.5,-2.5)-- (5.5,-2.5);
\draw [line width=.5pt,dash pattern=on 3pt off 3pt] (1.5,-1.0)-- (1.5,-3.2);
\draw [line width=1.6pt,dotted,color=ffqqqq] (-2.4091828215273132,-1.4999223418025338)-- (1.3917579060311982,-1.5071147697370142);
\draw [line width=1.6pt,dotted,color=ttttff] (1.6036544246558697,-1.4999223418025338)-- (5.419871583052794,-1.4999223418025338);
\draw (-1.1,-0.9) node[anchor=north west] {$Kinetic$};
\draw (3.,-0.9) node[anchor=north west] {$Fluid$};
\draw (-2.0,-2.75) node[anchor=north west] {$\mathcal{X}_{s-2}$};
\draw (0.0,-2.75) node[anchor=north west] {$\mathcal{X}_{s-1}$};
\draw (2.3,-2.75) node[anchor=north west] {$\mathcal{X}_{s}$};
\draw (4.1,-2.75) node[anchor=north west] {$\mathcal{X}_{s+1}$};
\draw (-2.0,-1.77) node[anchor=north west] {$\rho^{\varepsilon,n}_{s-2}$};
\draw (0.0,-1.77) node[anchor=north west] {$\rho^{\varepsilon,n}_{s-1}$};
\draw (1.0,-1.77) node[anchor=north west] {$g^{\varepsilon,n}_{s-\frac{1}{2}}$};
\draw (-1.0,-1.77) node[anchor=north west] {$g^{\varepsilon,n}_{s-\frac{3}{2}}$};
\draw (-3.0,-1.77) node[anchor=north west] {$g^{\varepsilon,n}_{s-\frac{5}{2}}$};
\draw (2.3,-1.77) node[anchor=north west] {$\rho^{n}_{s}$};
\draw (4.1,-1.77) node[anchor=north west] {$\rho^{n}_{s+1}$};
\begin{scriptsize}
\draw [fill=ffqqqq] (-2.5,-2.5) ++(-2.0pt,0 pt) -- ++(2.0pt,2.0pt)--++(2.0pt,-2.0pt)--++(-2.0pt,-2.0pt)--++(-2.0pt,2.0pt);
\draw [fill=ttttff] (5.5,-2.5) ++(-2.0pt,0 pt) -- ++(2.0pt,2.0pt)--++(2.0pt,-2.0pt)--++(-2.0pt,-2.0pt)--++(-2.0pt,2.0pt);
\draw [fill=ffqqqq] (-0.5,-2.5) ++(-2.0pt,0 pt) -- ++(2.0pt,2.0pt)--++(2.0pt,-2.0pt)--++(-2.0pt,-2.0pt)--++(-2.0pt,2.0pt);
\draw [fill=ttttff] (3.5,-2.5) ++(-2.0pt,0 pt) -- ++(2.0pt,2.0pt)--++(2.0pt,-2.0pt)--++(-2.0pt,-2.0pt)--++(-2.0pt,2.0pt);
\draw [fill=ffqqqq] (1.5,-2.5) ++(-2.0pt,0 pt) -- ++(2.0pt,2.0pt)--++(2.0pt,-2.0pt)--++(-2.0pt,-2.0pt)--++(-2.0pt,2.0pt);
\draw [fill=ffqqqq,shift={(-2.4091828215273132,-1.4999223418025338)},rotate=90] (0,0) ++(0 pt,3.75pt) -- ++(3.2475952641916446pt,-5.625pt)--++(-6.495190528383289pt,0 pt) -- ++(3.2475952641916446pt,5.625pt);
\draw [fill=ffqqqq,shift={(1.3917579060311982,-1.5071147697370142)},rotate=270] (0,0) ++(0 pt,3.75pt) -- ++(3.2475952641916446pt,-5.625pt)--++(-6.495190528383289pt,0 pt) -- ++(3.2475952641916446pt,5.625pt);
\draw [fill=ttttff,shift={(1.6036544246558697,-1.4999223418025338)},rotate=90] (0,0) ++(0 pt,3.75pt) -- ++(3.2475952641916446pt,-5.625pt)--++(-6.495190528383289pt,0 pt) -- ++(3.2475952641916446pt,5.625pt);
\draw [fill=ttttff,shift={(5.419871583052794,-1.4999223418025338)},rotate=270] (0,0) ++(0 pt,3.75pt) -- ++(3.2475952641916446pt,-5.625pt)--++(-6.495190528383289pt,0 pt) -- ++(3.2475952641916446pt,5.625pt);
\end{scriptsize}
\end{tikzpicture}
    \caption{Transition between kinetic and fluid cell for the micro-macro scheme.}
    \label{fig:TransMM}
\end{figure}
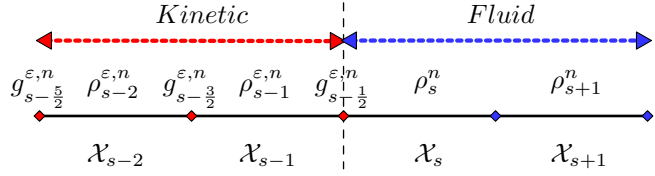

\noindent The algorithm can be summarized as follows:
\begin{algorithm}[H]
\caption{Hybrid scheme}\label{alg:Hybrid}
\begin{enumerate}
    \item Set $\varepsilon$, $\delta_0$, $\eta_0$ and a final time $T$.
    \item Initialize micro-macro unknowns using the relations $\rho^0=\langle f^0\rangle$ and $g^0=f^0-\rho^0M$.
    \item Initialize $\Omega^0_\mathcal{K}$ as the whole space ($\Omega^0_\mathcal{F}=$\O).
    \item Compute $g^{\varepsilon,n+1}$ and $\rho^{\varepsilon,n+1}$ in $\Omega^n_\mathcal{K}$ using the kinetic scheme \eqref{S_micro}-\eqref{S_macro}.
    \item Compute $\rho^{n+1}$ in $\Omega^n_\mathcal{F}$ using the limit scheme \eqref{S_lim}.
    \item Set $g^{n+1}=0$ in $\Omega^n_\mathcal{F}$.
    \item Update $\Omega^n_\mathcal{K}$ and $\Omega^n_\mathcal{F}$ to $\Omega^{n+1}_\mathcal{K}$ and $\Omega^{n+1}_\mathcal{F}$ using the criteria presented above.
    \item Increment time and repeat until $t^{n+1}=T$.
\end{enumerate}
\end{algorithm}
\noindent In particular, Algorithm \ref{alg:Hybrid} explicitly defines a numerical scheme on the hybrid density $\Tilde{\rho}$: 
\begin{equation}\label{HybridRho}
    \Tilde{\rho}_i^{n+1} = \Tilde{\rho}_i^{n} + \frac{\Delta t}{\Delta x}J_i^{H,n},
\end{equation}
where 
\begin{equation}\label{HybridFlux}
    J_i^{H,n} = \left\lbrace\begin{aligned}
        \left(J_\ipd^{K,n}-J_\imd^{K,n}\right)\mbox{ if }\quad\mathcal{X}_i\in\Omega^n_\mathcal{K},\\
        \left(J_\ipd^{F,n}-J_\imd^{F,n}\right)\mbox{ if }\quad\mathcal{X}_i\in\Omega^n_\mathcal{F}.\\
    \end{aligned}\right.
\end{equation}
Note that we want to start the resolution with the approach containing the full information on the system. Hence, it makes sense to initialize our domain as fully kinetic.

\subsection{Mass conservation}
This section is dedicated to investigate the mass conservation of the hybrid method. This property being satisfied in the continuous case, one expects conservation in the discrete setting. Each of the standard schemes is conservative on its own by construction. However, the question arises when considering the hybrid scheme.

\noindent In order to understand the loss of mass, we consider a toy model that isn't relevant in practice but will highlight the key elements to constrain the mass variation. Let us set the state of cells for every time step two domains:
\begin{equation}\label{DomainMassVar}
    \Omega_\mathcal{K}=\bigcup\limits_{i=1}^{s-1} \mathcal{X}_{i} \quad\mbox{ and }\quad \Omega_\mathcal{F}=\bigcup\limits_{i=s}^{Nx} \mathcal{X}_{i}.
\end{equation} 
Note that in the next result, we neglect what happens at the boundary. Our primary focus is to understand what happens at the interface $x_{s-\frac{1}{2}}$ between the two domains. Moreover, in that context and with periodic boundary conditions in space, the same analysis can be done at the interface $x_{\frac{1}{2}}$. Figure \ref{fig:ConstDomainMass} illustrates this framework. The following lemma quantifies the mass variation between two time steps.

\begin{figure}
    \centering
    \definecolor{ududff}{rgb}{0.30196078431372547,0.30196078431372547,1.}
\definecolor{ffqqqq}{rgb}{1.,0.,0.}
\begin{tikzpicture}[line cap=round,line join=round,>=triangle 45,x=2.6064236014351794cm,y=1.7238977153352264cm]
\clip(-0.47,0.832) rectangle (5.28,3.15);
\draw [line width=1.pt,color=ffqqqq] (0.,2.)-- (1.,2.);
\draw [line width=1.pt,color=ffqqqq] (1.,2.)-- (2.,2.);
\draw [line width=1.pt,color=ududff] (3.,2.)-- (4.,2.);
\draw [line width=1.pt,color=ududff] (4.,2.)-- (5.,2.);
\draw [line width=1.pt,dotted,color=ffqqqq] (0.,2.6)-- (2.4,2.6);
\draw [line width=1.pt,dotted,color=ududff] (2.6,2.6)-- (5.003103753290778,2.5922740582605805);
\draw [line width=1.pt,color=ffqqqq] (0.,1.5)-- (1.,1.5);
\draw [line width=1.pt,color=ffqqqq] (1.,1.5)-- (2.,1.5);
\draw [line width=1.pt,color=ududff] (3.,1.5)-- (4.,1.5);
\draw [line width=1.pt,color=ududff] (4.,1.5)-- (5.,1.5);
\draw [line width=1.pt,color=ffqqqq] (2.,1.5)-- (2.5,1.5);
\draw [line width=1.pt,color=ududff] (3.,1.5)-- (2.5,1.5);
\draw [line width=1.pt,color=ffqqqq] (2.,2.)-- (2.5,2.);
\draw [line width=1.pt,color=ududff] (2.5,2.)-- (3.,2.);
\draw (-0.4,1.65) node[anchor=north west] {$t^{n-1}$};
\draw (-0.4,2.16) node[anchor=north west] {$t^{n}$};
\draw (0.90,2.9) node[anchor=north west] {Kinetic};
\draw (3.60,2.9) node[anchor=north west] {Fluid};
\draw (2.50,1.36) node[anchor=north west] {$x_{s-\frac{1}{2}}$};
\draw (3.40,1.36) node[anchor=north west] {$x_{s+\frac{1}{2}}$};
\draw (4.48,1.36) node[anchor=north west] {$x_{s+\frac{3}{2}}$};
\draw (1.46,1.36) node[anchor=north west] {$x_{s-\frac{3}{2}}$};
\draw (0.40,1.36) node[anchor=north west] {$x_{s-\frac{5}{2}}$};
\draw (1.93,1.4) node[anchor=north west] {$x_{s-1}$};
\draw (3.90,1.4) node[anchor=north west] {$x_{s+1}$};
\draw (0.89,1.4) node[anchor=north west] {$x_{s-2}$};
\draw (-0.11,1.4) node[anchor=north west] {$x_{s-3}$};
\draw (2.92,1.4) node[anchor=north west] {$x_{s}$};
\draw (4.92,1.4) node[anchor=north west] {$x_{s+2}$};
\draw (2.32,2.5) node[anchor=north west, color=ffqqqq] {$J^{K,n}_{s-\frac{1}{2}}$};
\draw (1.35,2.5) node[anchor=north west, color=ffqqqq] {$J^{K,n}_{s-\frac{3}{2}}$};
\draw (0.35,2.5) node[anchor=north west, color=ffqqqq] {$J^{K,n}_{s-\frac{5}{2}}$};
\draw (2.32,2.00) node[anchor=north west, color=ududff] {$J^{F,n}_{s-\frac{1}{2}}$};
\draw (3.32,2.00) node[anchor=north west, color=ududff] {$J^{F,n}_{s+\frac{1}{2}}$};
\draw (4.32,2.00) node[anchor=north west, color=ududff] {$J^{F,n}_{s+\frac{3}{2}}$};
\begin{scriptsize}
\draw [fill=black] (0.,2.) circle (2.0pt);
\draw [fill=black] (1.,2.) circle (2.0pt);
\draw [fill=black] (2.,2.) circle (2.0pt);
\draw [fill=black] (3.,2.) circle (2.0pt);
\draw [fill=black] (4.,2.) circle (2.0pt);
\draw [fill=black] (5.,2.) circle (2.0pt);
\draw [fill=ffqqqq,shift={(0.,2.6)},rotate=90] (0,0) ++(0 pt,2.25pt) -- ++(1.9485571585149868pt,-3.375pt)--++(-3.8971143170299736pt,0 pt) -- ++(1.9485571585149868pt,3.375pt);
\draw [fill=ffqqqq,shift={(2.4,2.6)},rotate=270] (0,0) ++(0 pt,2.25pt) -- ++(1.9485571585149868pt,-3.375pt)--++(-3.8971143170299736pt,0 pt) -- ++(1.9485571585149868pt,3.375pt);
\draw [fill=ududff,shift={(2.6,2.6)},rotate=90] (0,0) ++(0 pt,2.25pt) -- ++(1.9485571585149868pt,-3.375pt)--++(-3.8971143170299736pt,0 pt) -- ++(1.9485571585149868pt,3.375pt);
\draw [fill=ududff,shift={(5.003103753290778,2.5922740582605805)},rotate=270] (0,0) ++(0 pt,2.25pt) -- ++(1.9485571585149868pt,-3.375pt)--++(-3.8971143170299736pt,0 pt) -- ++(1.9485571585149868pt,3.375pt);
\draw [fill=black] (0.,1.5) circle (2.0pt);
\draw [fill=black] (1.,1.5) circle (2.0pt);
\draw [fill=black] (2.,1.5) circle (2.0pt);
\draw [fill=black] (3.,1.5) circle (2.0pt);
\draw [fill=black] (4.,1.5) circle (2.0pt);
\draw [fill=black] (5.,1.5) circle (2.0pt);
\draw [fill=black] (1.5,1.5) ++(-2.0pt,0 pt) -- ++(2.0pt,2.0pt)--++(2.0pt,-2.0pt)--++(-2.0pt,-2.0pt)--++(-2.0pt,2.0pt);
\draw [fill=black] (1.5,2.) ++(-2.0pt,0 pt) -- ++(2.0pt,2.0pt)--++(2.0pt,-2.0pt)--++(-2.0pt,-2.0pt)--++(-2.0pt,2.0pt);
\draw [fill=black] (0.5,2.) ++(-2.0pt,0 pt) -- ++(2.0pt,2.0pt)--++(2.0pt,-2.0pt)--++(-2.0pt,-2.0pt)--++(-2.0pt,2.0pt);
\draw [fill=black] (0.5,1.5) ++(-2.0pt,0 pt) -- ++(2.0pt,2.0pt)--++(2.0pt,-2.0pt)--++(-2.0pt,-2.0pt)--++(-2.0pt,2.0pt);
\draw [fill=black] (3.5,2.) ++(-2.0pt,0 pt) -- ++(2.0pt,2.0pt)--++(2.0pt,-2.0pt)--++(-2.0pt,-2.0pt)--++(-2.0pt,2.0pt);
\draw [fill=black] (3.5,1.5) ++(-2.0pt,0 pt) -- ++(2.0pt,2.0pt)--++(2.0pt,-2.0pt)--++(-2.0pt,-2.0pt)--++(-2.0pt,2.0pt);
\draw [fill=black] (4.5,2.) ++(-2.0pt,0 pt) -- ++(2.0pt,2.0pt)--++(2.0pt,-2.0pt)--++(-2.0pt,-2.0pt)--++(-2.0pt,2.0pt);
\draw [fill=black] (4.5,1.5) ++(-2.0pt,0 pt) -- ++(2.0pt,2.0pt)--++(2.0pt,-2.0pt)--++(-2.0pt,-2.0pt)--++(-2.0pt,2.0pt);
\draw [fill=black] (2.5,1.5) ++(-2.0pt,0 pt) -- ++(2.0pt,2.0pt)--++(2.0pt,-2.0pt)--++(-2.0pt,-2.0pt)--++(-2.0pt,2.0pt);
\draw [fill=black] (2.5,2.) ++(-2.0pt,0 pt) -- ++(2.0pt,2.0pt)--++(2.0pt,-2.0pt)--++(-2.0pt,-2.0pt)--++(-2.0pt,2.0pt);
\end{scriptsize}
\end{tikzpicture}
    \caption{Zoom on the interface of a steady domain decomposition.}
    \label{fig:ConstDomainMass}
\end{figure}
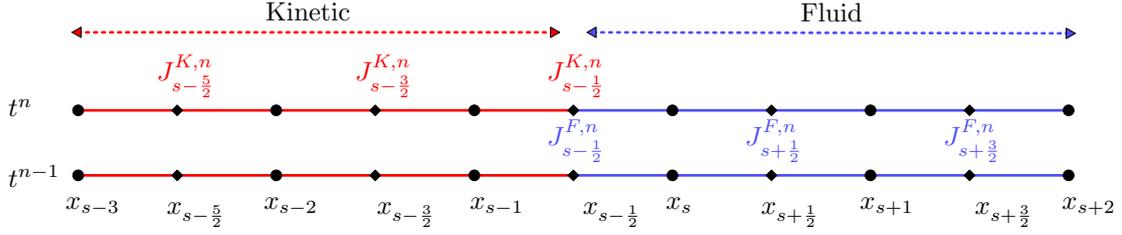

\begin{Lemma}\label{LemmaMass}
Let $\left(\Tilde{\rho}_i^n\right)_{i}$ and $\left(g_{i+\frac{1}{2},j}^{\varepsilon,n}\right)_{ij}$ be computed using the hybrid scheme \eqref{S_micro}-\eqref{HybridRho}. Let the mass variation between $t^n$ and $t^{n+1}$ be defined as:
\begin{equation}\label{Def:Deltam}
    \Delta m^{n+\frac{1}{2}}=\sum_{i\in\mathcal{I}}\Delta x\frac{(\Tilde{\rho}_i^{n+1}-\Tilde{\rho}_i^{n})}{\Delta t}.
\end{equation} 
In the context of the steady domain decomposition \eqref{DomainMassVar} and neglecting the boundaries, one has:
\begin{equation}\label{DeltaM}
    \Delta m^{n+\frac{1}{2}} = -\left\langle vg_{s-\frac{1}{2}}^n\right\rangle_\Delta\frac{e^{-\Delta t/\varepsilon^2}}{\varepsilon} + \frac{1-e^{-\Delta t/\varepsilon^2}}{\Delta x\Delta v}\left\langle Q_{s-\frac{1}{2}}^{\varepsilon,n}\right\rangle_\Delta - e^{-\Delta t/\varepsilon^2}J_{s-\frac{1}{2}}^{F,n}+\left(J_{s-\frac{1}{2}}^{\varepsilon,F,n} - J_{s-\frac{1}{2}}^{F,n}\right),
\end{equation}
where $Q_{s-\frac{1}{2},j}^{\varepsilon,n}$ is given by \eqref{Q}, $J_{s-\frac{1}{2}}^{F,n}$ by \eqref{LimFlux} and $J_{s-\frac{1}{2}}^{\varepsilon,F,n}$ is the limit flux \eqref{LimFlux} computed with the kinetic density $\rho^\varepsilon$.
\end{Lemma}
\begin{proof}
Let us consider the hybrid scheme \eqref{HybridRho}-\eqref{HybridFlux} on the density. Using the fixed domain decomposition \eqref{DomainMassVar} and neglecting the boundary, the mass variation writes:
\begin{equation}\label{DeltaM1}
\begin{aligned}
    \Delta m^{n+\frac{1}{2}}&=\sum_{i\in\mathcal{I}}\frac{\Delta x}{\Delta t}(\Tilde{\rho}_i^{n+1}-\Tilde{\rho}_i^{n})\\
    &=\sum_{i\in\mathcal{I}}J_i^{H,n}\\
    &=\sum\limits_{i=1}^{s-1}\left(J_\ipd^{K,n}-J_\imd^{K,n}\right) + \sum\limits_{i=s}^{N_x}\left(J_\ipd^{F,n}-J_\imd^{F,n}\right)\\
    &= J_{s-\frac{1}{2}}^{K,n} - J_{s-\frac{1}{2}}^{F,n}\\
    &= -\frac{1}{\varepsilon}\left\langle vg^{\varepsilon,n+1}_{s-\frac{1}{2}}\right\rangle_\Delta - J_{s-\frac{1}{2}}^{F,n}.
\end{aligned}
\end{equation} 
Similarly as in the proof of Proposition \ref{prop:AP}, $g^{\varepsilon,n+1}_{s-\frac{1}{2},j}$ is replaced by its expression \eqref{S_micro} and $T^{\varepsilon,n}_{s-\frac{1}{2},j}$ is expanded using \eqref{P+Q}. The quantity $\frac{1}{\varepsilon}\left\langle vg^{\varepsilon,n+1}_{s-\frac{1}{2}}\right\rangle_\Delta$ then reads:
\begin{equation}\label{DeltaM2}
\begin{aligned}
    \frac{1}{\varepsilon}\left\langle vg^{\varepsilon,n+1}_{s-\frac{1}{2}}\right\rangle_\Delta =\left\langle vg^{\varepsilon,n}_{s-\frac{1}{2}}\right\rangle_\Delta\frac{e^{-\Delta t/\varepsilon^2}}{\varepsilon}
    &- \left\langle vQ_{s-\frac{1}{2}}^{\varepsilon,n}\right\rangle_\Delta\left(1-e^{-\Delta t/\varepsilon^2}\right)\frac{1}{\Delta x\Delta v}\\
    &- \left\langle vP_{s-\frac{1}{2}}^{\varepsilon,n}\right\rangle_\Delta\left(1-e^{-\Delta t/\varepsilon^2}\right)\frac{1}{\Delta x\Delta v}.
\end{aligned}
\end{equation}
Using definition \eqref{P} of $P_{s-\frac{1}{2},j}^{\varepsilon,n}$ and the definition \eqref{m1Prime} of $m_1'^{\Delta v}$, the third term reduces to:
\begin{equation}\label{DeltaM3}
\begin{aligned}
    \left\langle vP_{s-\frac{1}{2}}^{\varepsilon,n}\right\rangle_\Delta\frac{1}{\Delta x\Delta v}&=
    \langle v^2M\rangle_\Delta\frac{\rho_{s}^{\varepsilon,n}-\rho_{s-1}^{\varepsilon,n}}{\Delta x} + E_{s-\frac{1}{2}}\rho_{s-\frac{1}{2}}^{\varepsilon,n}\left\langle v\frac{M_{\cdot+1}-M_{\cdot-1}}{2\Delta v}\right\rangle_\Delta\\
    &= m_2^{\Delta v}\frac{\rho_{s}^{\varepsilon,n}-\rho_{s-1}^{\varepsilon,n}}{\Delta x} + m_1'^{\Delta v} E_{s-\frac{1}{2}}\rho_{s-\frac{1}{2}}^{\varepsilon,n}\\
    &= J_{s-\frac{1}{2}}^{\varepsilon,F,n}
\end{aligned}
\end{equation}
Finally, plugging \eqref{DeltaM3} and \eqref{DeltaM2} into \eqref{DeltaM1} yields:
\begin{equation*}
    \Delta m^{n+\frac{1}{2}} = -\left\langle vg_{s-\frac{1}{2}}^n\right\rangle_\Delta\frac{e^{-\Delta t/\varepsilon^2}}{\varepsilon} + \frac{1-e^{-\Delta t/\varepsilon^2}}{\Delta x\Delta v}\left\langle Q_{s-\frac{1}{2}}^n\right\rangle_\Delta - e^{-\Delta t/\varepsilon^2}J_{s-\frac{1}{2}}^{\varepsilon,F,n}+\left(J_{s-\frac{1}{2}}^{\varepsilon,F,n}-J_{s-\frac{1}{2}}^{F,n}\right).
\end{equation*}
\end{proof}
\begin{Remark}
    The proof only holds in the context of the toy problem \eqref{DomainMassVar}. However, it can be extended to a more general setting seeing that the mass variation occurs at all interfaces between kinetic and fluid subdomains. Namely, $$\Delta t\Delta m^{n+\frac{1}{2}}=\sum_{\alpha\in\mathcal{S}}\beta\left(-\left\langle vg_{\alpha}^n\right\rangle_\Delta\frac{e^{-\Delta t/\varepsilon^2}}{\varepsilon} + \frac{1-e^{-\Delta t/\varepsilon^2}}{\Delta x\Delta v}\left\langle Q_{\alpha}^n\right\rangle_\Delta - e^{-\Delta t/\varepsilon^2}J_{\alpha}^{F,n}+\left(J_{\alpha}^{\varepsilon,F,n}-J_{\alpha}^{F,n}\right)\right),$$
    where $\mathcal{S}$ is the set of interfaces between kinetic and fluid subdomains and $\beta=\pm1$ depends on the orientation of the subdomains. Another important observation is that thanks to \eqref{LimitG} and assuming that $\rho^\varepsilon\underset{\varepsilon\rightarrow0}{\longrightarrow}\rho$, the mass variation converges to $0$ as $\varepsilon$ tends to $0$. 
\end{Remark}
\section{Numerical simulations}\label{sec:NumSim}
In the following, unless specified otherwise, the phase-space is discretized as follows:
$$N_v = 256,\quad N_x = 100,\quad v_\star =8,\quad x_\star = \pi,\quad \Delta t=10^{-4}.$$
The same time step is used for all schemes. Note that since the limit scheme is explicit, its stability is therefore guaranteed under a parabolic condition: $\Delta t\leq C\Delta x^2$. Let us assume that the electrical field is the gradient of a potential V: $E=-\dx V$. To satisfy the periodicity of the domain, we choose $V(x)=-\frac{\sin(2x)}{4}$ so $E(x)~=\frac{1}{2}\cos(2x).$ 
We also set 
\begin{equation}\label{InitDataLocEqui}
f_0^{1}=\frac{1}{\sqrt{2\pi}}e^{-v^2/2}(1+\cos(2x)),
\end{equation} an initial data at local equilibrium in velocity and 
\begin{equation}\label{InitDataFarEqui}
f_0^{2}=\frac{4}{\sqrt{2\pi}}v^4e^{-v^2/2}(1+\cos(2x)),
\end{equation} an initial data far from the local equilibrium in velocity.
Finally, we consider four configurations:
\begin{itemize}
    \item Case 1: $E = 0$, with initial data \eqref{InitDataLocEqui};
    \item Case 2: $E \neq0$, with initial data \eqref{InitDataLocEqui};
    \item Case 3: $E = 0$, with initial data \eqref{InitDataFarEqui};
    \item Case 4: $E \neq0$, with initial data \eqref{InitDataFarEqui}.
\end{itemize}
Each configuration is tested for different values of $\varepsilon$.

\subsection{The full kinetic scheme}
\paragraph{Convergence towards the drift-diffusion equation.} 
Let us first numerically investigate the AP property of the \eqref{Micro}-\eqref{Macro} scheme. We consider this analysis for the Cases 1 and 2. The results can be found in Figure \ref{fig:APE=0} for Case 1 and in Figure \ref{fig:APEneq0} for Case 2. We can observe a convergence of the kinetic scheme to the limit one as $\varepsilon\rightarrow0$. In particular, the curves for $\varepsilon=0.05$ and $10^{-4}$ overlap and are close to the limit case. This validates the asymptotic consistency of the \eqref{Micro}-\eqref{Macro} scheme. The stability is numerically verified as the same $\Delta t$ is used for every $\varepsilon$.

\begin{figure}
    \centering
    \includegraphics[width=0.8\linewidth]{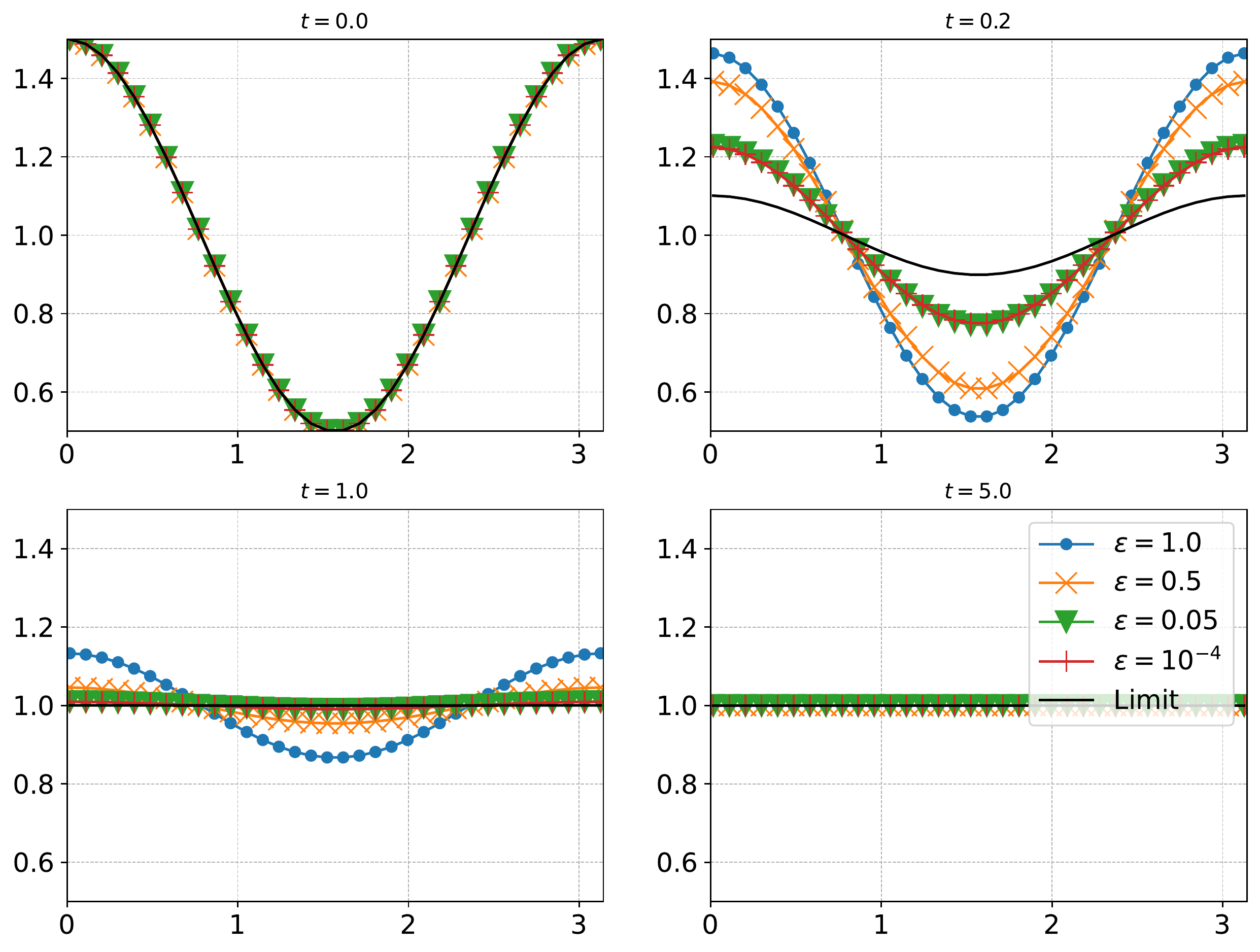}
    \caption{Case 1. Comparison of the solution of the limit scheme \eqref{S_lim} with the solution obtained with the \eqref{Micro}-\eqref{Macro} scheme with different $\varepsilon$, $t= 0.0,\,0.2,\,1.0$ and $5.0$.}
    \label{fig:APE=0}
\end{figure}

\begin{figure}
    \centering
    \includegraphics[width=0.8\linewidth]{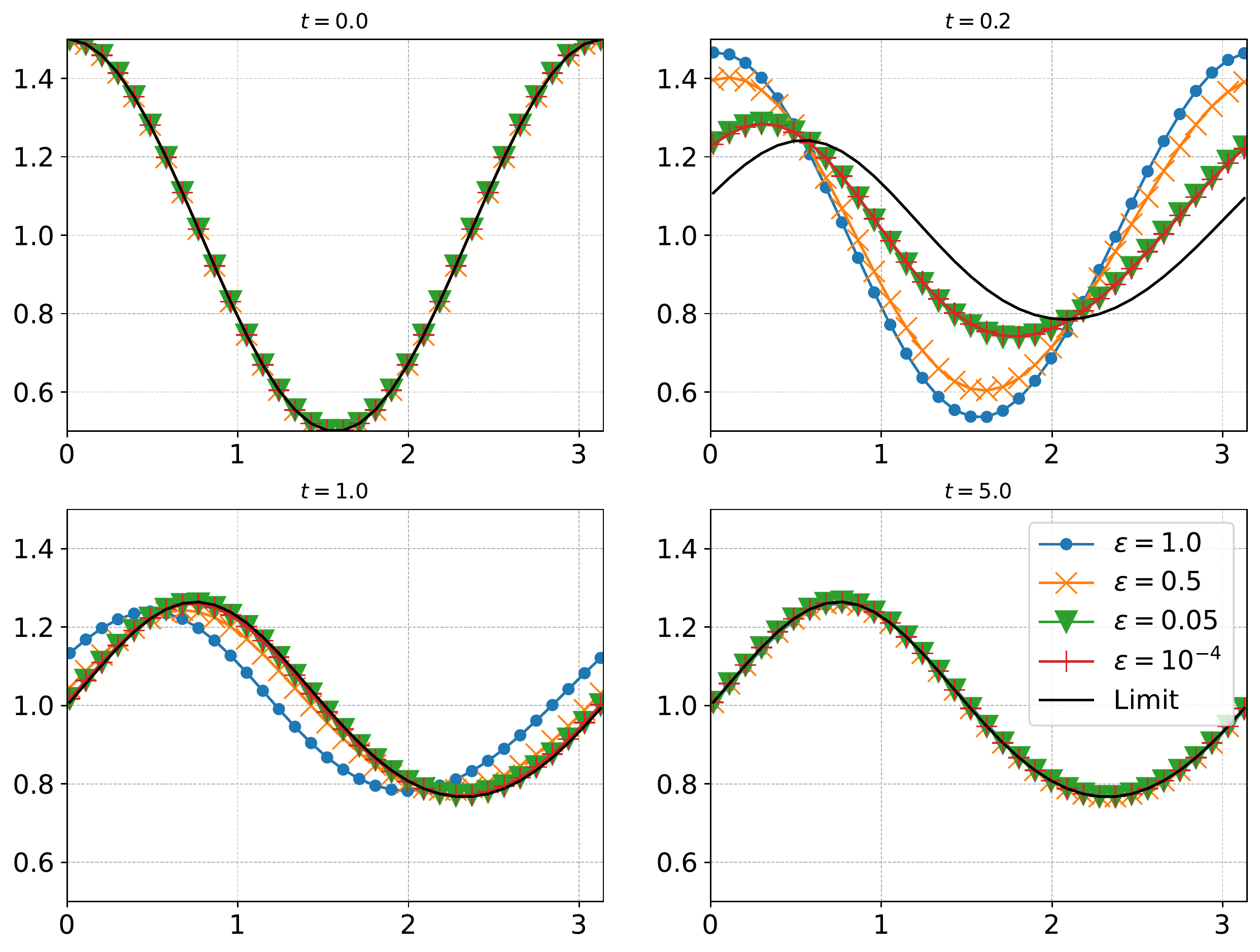}
    \caption{Case 2. Comparison of the solution of the limit scheme \eqref{S_lim} with the solution obtained with the \eqref{Micro}-\eqref{Macro} scheme with different $\varepsilon$, $t= 0.0,\,0.2,\,1.0$ and $5.0$.}
    \label{fig:APEneq0}
\end{figure}

\paragraph{Long time behaviour.}
The long time behaviour of solutions to \eqref{prob:Peps} have been extensively studied in the past decade. Let us recall the following quantities, namely the local equilibria in velocity and space and the global equilibrium:
$$\M(v)= \frac{e^{-|v^2|/2}}{(2\pi)^{d/2}},\quad\phi(x)= \frac{e^{-V(x)}}{\int_{\Omega_x} e^{-V(x)}\,\dd x},\quad F(x,v)= M_0\phi(x)\mathcal{M}(v).$$
When one considers models such as \eqref{prob:Peps}, there are various ways to show that there exists $\kappa(\varepsilon)>0$ and $C(\varepsilon)>0$, such that if $f^\varepsilon$ is solution to \eqref{prob:Peps},
\begin{equation}\label{eq:hyporate}
    ||f^\varepsilon(t)-F||_\mathcal{V}\leq C(\varepsilon)||f_0-F||_\mathcal{V}e^{-\kappa(\varepsilon) t},
\end{equation}
where $\mathcal{V}$ is an appropriate functional space. A proof of \eqref{eq:hyporate} was done in \cite{HerauNier2004} in a setting without electric field. In recent years, the literature on the subject expanded a lot. Robust and systematic methods were developed to show the convergence to an equilibrium. Those are called hypocoercivity methods. A general abstract framework for the $H^1$ norm has been given in the memoir by Villani \cite{VillaniHypo2009}. These methods often have the drawback of requiring regularity on the initial data. However the techniques have been adapted to consider only weighted $L^2$ initial data. More recently, an $L^2$-hypocoercivity method has been developed in \cite{DolbeaultMouhotSchmeiser2015} for linear kinetic equations. We also refer to \cite{AddalaDolbeaultLiTayeb2019} where $L^2$-hypocoercivity is shown for a more general kinetic equation. Both a self-consistent potential given by the Poisson equation and an exterior potential are considered and such a model is closer to the physics of semiconductors. From a numerical point of view, recovering such long-time behaviour at the discrete level is a significant property to obtain. In recent papers, hypocoercivity methods were adapted to the discrete setting using finite differences \cite{DujardinHerauLaffitte2018}, finite elements \cite{Georgoulis2018} and finite volumes \cite{BessemoulinHerdaRey2019}. 

\noindent Following these ideas, we want to observe the convergence of the \eqref{S_micro}-\eqref{S_macro} scheme to equilibrium in a large time scale. Figure \ref{fig:SnapshotDist} shows the evolution of the distribution as time increases (Case 3, $\varepsilon=1.0$). In particular, the numerical solution indeed seems to converge to equilibrium. 
\begin{figure}
    \centering
    \includegraphics[width=1.0\linewidth]{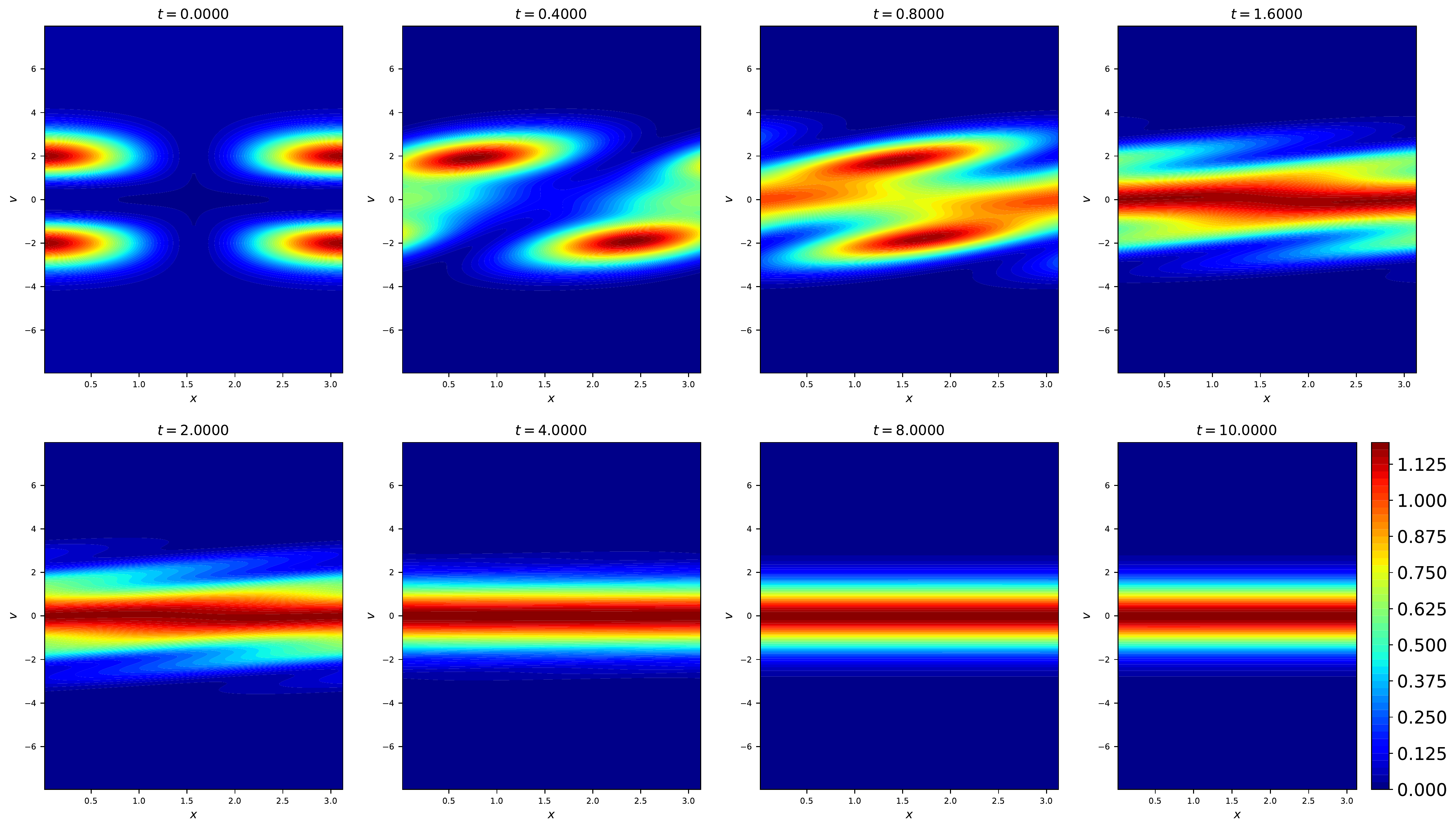}
    \caption{Case 3. Snapshots of the distribution function computed with the scheme \eqref{S_micro}-\eqref{S_macro}, $\varepsilon=1.0$.}
    \label{fig:SnapshotDist}
\end{figure}
Let us introduce the following discrete norm for $f=\left(f_{ij}\right)_{ij}$:
\begin{equation}\label{DiscreteNormGamma}
    ||f||_{\Delta}=\sqrt{\sum_{(i,j)\in\mathcal{I}\times\mathcal{J}} f_{ij}^2\frac{\Delta x\Delta v}{F_{ij}}},
\end{equation}
where $\left(F_{ij}\right)_{ij}$ is a discretization of the global equilibrium $F$, $F_{ij}=F(x_i,v_j)$. For $\left(f_i\right)_{i\in\mathcal{I}}$, we also denote by $||f||_2=\sum_{i\in\mathcal{I}}f_i^2\Delta x$ the discrete $L^2$-norm in position. We now investigate the rate of convergence of the following discrete norms: 
\begin{equation}\label{Norms}
    ||f-F||_{\Delta},\quad ||g||_{\Delta},\quad ||\rho^\varepsilon-\langle F\rangle_\Delta||_2\quad\text{and}\quad ||\rho-\langle F\rangle_\Delta||_2,
\end{equation}
where $\rho^\varepsilon$ is the solution obtained with the kinetic scheme and $\rho$ is obtained with the limit scheme. We consider Case 2. On Figure \ref{fig:HypoE=0} we choose $\varepsilon=1.0$ and $0.1$, and show the norms \eqref{Norms} as functions of time in semilog scale. The exponential convergence of the various norms is clear. Moreover, the rates $r_\varepsilon$ observed are $r_{1}=-2.07$ and $r_{0.1}=-7.65$. The rate $r_\varepsilon$ increases as the Knudsen number gets smaller. In particular we observe the same rate of convergence between the fully kinetic scheme and the limit one for small values of $\varepsilon$.

\begin{figure}
    \centering
    \begin{tabular}{ll}
    \includegraphics[width=0.45\linewidth]{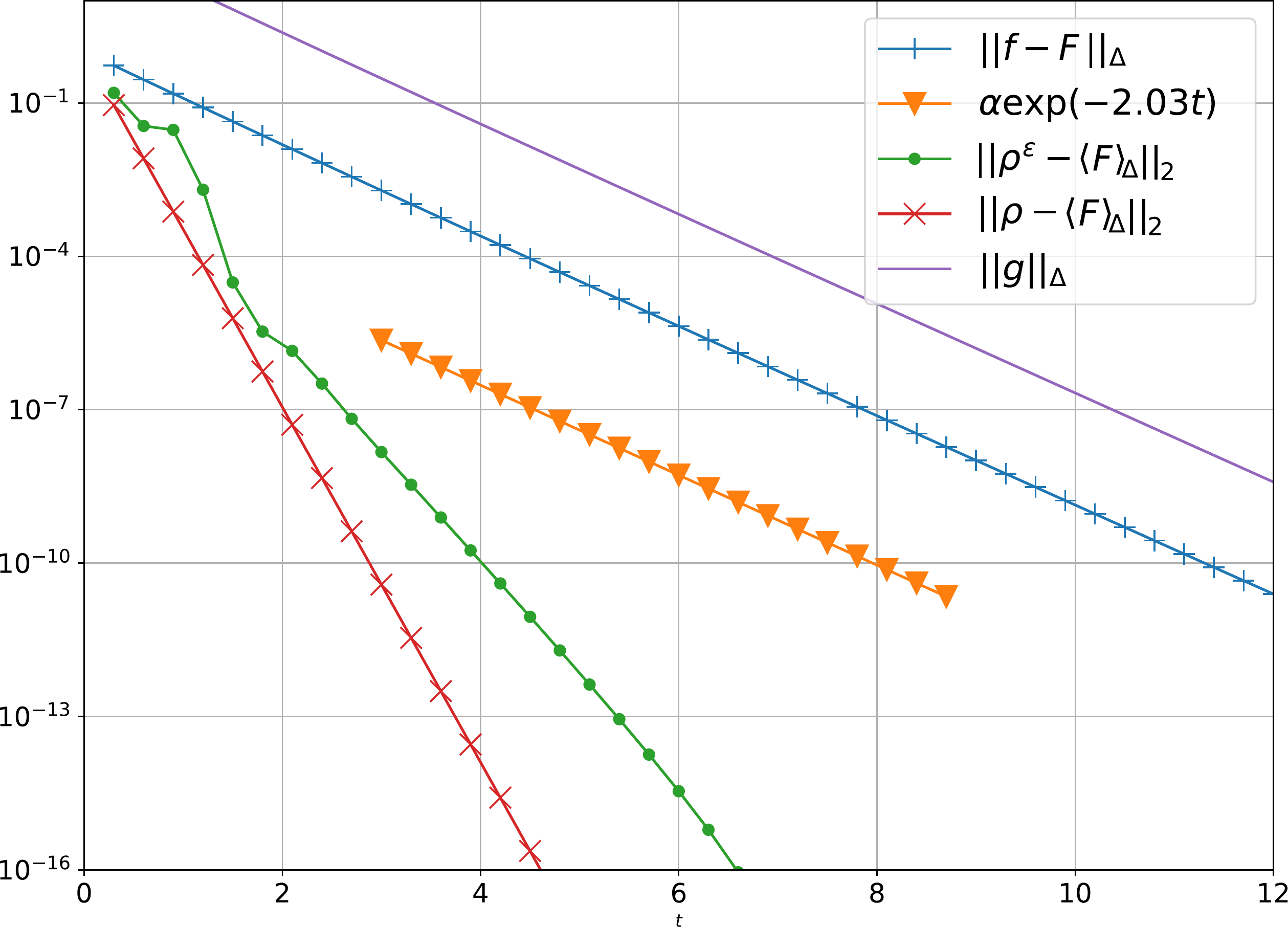}
    &
    \includegraphics[width=0.45\linewidth]{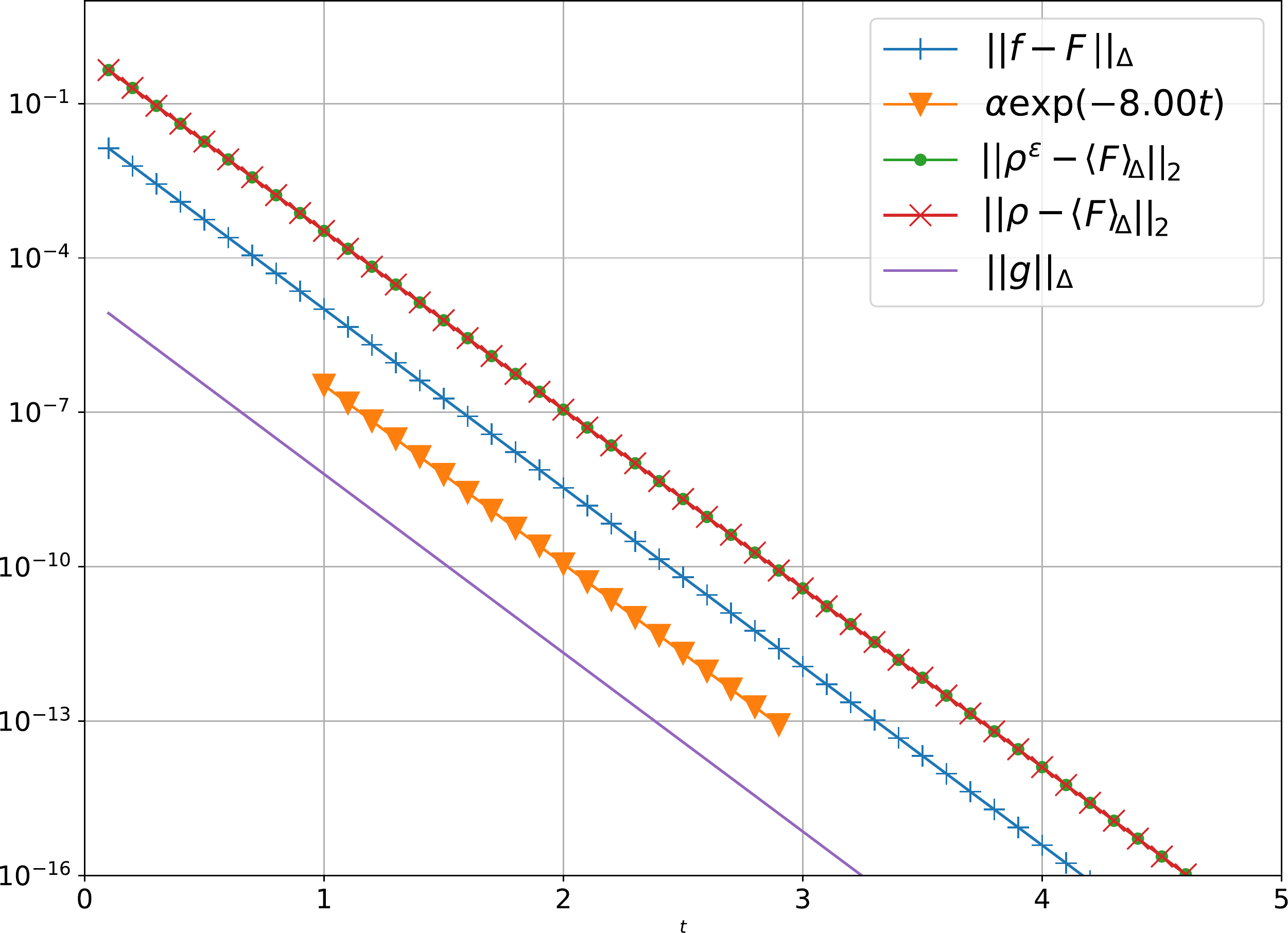}
    \end{tabular}
    \caption{Case 3. Time evolution of the norms \eqref{Norms} computed with the fully kinetic scheme and limit scheme, $\varepsilon = 1.0$ (left), $\varepsilon = 0.1$ (right).}
    \label{fig:HypoE=0}
\end{figure}

Let us point out that in the case of a nonzero electric field, we do not recover the same convergence to equilibrium. Indeed, our numerical scheme is not well-balanced, i.e. designed to preserve steady states. As a consequence, the numerical solution only converges to an equilibrium that is an approximation of the steady state. Figure \ref{fig:HypoEfield} shows the convergence to the equilibrium as the number of cells in space increases.

\begin{figure}
    \centering
    \includegraphics[width=0.5\linewidth]{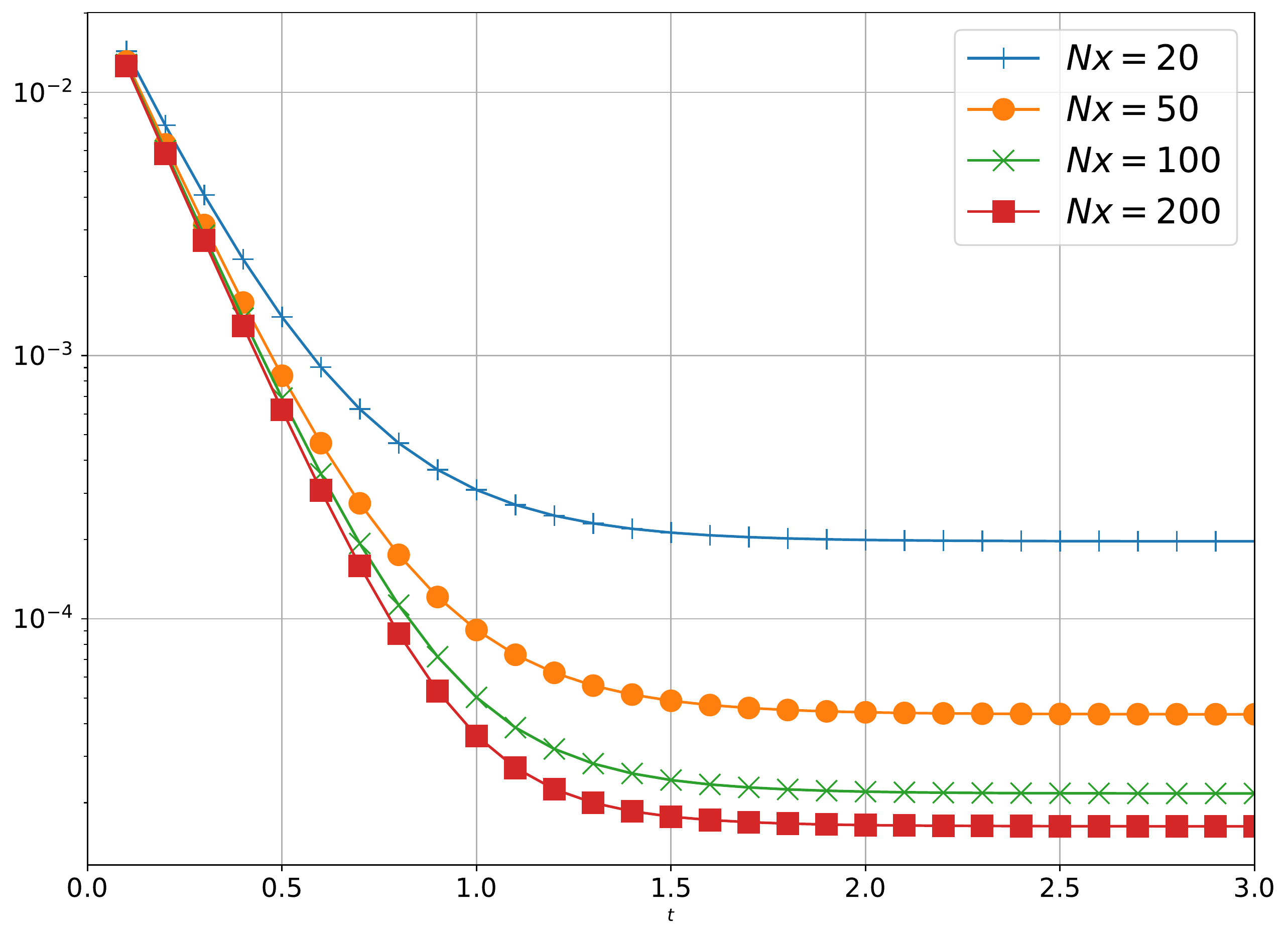}
    \caption{Case 4. Time evolution of the norm $||f^\varepsilon-F\,||_{\Delta}$ computed with the fully kinetic scheme for $N_x=20$, $50$, $100$ and $200$, $\varepsilon = 0.1$.}
    \label{fig:HypoEfield}
\end{figure}

\subsection{Properties of the hybrid scheme}
\paragraph{Choice of the coupling parameters.} 
Before investigating the properties of the hybrid scheme, a natural question is the choice of the coupling parameters. Indeed, as we have seen earlier, that choice has an impact on the conservation of mass. The smaller the parameters, the more one can control this variation. However, the bigger the parameters are, the faster is the resulting hybrid scheme as one allows more fluid cells to appear. Therefore, one must find a good balance between accuracy and computation time. To illustrate how the macroscopic indicator behaves, we compute it without updating the state of the cells. Figure \ref{fig:Indicators} shows the indicator compared to the difference between the kinetic and fluid densities. One can observe that this indicator behaves as expected. When the kinetic and limit densities are close, the indicator is also small. Regarding the norm of $g^\varepsilon$, its behaviour is also expected. Indeed, we chose an initial data far from the local equilibrium in velocity and therefore, the norm can be high even if the densities are close (See first column, third row in Figure \ref{fig:Indicators}). Lastly, both the macroscopic indicator and the norm of $g^\varepsilon$ tend to 0 as time increases. As a consequence, the closer to the equilibrium the solution is, the more fluid cells will appear.

\begin{figure}
    \centering
    \includegraphics[width=\linewidth]{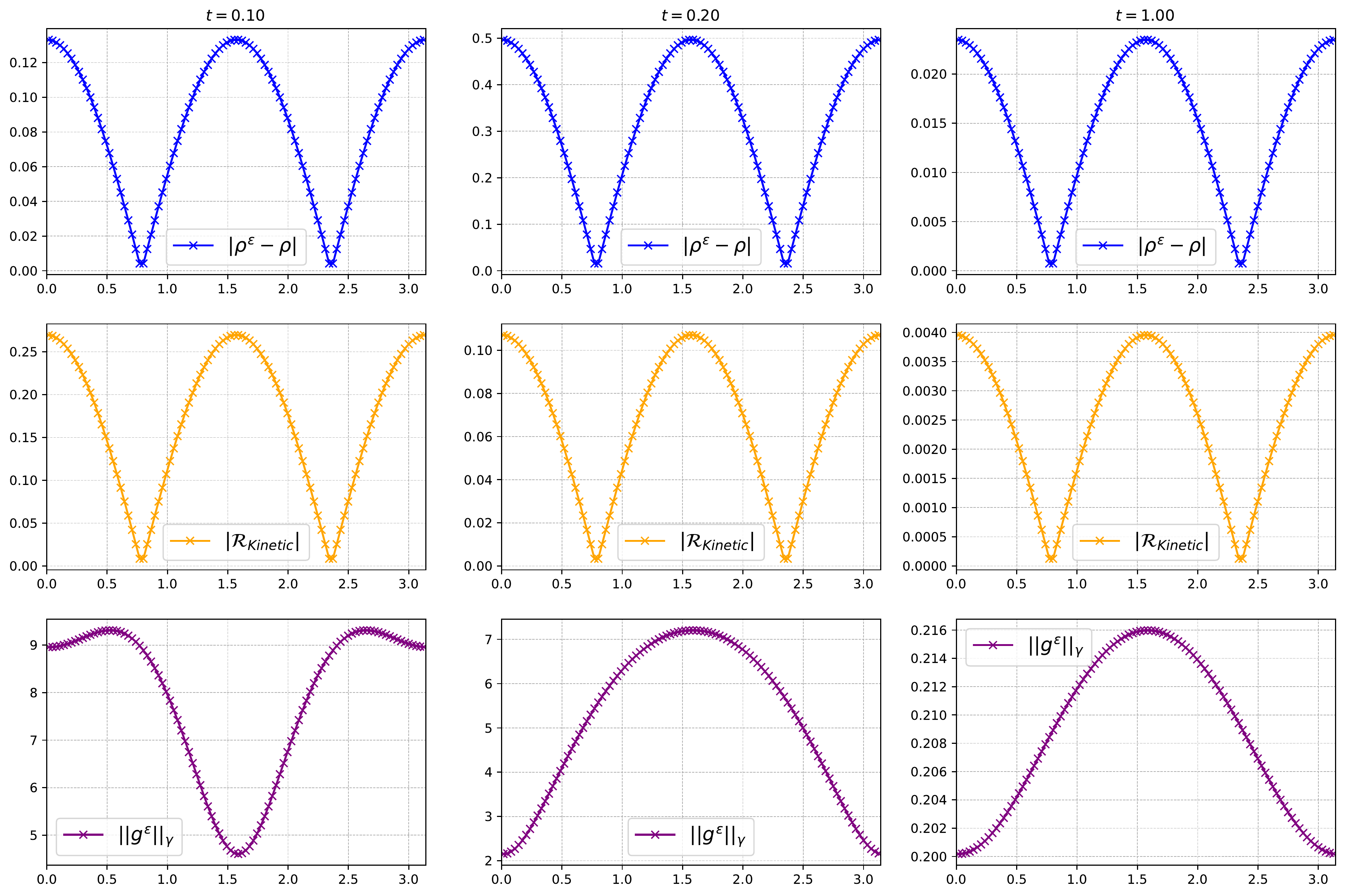}
    \caption{Case 3. Snapshots of the difference between the densities computed with the kinetic and the limit schemes (Top), macroscopic indicator (Middle), $L^2(\mathrm{d\gamma})$ norm of the perturbation $g$ (Bottom), $\varepsilon=0.5$.}
    \label{fig:Indicators}
\end{figure}

\paragraph{Qualitative comparison.}
Let us now compare the kinetic and the hybrid schemes. Figures \ref{fig:HybridCompEfield} and \ref{fig:HybridComp} shows the densities computed by the kinetic, hybrid and limit schemes for cases 3 and 4 with $\varepsilon=10^{-3}$. We can see a good agreement between the three schemes for the smaller Knudsen number. When $\varepsilon$ is large, the solution relaxes slowly toward the local equilibrium and the coupling occurs only for large final time (see Figure \ref{fig:MassVar1}). This behaviour is however expected as the solution relaxes faster to equilibrium as $\varepsilon$ gets smaller.

\begin{figure}
    \centering
    \includegraphics[width=\linewidth]{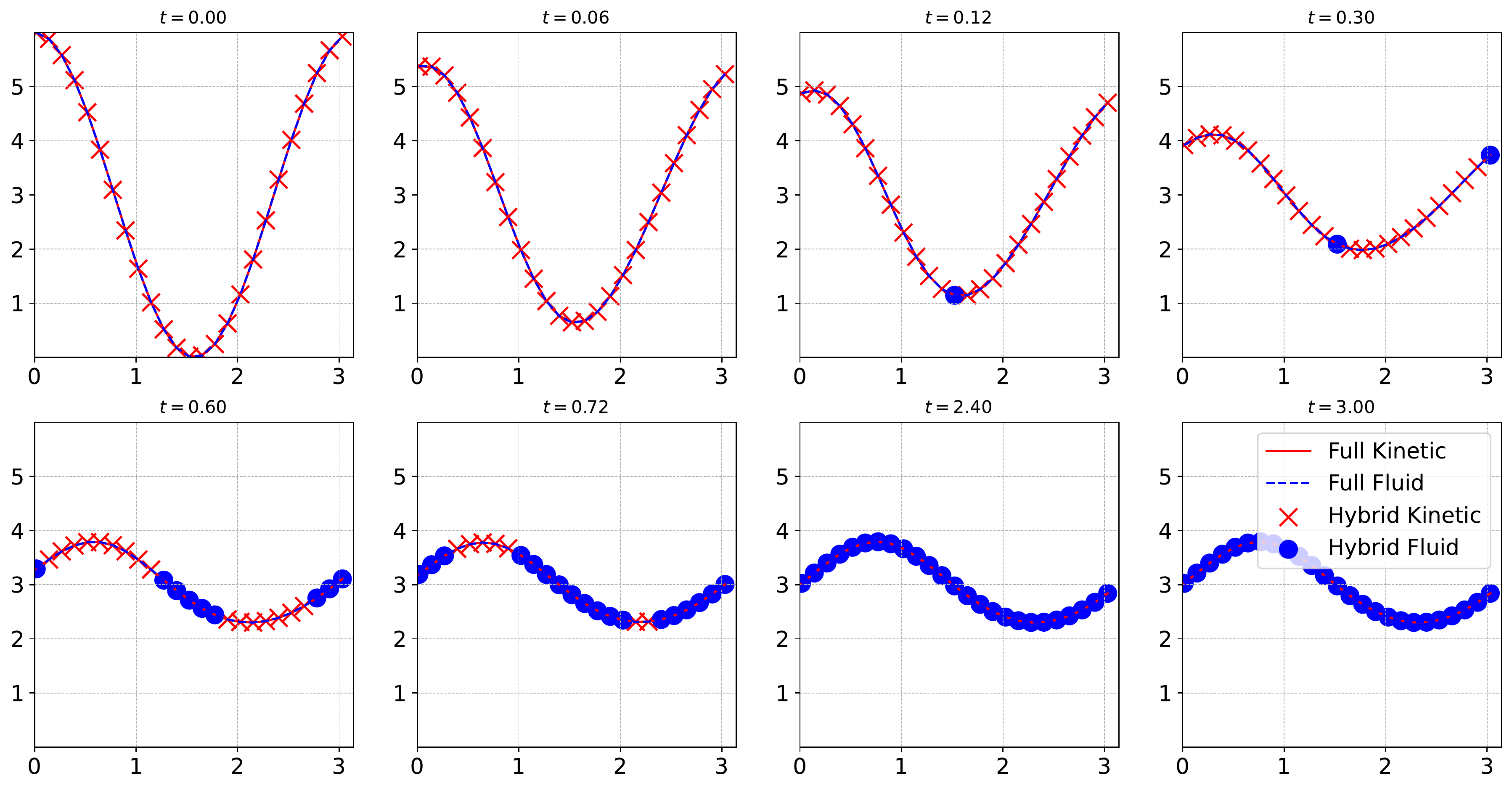}
    \caption{Case 4. Snapshots of the densities computed using the full kinetic, hybrid and limit schemes, $\varepsilon=10^{-3}$, $\eta_0=\delta_0=10^{-4}$.}
    \label{fig:HybridCompEfield}
\end{figure}

\begin{figure}
    \centering
    \includegraphics[width=\linewidth]{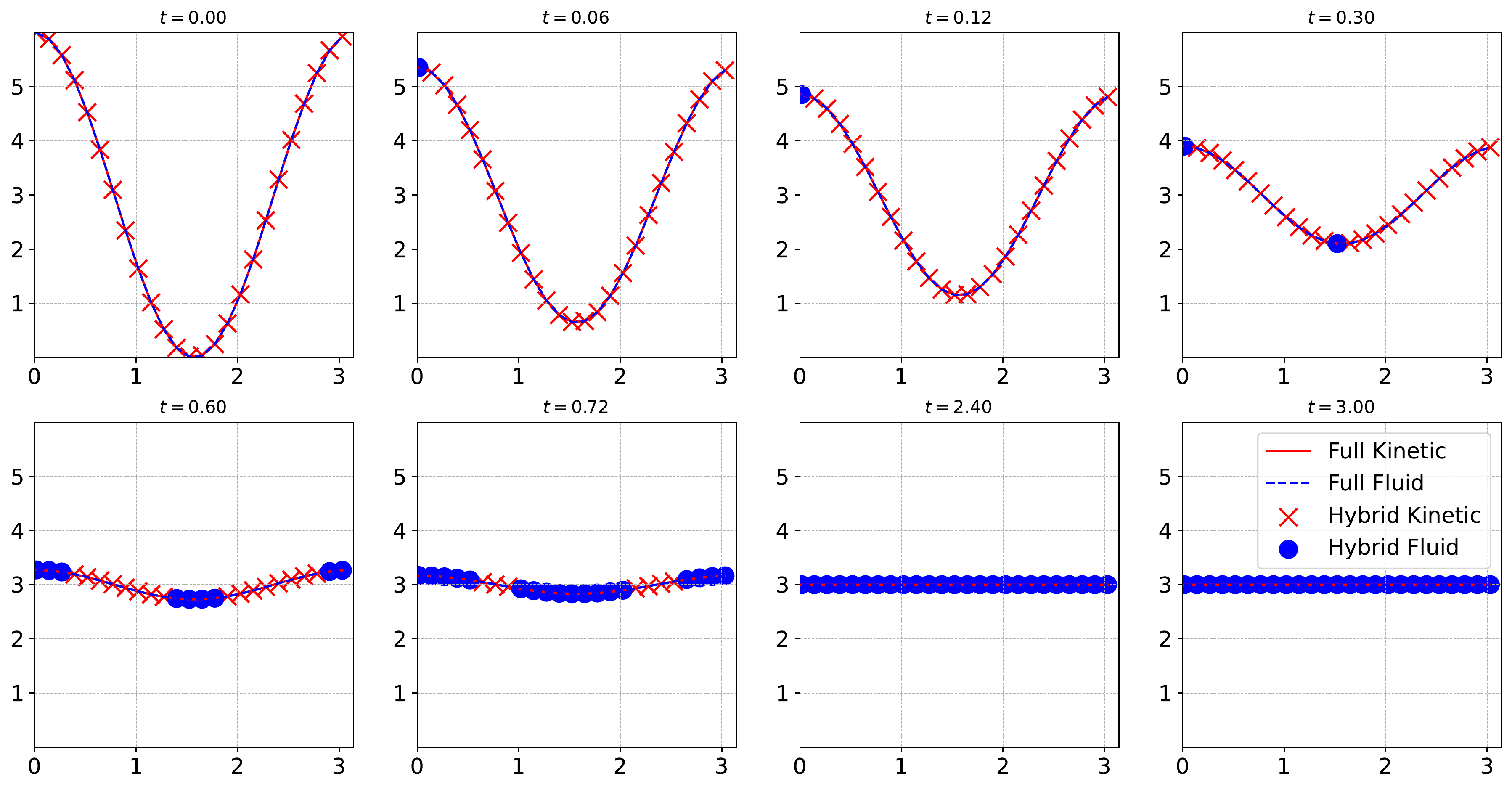}
    \caption{Case 3. Snapshots of the densities computed using the full kinetic, hybrid and limit schemes, $\varepsilon=10^{-3}$, $\eta_0=\delta_0=10^{-4}$.}
    \label{fig:HybridComp}
\end{figure}

\paragraph{Conservation of mass.}
Let us now numerically investigate the conservation of mass. Indeed, we have shown in Lemma \ref{LemmaMass} that the hybrid method is not exactly conservative. However, it becomes conservative asymptotically. In addition, the hybrid method was constructed so that the cells become fluid when the solution is close to a local equilibrium in velocity. As a consequence, the perturbation is small when the coupling occurs and so is the mass variation. In practice, we can observe a mass variation of the order of the machine accuracy. We illustrate the state of the cells and the corresponding mass variation on Figures \ref{fig:MassVar1} and \ref{fig:MassVar2} where we consider Case 3 for $\varepsilon=1.0$ and $10^{-3}$. We also represent the evolution of the state of the cells. In particular, we observe a transition from a full kinetic state to a full fluid one as time increases, which is expected due to the relaxation towards an equilibrium.

\begin{figure}
    \centering
    \includegraphics[width=0.9\linewidth]{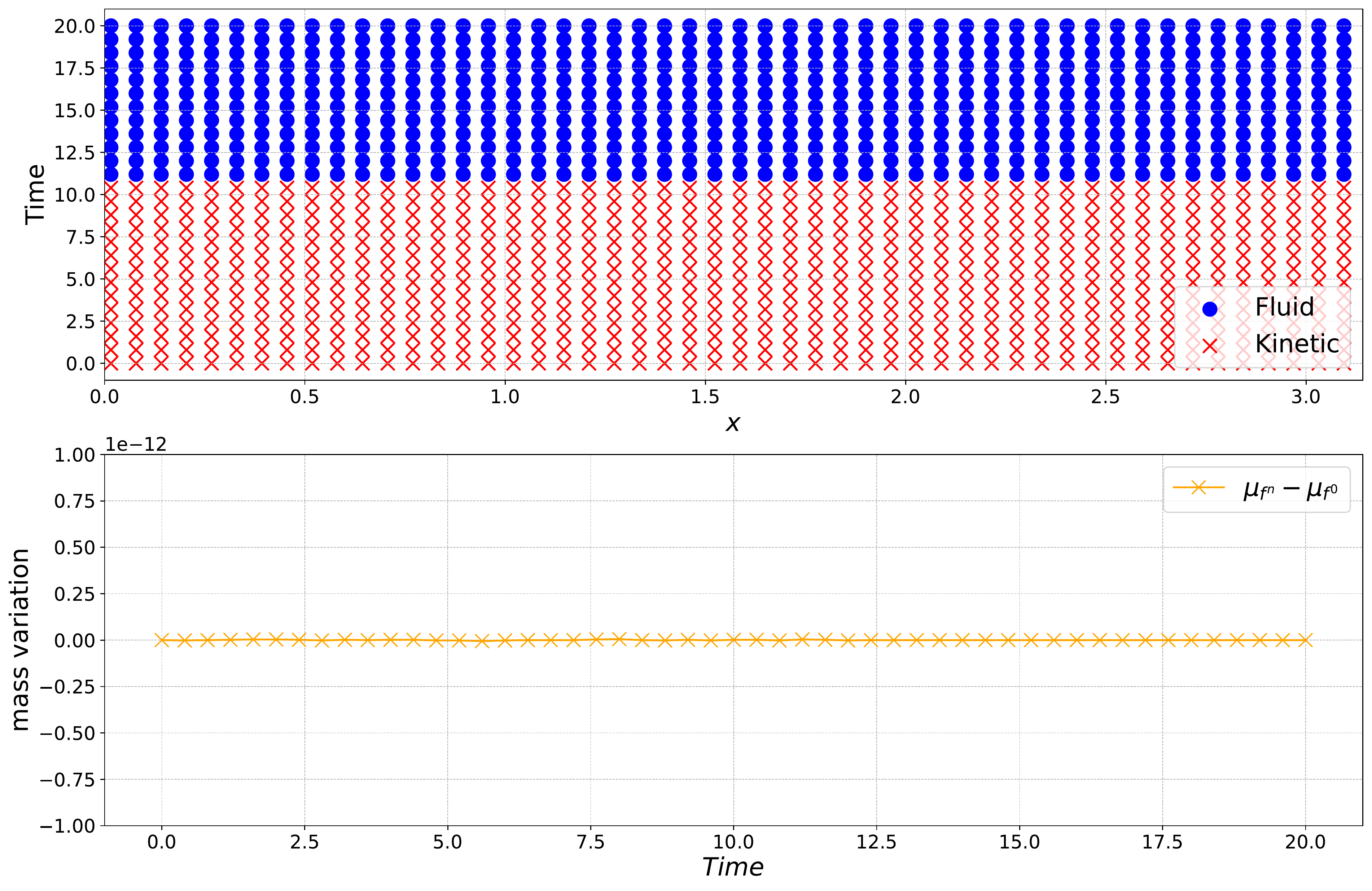}
    \caption{Case 3. Time evolution of the state of the cells (Top) and mass variation (Bottom), $\varepsilon=1.0$, $\eta_0=\delta_0=10^{-4}$.}
    \label{fig:MassVar1}
\end{figure}

\begin{figure}
    \centering
    \includegraphics[width=0.9\linewidth]{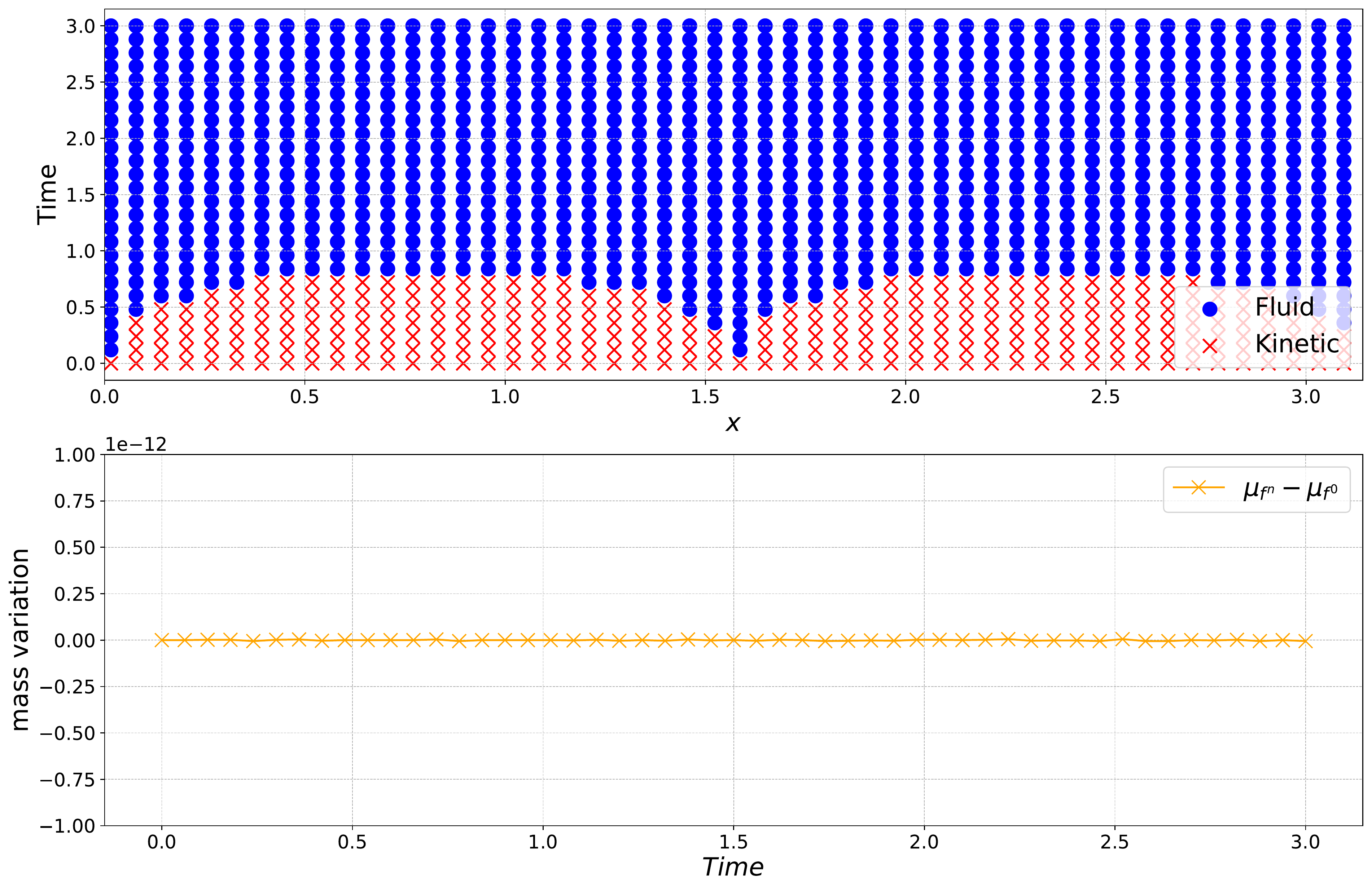}
    \caption{Case 3. Time evolution of the state of the cells (Top) and mass variation (Bottom), $\varepsilon=10^{-3}$, $\eta_0=\delta_0=10^{-4}$.}
    \label{fig:MassVar2}
\end{figure}

\paragraph{Error analysis.}
We are now interested in the error made when using the hybrid method. In particular, we investigate the error between the full kinetic scheme and the hybrid method. The goal of the hybrid method is to be more efficient than the full kinetic solver. However, the gain in computation time comes with a slight loss in accuracy. Let $\rho_{Kinetic}$ be the density computed using the full kinetic scheme and $\rho_{Hybrid}$ be the density obtained from the hybrid scheme. On Figure \ref{fig:HybridError} we compute the error between the two densities in $L^\infty$ norm:  $||\rho_{Kinetic}-\rho_{Hybrid}||_\infty$ at several time steps. The corresponding state of the cells can be found in Figures \ref{fig:MassVar1} and \ref{fig:MassVar2}. Quite expectedly, there indeed is a slight loss in accuracy as soon as the coupling occurs. However, it quickly diminishes as the coupled solution relaxes to equilibrium. Moreover, one can control this error by tuning the coupling parameters. Again, a balance must be chosen between speed and accuracy.

\begin{figure}
    \centering
    \begin{tabular}{ll}
    \includegraphics[width=.45\linewidth]{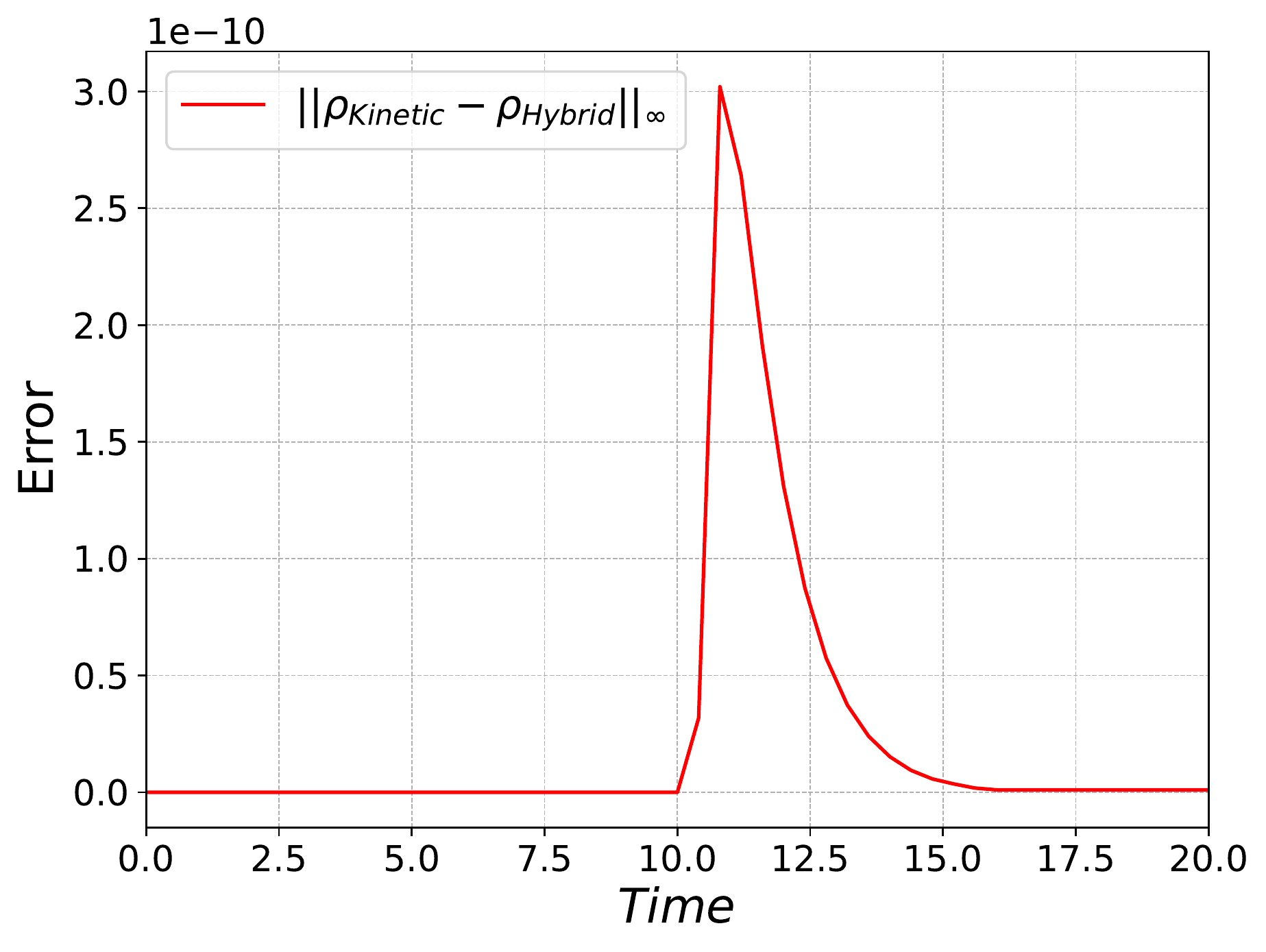}
    &
    \includegraphics[width=.45\linewidth]{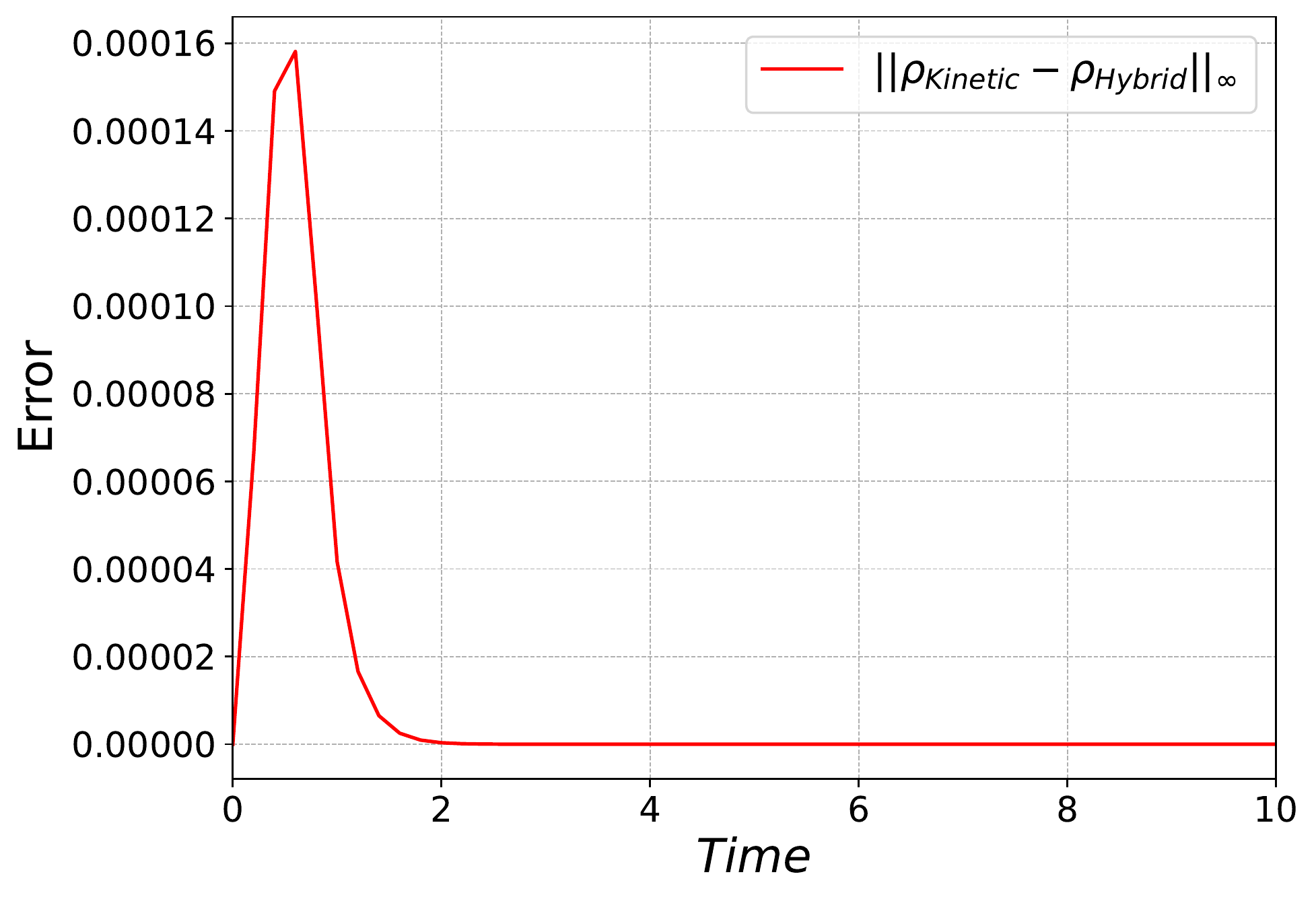}
    \end{tabular}
    \caption{Case 3. Time evolution of the $L^\infty$-error between the kinetic density $\rho_{Kinetic}$ and the hybrid one $\rho_{Hybrid}$, $\varepsilon = 1.0$ (left), $\varepsilon = 0.001$ (right), $\eta_0=\delta_0=10^{-4}$.}
    \label{fig:HybridError}
\end{figure}

\paragraph{Long time behaviour.}
Similarly as for the full kinetic scheme, we are interested in the long-time behaviour of the hybrid scheme. As our scheme is not well-balanced, we shall focus on the case $E=0$. Figure \ref{fig:HypoE=0Hyb} shows the convergence of the norms \eqref{Norms} in the hybrid setting. As the perturbation is not updated when a cell stays fluid and the density is very close to equilibrium, the norms of $g$ and $f$ stagnate when all cells have switched to fluid. However, one can still observe the convergence of the density $\Tilde{\rho}$ towards the global mass.
\begin{figure}
    \centering
    \begin{tabular}{ll}
    \includegraphics[width=0.45\linewidth]{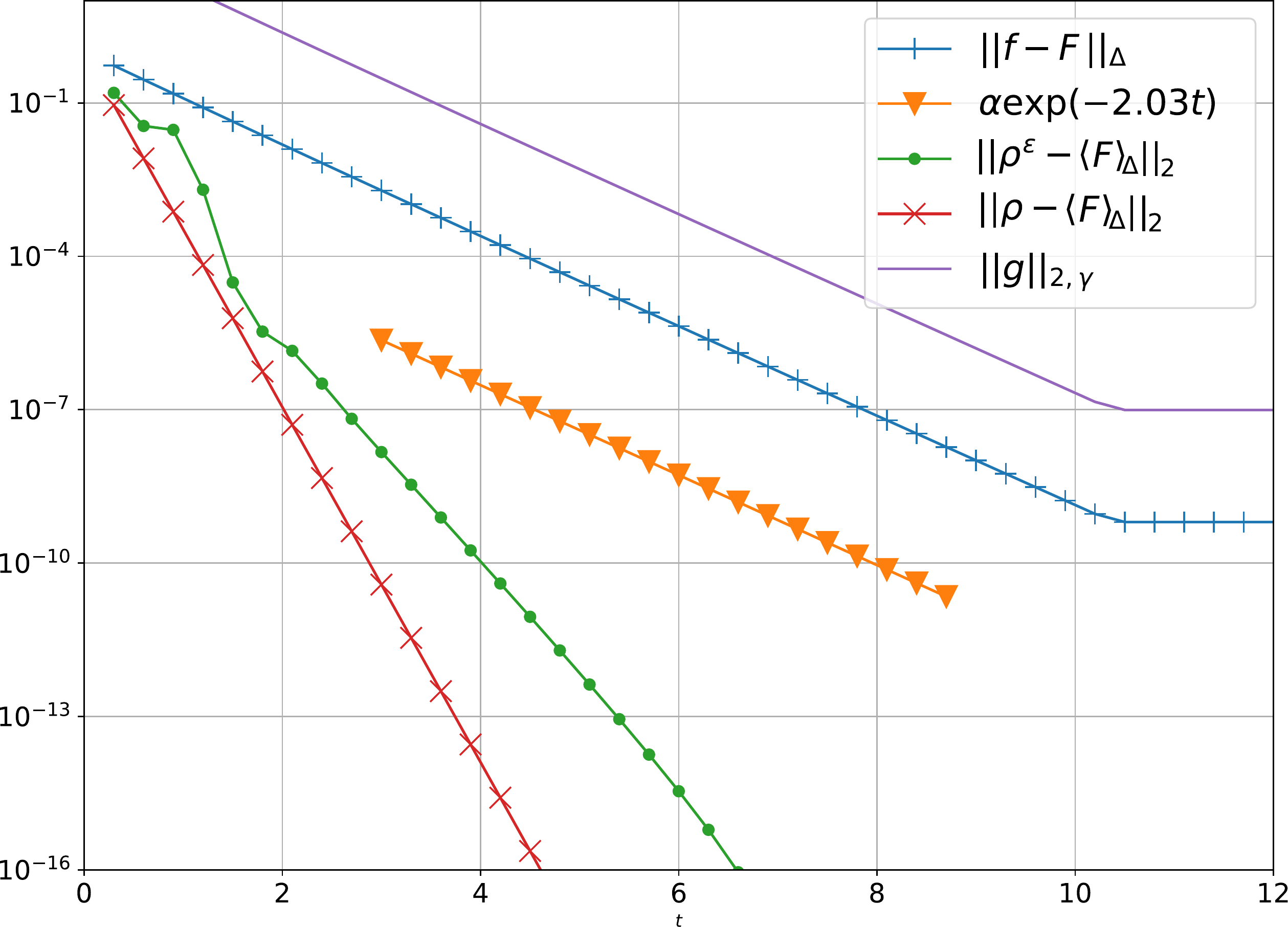}
    &
    \includegraphics[width=0.45\linewidth]{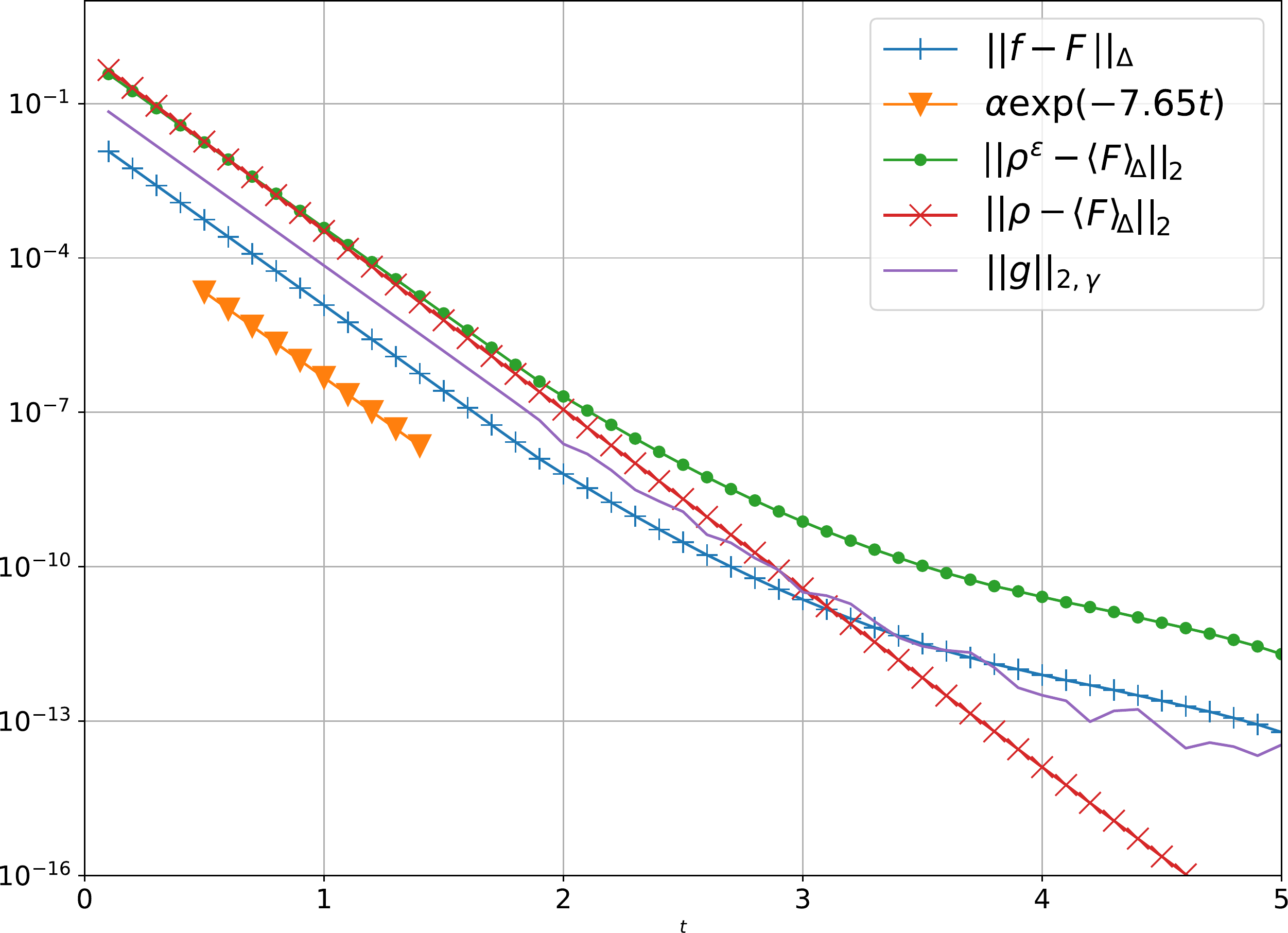}
    \end{tabular}
    \caption{Case 3. Time evolution of the norms \eqref{Norms} computed with the hybrid and limit schemes, $\varepsilon = 1.0$ (left), $\varepsilon = 0.1$ (right), $\eta_0=\delta_0=10^{-4}$.}
    \label{fig:HypoE=0Hyb}
\end{figure}

\paragraph{Computation time.}
Let us now consider the efficiency of the hybrid method. We set $N_x~=~200$ and the final time $T=20.0$ to compare the computation time. Tables \ref{tab:Perf1.0}-\ref{tab:Perf0.1}-\ref{tab:Perf10-3}-\ref{tab:Perf10-6} show the computation time of the full kinetic, hybrid and limit scheme for different test cases with two sets of coupling parameters: $\eta_0=\delta_0=10^{-4}$ and $\eta_0=\delta_0=10^{-3}$. Let us stress out that the computation time is linked to the choice of the coupling parameters. Also note that the same time step, $\Delta t=7.4\times10^{-5}$, is used for the three schemes. 

\begin{table}
    \centering
    \begin{tabular}{|c||c|c|c|c||c|c|c|c|}
        \hline
        & \multicolumn{4}{c||}{$\eta_0=\delta_0=10^{-4}$} & \multicolumn{4}{c|}{$\eta_0=\delta_0=10^{-3}$}\\
        \hline
        Case & 1 & 2 & 3 & 4 & 1 & 2 & 3 & 4\\
        \hline
        MM & 106.3 & 108.5 & 108.7 & 107.5 & 106.3 & 108.5 & 108.7 & 107.5 \\
        \hline
        Hybrid & 102.7 & 107.9 & 62.4 & 106.7 & 101.1 & 107.1 & 50.2 & 107.2\\
        \hline
        Limit & 0.07 & 0.07 & 0.07 & 0.07 & 0.07 & 0.07 & 0.07 & 0.07 \\
        \hline
    \end{tabular}
    \caption{Test 1. Comparison of the computation time (sec), $\varepsilon=1.0$, $T=20.0$, $N_x=200$, $N_v=256$.}
    \label{tab:Perf1.0}
\end{table}

\begin{table}
    \centering
    \begin{tabular}{|c||c|c|c|c||c|c|c|c|}
        \hline
        & \multicolumn{4}{c||}{$\eta_0=\delta_0=10^{-4}$} & \multicolumn{4}{c|}{$\eta_0=\delta_0=10^{-3}$}\\
        \hline
        Case & 1 & 2 & 3 & 4 & 1 & 2 & 3 & 4\\
        \hline
        MM & 107.7 & 108.1 & 109.7 & 108.9 & 107.7 & 108.1 & 109.7 & 108.9\\
        \hline
        Hybrid & 29.9 & 107.8 & 32.6 & 106.7 & 26.6 & 109.8 & 29.5 & 108.0\\
        \hline
        Limit & 0.07 & 0.07 & 0.07 & 0.07 & 0.07 & 0.07 & 0.07 & 0.07\\
        \hline
    \end{tabular}
    \caption{Test 2. Comparison of the computation time (sec), $\varepsilon=0.1$, $T=20.0$, $N_x=200$, $N_v=256$.}
    \label{tab:Perf0.1}
\end{table}

\begin{table}
    \centering
    \begin{tabular}{|c||c|c|c|c||c|c|c|c|}
        \hline
        & \multicolumn{4}{c||}{$\eta_0=\delta_0=10^{-4}$} & \multicolumn{4}{c|}{$\eta_0=\delta_0=10^{-3}$}\\
        \hline
        Case & 1 & 2 & 3 & 4 & 1 & 2 & 3 & 4\\
        \hline
        MM & 106.8 & 107.8 & 109.4 & 107.3 & 106.8 & 107.8 & 109.4 & 107.3\\
        \hline
        Hybrid & 0.72 & 0.71 & 5.35 & 106.4 & 0.72 & 0.70 & 2.42 & 2.47\\
        \hline
        Limit & 0.07 & 0.07 & 0.07 & 0.07 & 0.07 & 0.07 & 0.07 & 0.07\\
        \hline
    \end{tabular}
    \caption{Test 3. Comparison of the computation time (sec), $\varepsilon=10^{-3}$, $T=20.0$, $N_x=200$, $N_v=256$.}
    \label{tab:Perf10-3}
\end{table}

\begin{table}
    \centering
    \begin{tabular}{|c||c|c|c|c||c|c|c|c|}
        \hline
        & \multicolumn{4}{c||}{$\eta_0=\delta_0=10^{-4}$} & \multicolumn{4}{c|}{$\eta_0=\delta_0=10^{-3}$}\\
        \hline
        Case & 1 & 2 & 3 & 4 & 1 & 2 & 3 & 4\\
        \hline
        MM & 108.8 & 108.6 & 108.1 & 105.7 & 108.8 & 108.6 & 108.1 & 105.7\\
        \hline
        Hybrid & 0.72 & 0.72 & 0.74 & 0.74 & 0.72 & 0.72 & 0.74 & 0.71\\
        \hline
        Limit & 0.07 & 0.07 & 0.07 & 0.07 & 0.07 & 0.07 & 0.07 & 0.07\\
        \hline
    \end{tabular}
    \caption{Test 4. Comparison of the computation time (sec), $\varepsilon=10^{-6}$, $T=20.0$, $N_x=200$, $N_v=256$.}
    \label{tab:Perf10-6}
\end{table}

\noindent We can make several observations. First, the fluid solver is as expected, much faster than the full kinetic one. Moreover it is also always faster than the hybrid method. This can easily be explained by the additional cost of computing the indicators and the added cost of dealing with interfaces between kinetic and fluid. Table \ref{tab:SpeedUp} shows the computational gain for the previous tests. In particular, the hybrid method appears does not offer a significant gain in very low collision regimes. Because of the dynamics of the solution, the coupling occurs very late and it cannot compensate the additional cost of the method. Another observation is that the hybrid method is less efficient when $E\neq0$. This could again be explained by the fact that the scheme is not well-balanced. Nevertheless, the speedups for small values of $\varepsilon$ become significant in both cases and the hybrid method becomes competitive with the fluid solver. A final observation is that the choice of larger coupling parameters indeed speeds up the method.

\begin{table}
    \centering
    \begin{tabular}{|c||c|c|c|c||c|c|c|c|}
        \hline
        & \multicolumn{4}{c||}{$\eta_0=\delta_0=10^{-4}$} & \multicolumn{4}{c|}{$\eta_0=\delta_0=10^{-3}$}\\
        \hline
        \diagbox{$\varepsilon$}{Case} & 1 & 2 & 3 & 4 & 1 & 2 & 3 & 4\\
        \hline
        $1.0$ & 1.03 & 1.00 & 1.74 & 1.01 & 1.05 & 1.01 & 2.16 & 1.00\\
        \hline
        $0.1$ & 3.60 & 1.00 & 3.37 & 1.02 & 4.04 & 0.98 & 3.72 & 1.01\\
        \hline
        $10^{-3}$ & 148.3 & 151.8 & 1.01 & 2.45 & 148.3 & 154.0 & 45.2 & 43.44\\
        \hline
        $10^{-6}$ & 151.1 & 150.8 & 142.8 & 146.1 & 147.0 & 158.0 & 146.1 & 148.9\\
        \hline
    \end{tabular}
    \caption{Speedup of the hybrid method compared to the full kinetic scheme, $N_x=200$, $N_v=256$.}
    \label{tab:SpeedUp}
\end{table}

\paragraph{Non homogeneous Knudsen number.}
In this last numerical experiment, we consider a non homogeneous Knudsen number in the physical domain. Let us define the function 
\begin{equation}
    e(x) = \frac{1}{2}(\arctan(5+10(x-\frac{\pi}{2}))+\arctan(5-10(x-\frac{\pi}{2}))).
\end{equation}
In particular, we choose $\varepsilon=\varepsilon(x)$ as
\begin{equation}
    \varepsilon(x) = \frac{e(x)}{\max(e(x))}.
\end{equation}
Such a function admits a maximum of $1$ in the center of the domain and decays to $0$ near the boundaries. Physically, it corresponds to few collisions in the center of the domain and to a fluid behaviour elsewhere. In the following simulations, $\Delta t=10^{-4}$ and the coupling parameters are $\delta_0=\eta_0=10^{-4}$. Note that depending on the choice of $\varepsilon(x)$ one may need to decrease the time step to ensure stability. From an implementation point of view, the constant $\varepsilon$ is simply replaced by $\varepsilon_\ipd=\varepsilon(x_\ipd)$ without any change to the indicators.

\noindent Figure \ref{fig:SnapDistNH} shows that the hybrid scheme captures well the behaviour of the distribution. Indeed, we observe a fast relaxation where $\varepsilon(x)$ is small and a much slower one in the center of the domain where $\varepsilon(x)$ is around $1$. Regarding the state of the cells, one can see on Figures \ref{fig:HybridCompDensNH} and \ref{fig:MassVarNH} that the fluid solver is quickly used where $\varepsilon(x)$ is small. Moreover, the last cells to become fluid are the one where the gradient of $\varepsilon(x)$ is large. It is explained by the nature of the macroscopic indicator which uses derivatives up to order 4. In this setting, the variation of mass was again of order $10^{-12}$. Finally, we looked at the convergence in time to the global equilibrium. On Figure \ref{fig:HypoNH}, one can again observe the exponential convergence to equilibrium and the rate is slightly higher than the one obtained with an homogeneous value of $\varepsilon=1$. Up to stability considerations, this experiment shows the robustness of the hybrid method. Performance-wise, considering Case 3 with $N_x=200$ and $N_v=256$ the full kinetic scheme takes $183.7$ seconds to run while the hybrid one takes $147.6$ seconds offering a speedup of $1.24$.

\begin{figure}
    \centering
    \includegraphics[width=\linewidth]{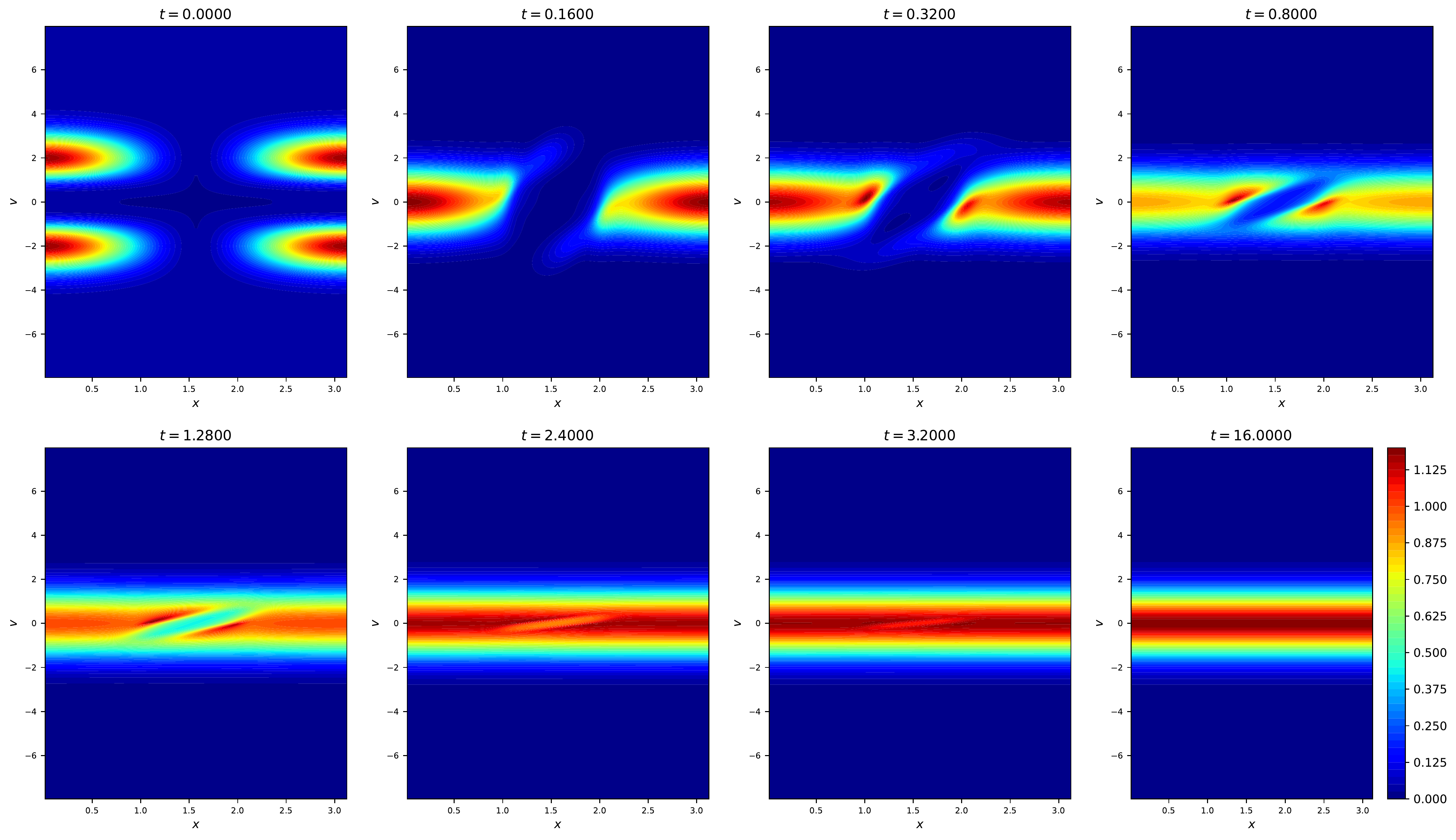}
    \caption{Case 3. Snapshots of distributions obtained with the hybrid scheme for a non homogeneous Knudsen number.}
    \label{fig:SnapDistNH}
\end{figure}
\begin{figure}
    \centering
    \includegraphics[width=\linewidth]{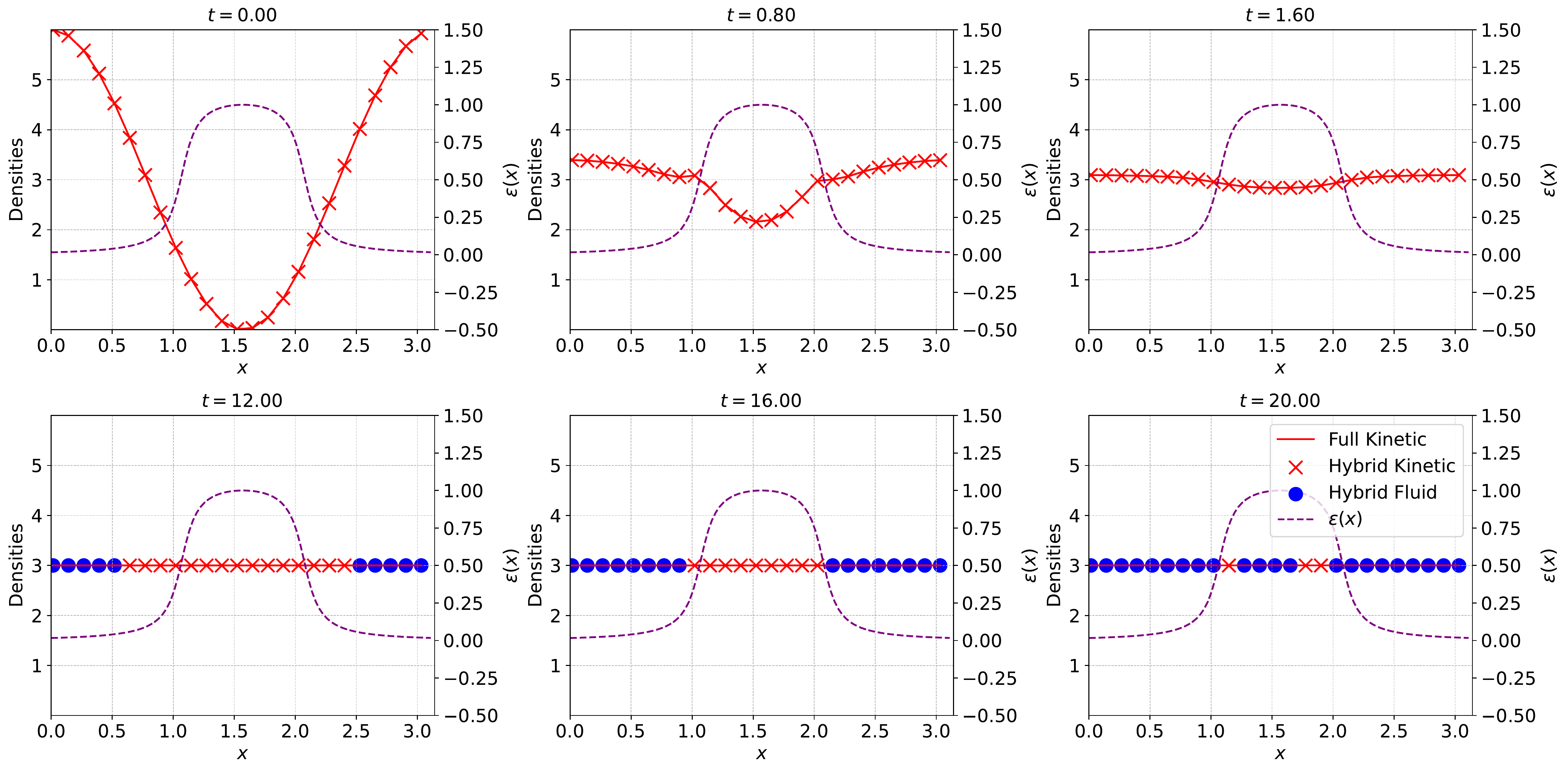}
    \caption{Case 3. Snapshots of the densities computed using the full kinetic and hybrid schemes for a non homogeneous Knudsen number.}
    \label{fig:HybridCompDensNH}
\end{figure}

\begin{figure}
    \centering
    \includegraphics[width=\linewidth]{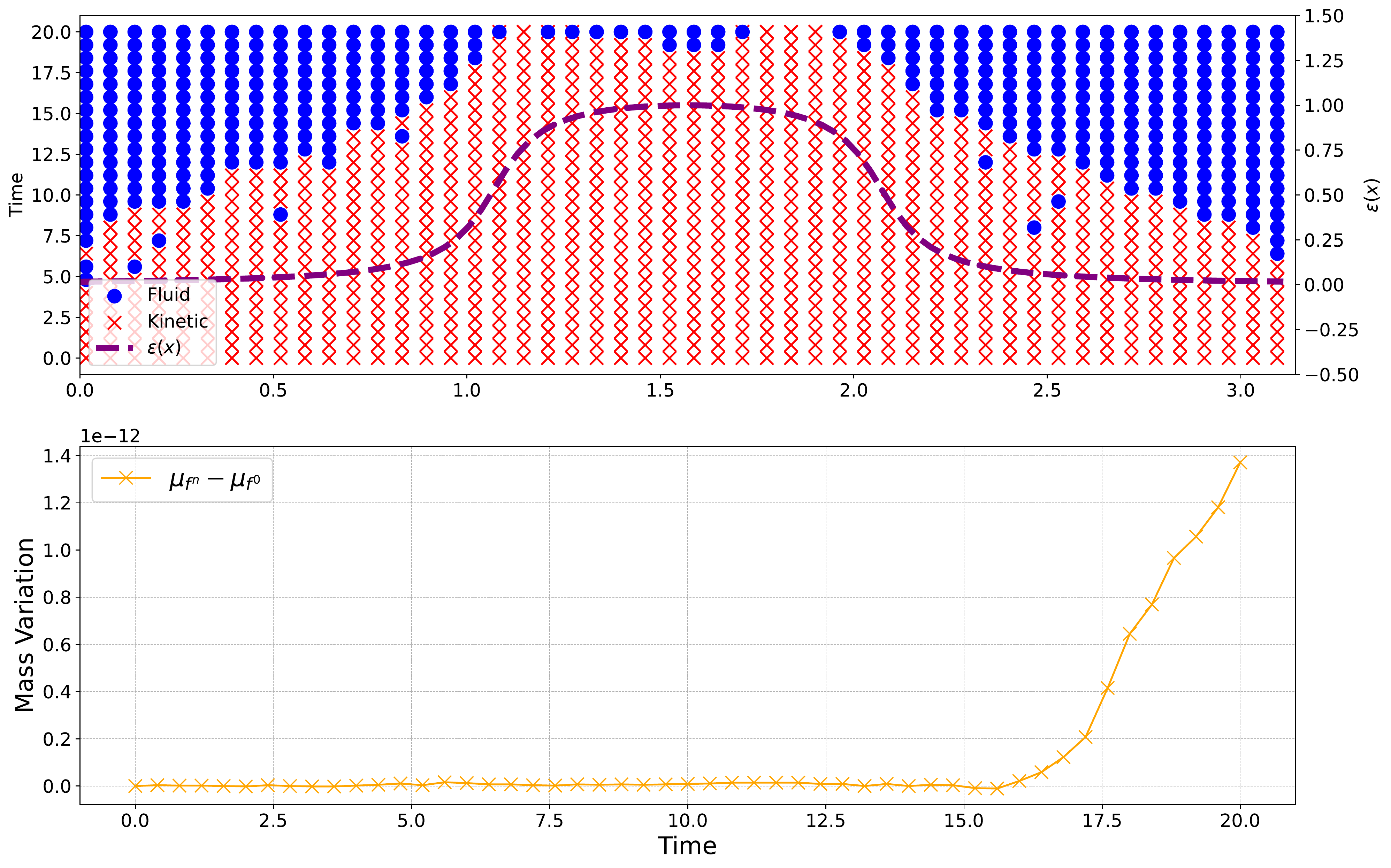}
    \caption{Case 3. Time evolution of the state of the cells (Top) and mass variation (Bottom) for a non homogeneous Knudsen number.}
    \label{fig:MassVarNH}
\end{figure}

\begin{figure}
    \centering
    \begin{tabular}{ll}
        \includegraphics[width=0.45\linewidth]{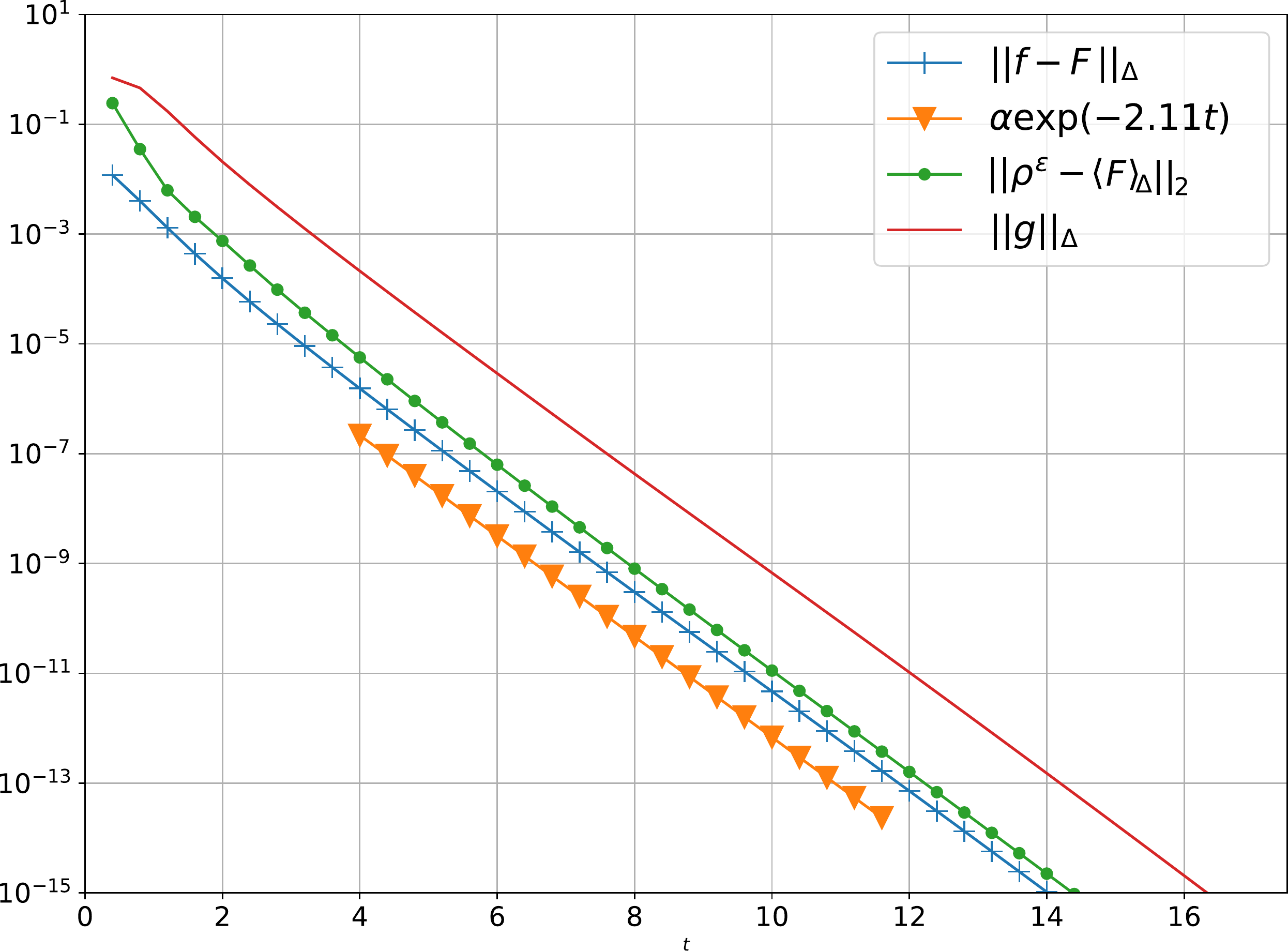} 
        &
        \includegraphics[width=0.45\linewidth]{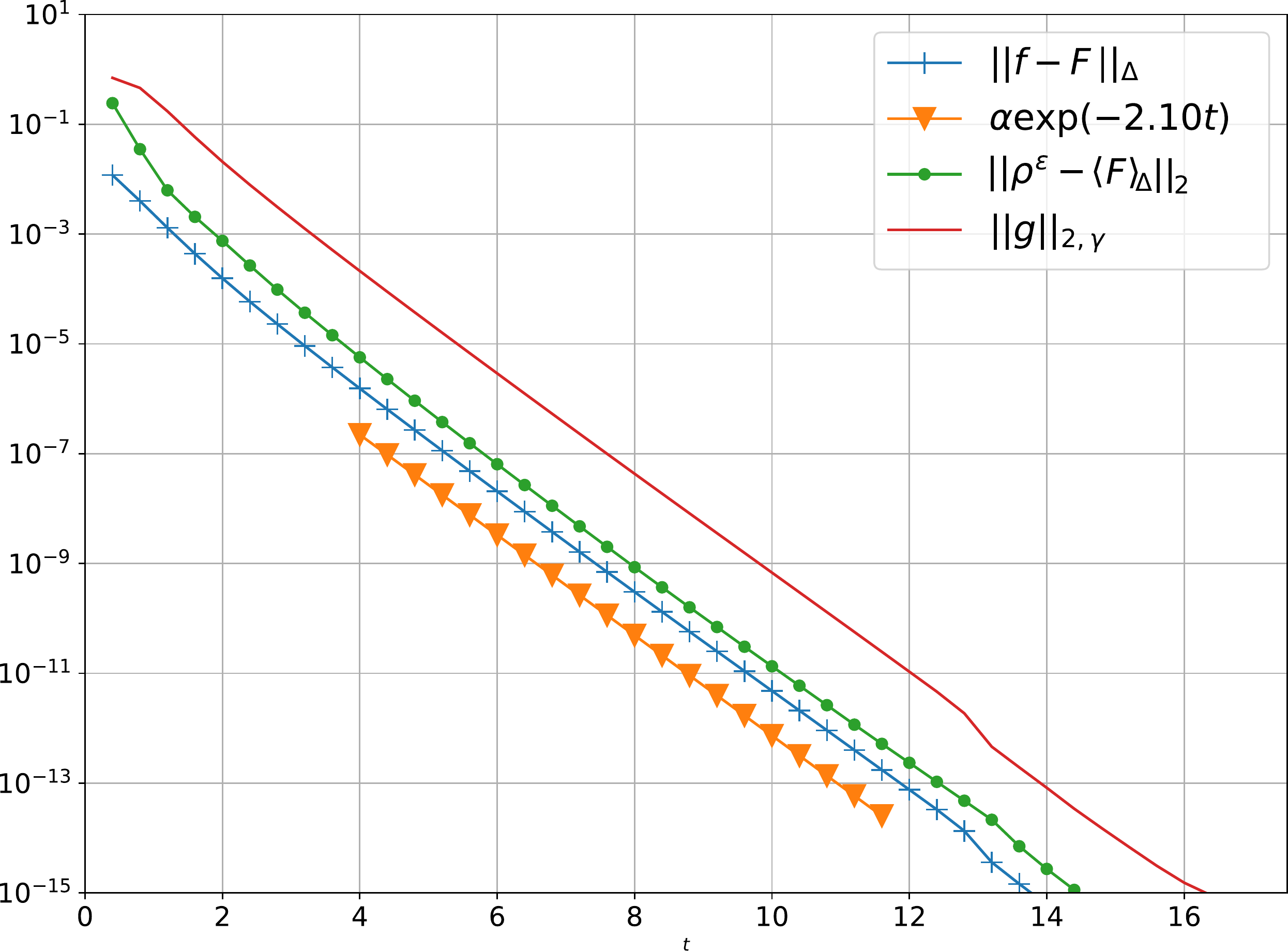}
    \end{tabular}
    \caption{Case 3. Convergence to a global equilibrium, full kinetic(Left), hybrid (Right) for a non homogeneous Knudsen number.}
    \label{fig:HypoNH}
\end{figure}
\section{Conclusion}
In this work, a new hybrid numerical method for linear kinetic equations in the diffusive scaling was presented. The method relies on two criteria motivated by a pertubative approach. The first one quantifies how far from a local equilibrium the distribution function is. The second criterion depends on the macroscopic quantities that are available on the whole computing domain. We have managed to quantify the mass variation induced by the method and we have shown that it is in practice very small. The method has proven to be efficient through various numerical experiments: the computational gain compared to a full kinetic scheme is significant. Moreover, the method performs well with a non-homogeneous Knudsen number in position which is encouraging to tackle more physically motivated problems.

In future works, a more general and physically relevant setting will be considered. In particular, it involves the multidimensional setting, a more general collision operator and a coupling with the Poisson equation. We are confident that the computational gain will be even more worthwhile in a full $3D-3D$ setting in the case of the Boltzmann operator which is known to be costly numerically. 

\section*{Acknowledgements}
    T.L. was partially funded by Labex CEMPI (ANR-11-LABX-0007-01)
    and ANR Project MoHyCon (ANR-17-CE40-0027-01). He also would like to thank Thomas Rey and Marianne Bessemoulin-Chattard for the fruitful discussions and insights.

\bibliography{HybridMicroMacro2021.bib}{}
\bibliographystyle{acm}

\end{document}